\documentclass[11pt]{article}

\usepackage[T1]{fontenc}
\usepackage{lmodern}
\usepackage[utf8]{inputenc}
\usepackage{microtype}
\usepackage{framed}
\usepackage{listings}
\usepackage{vmargin}
\usepackage{setspace}
\usepackage{mathrsfs, mathenv}
\usepackage{amsmath, amsthm, amssymb, amsfonts, amscd}
\usepackage{graphicx}
\usepackage{epstopdf}
\usepackage[svgnames]{xcolor}
\usepackage{hyperref}
\hypersetup{citecolor=blue, colorlinks=true, linkcolor=black}
\usepackage[ruled, vlined]{algorithm2e}
\usepackage[capitalise]{cleveref}
\usepackage{subcaption}
\usepackage{bbm}
\usepackage{cite}
\usepackage{verbatim}
\usepackage{pgfplots}
\usepackage{tikz}
\usetikzlibrary{arrows,decorations.pathmorphing,backgrounds,positioning,fit,matrix}
\usepackage{etoolbox}
\usepackage{color}
\usepackage{lipsum}
\usepackage{ifthen}
\usepackage[title]{appendix}
\usepackage{autonum}
\usepackage{accents}
\usepackage{xpatch}
\usepackage[]{algpseudocode}
\usepackage{multicol}
\usepackage[abs]{overpic}
\usepackage{eqnarray}
\usepackage{bm}

\crefname{assumption}{Assumption}{Assumptions}
\crefname{figure}{Figure}{Figures}

\theoremstyle{plain}
\newtheorem{theorem}{Theorem}[section]

\newtheorem{proposition}[theorem]{Proposition}
\numberwithin{equation}{section}

\theoremstyle{definition}

\theoremstyle{remark}
\newtheorem{remark}[theorem]{Remark}

\ifpdf
  \DeclareGraphicsExtensions{.eps,.pdf,.png,.jpg}
\else
  \DeclareGraphicsExtensions{.eps}
\fi

\usepackage{mathtools}

\usepackage{booktabs}

\usepackage{pgfplots}
\usepackage{tikz}
\usetikzlibrary{patterns,arrows,decorations.pathmorphing,backgrounds,positioning,fit,matrix}
\usepackage[labelfont=bf]{caption}
\usepackage{here}
\usepackage[font=normal]{subcaption}

\usepackage{enumitem}
\setlist[itemize]{leftmargin=.5in}
\setlist[enumerate]{leftmargin=.5in,topsep=3pt,itemsep=3pt,label=(\roman*)}

\newcommand*\samethanks[1][\value{footnote}]{\footnotemark[#1]}

\newcommand{\TheTitle}{Neural active manifolds: nonlinear dimensionality reduction for uncertainty quantification}
\newcommand{\TheAuthors}{A. Zanoni, M. Salvador, G. Geraci, A. L. Marsden, D. E. Schiavazzi}
\title{\TheTitle}
\author{
Andrea Zanoni \thanks{Centro di Ricerca Matematica Ennio De Giorgi, Scuola Normale Superiore, Pisa, Italy.} \thanks{Institute for Computational and Mathematical Engineering, Stanford University, Stanford, CA, USA.}
\and Gianluca Geraci \thanks{Center for Computing Research, Sandia National Laboratories, Albuquerque, NM, USA.}
\and Matteo Salvador \samethanks[2] \thanks{Pasteur Labs, Brooklyn, NY, USA}
\and Alison L. Marsden \samethanks[2] \thanks{Pediatric Cardiology, Stanford University, Stanford, CA, USA.} \thanks{Bioengineering, Stanford University, Stanford, CA, USA.}
\and Daniele E. Schiavazzi \thanks{Department of Applied and Computational Mathematics and Statistics, University of Notre Dame, Notre Dame, IN, USA.}
}
\date{}

\newcommand{\Iion}{{\mathcal{I}_{\mathrm{ion}}}}
\newcommand{\Iapp}{{\mathcal{I}_{\mathrm{app}}}}
\newcommand{\DiffTens}{D_{\mathrm{M}}}
\newcommand{\Pot}{u}
\newcommand{\Ionic}{y}
\newcommand{\Di}{D_\mathrm{iso}}
\newcommand{\tboxstim}{\tau_\mathrm{box}^\mathrm{stim}}
\newcommand{\RhsIonic}{F}
\newcommand{\GNa}{G_\mathrm{Na}}
\newcommand{\GCaL}{G_\mathrm{CaL}}
\newcommand{\GKr}{G_\mathrm{Kr}}
\newcommand{\GKs}{G_\mathrm{Ks}}

\newcommand{\Q}{\mathcal Q}
\newcommand{\E}{\mathcal E}
\newcommand{\D}{\mathcal D}
\newcommand{\F}{\mathcal F}
\renewcommand{\S}{\mathcal S}

\newcommand{\LF}{\mathrm{LF}}
\newcommand{\HF}{\mathrm{HF}}

\usepackage{amsopn}

\newcommand{\abs}[1]{\left\lvert#1\right\rvert}
\newcommand{\norm}[1]{\left\|#1\right\|}
\renewcommand{\Pr}{\mathbb{P}}

\newcommand{\R}{\mathbb{R}}

\newcommand{\Var}{\operatorname{\mathbb{V}ar}}
\newcommand{\Cov}{\operatorname{\mathbb{C}ov}}
\newcommand{\Ex}{\operatorname{\mathbb{E}}}

\newcommand{\dd}{\,\mathrm{d}}
\definecolor{shade}{RGB}{100, 100, 100}
\definecolor{bordeaux}{RGB}{128, 0, 50}

\usepackage[usestackEOL]{stackengine}

\definecolor{leg1}{RGB}{0,114,189}
\definecolor{leg2}{RGB}{217,83,25}
\definecolor{leg3}{RGB}{237,177,32}
\definecolor{leg4}{RGB}{126,47,142}
\definecolor{leg5}{RGB}{119,172,48}

\definecolor{leg21}{RGB}{62,38,169}
\definecolor{leg22}{RGB}{46,135,247}
\definecolor{leg23}{RGB}{55,200,151}
\definecolor{leg24}{RGB}{254,195,56}

\ifpdf
\hypersetup{
	pdftitle={\TheTitle},
	pdfauthor={\TheAuthors}
}
\fi

\begin{document}
	
\maketitle	

\begin{abstract} 
\noindent We present a new approach for nonlinear dimensionality reduction, specifically designed for computationally expensive mathematical models. We leverage autoencoders to discover a one-dimensional \emph{neural active manifold} (NeurAM) capturing the model output variability, through the aid of a simultaneously learnt surrogate model with inputs on this manifold. Our method only relies on model evaluations and does not require the knowledge of gradients. The proposed dimensionality reduction framework can then be applied to assist outer loop many-query tasks in scientific computing, like sensitivity analysis and multifidelity uncertainty propagation. In particular, we prove, both theoretically under idealized conditions, and numerically in challenging test cases, how NeurAM can be used to obtain multifidelity sampling estimators with reduced variance by sampling the models on the discovered low-dimensional and shared manifold among models. Several numerical examples illustrate the main features of the proposed dimensionality reduction strategy and highlight its advantages with respect to existing approaches in the literature.
\end{abstract}



\section{Introduction}

Mathematical modeling has become an indispensable tool to understand real-world phenomena in physics, engineering, and social sciences~\cite{Pal22,LMN20}. Mathematical models depend on parameters that might be unknown and/or uncertain. As the computational models become more complex, quantification of predictive uncertainty plays an increasingly important role to fully assess the validity and accuracy of the results they provide.
However, standard techniques for uncertainty quantification (UQ), such as Monte Carlo methods, require a large number of samples to reliably estimate the statistical moments (e.g., mean and variance) of the stochastic response, and reaching high precision can be unfeasible if the model is computationally expensive.
Moreover, if the number of input parameters of the model increases, most methods suffer from the curse of dimensionality, meaning that the complexity of the estimation task grows exponentially with the dimensionality, making the uncertainty quantification problem intractable~\cite{GBC16}. In uncertainty quantification, as in any other many-query workflows such as optimization, global sensitivity analysis, and design, dimensionality reduction techniques become crucial.
Dimensionality reduction offers a powerful approach to alleviate the computational cost of workflows that need to query the model numerous times, e.g., uncertainty quantification, by reducing the number of model parameters while preserving the main output features. 
Moreover, some mathematical models can be overparameterized, and reducing the dimensionality can help to capture the most relevant information, while disregarding redundant variables or unimportant combinations thereof.

Overviews of dimensionality reduction strategies can be found in, e.g.,~\cite{Bur10,KLG22}. 
We mention, in particular, two recent approaches related to sliced inverse regression~\cite{DuL91,Li_91} that motivated our work. 
Active subspaces (AS) was first proposed to determine linear low-dimensional subspaces of maximum model variance in~\cite{Rus10}, and then thoroughly analyzed in~\cite{CDW14,Con15}. 
The \emph{active} directions of a model consist of the eigenvectors corresponding to the largest eigenvalues from the expected gradient covariance matrix.
AS have been demonstrated in numerous contexts and applications including Bayesian inverse problems~\cite{CKB16}, sensitivity analysis~\cite{CoD17}, and multifidelity dimensionality reduction~\cite{LZM20}. 
The main limitation of AS, inherited from its linear character, is that it only captures the model variation on average, and, therefore, the existence of a lower dimensional manifold cannot be guaranteed~\cite{BGF19}.
A second approach called active manifolds (AM)~\cite{BGF19} overcomes some of the limitations of AS. This approach is based on a local construction of level sets which approximates any continuously differentiable multi-dimensional function using a surrogate defined over a one-dimensional manifold, obtained through a nonlinear transformation.
The main advantages of AM over AS are that they always provide a one-dimensional reduced manifold, give a more accurate surrogate, and enable more informative sensitivity analysis.
Unfortunately, as discussed in~\cite{BGF19}, the geometric nature of the AM algorithm requires that the function's level sets are fully contained in the original domain of the function. Indeed, the hypothesis of AM is that any arbitrary multi-dimensional $\mathcal C^1$ function can be related to a one-dimensional $\mathcal C^1$ counterpart where each point of the function in the original space is mapped to a point on the one-dimensional AM moving tangentially to the corresponding level set. As a consequence, if a point lies on a level set that leaves the domain, it is impossible to map it to the AM. 
Additionally, both AS and AM methods require the computation of the model output gradient, which is often unavailable for complex scientific codes, or computationally expensive to approximate for high-dimensional problems. We note that alternative approaches like the so-called Adaptive Basis~\cite{TiG14,ZRG21} are able to overcome the need for gradients by relying on the construction of a first order polynomial chaos surrogate~\cite{LeK10}. However, since this approach is still limited to the identification of a linear manifold, in the rest of the paper we only consider AS as representative of both methods. 

In this work we propose a novel methodology inspired by AM, that does not rely on the gradient of the model, and discovers a transformation from the original input space to a one-dimensional nonlinear manifold without being restricted by level sets.
We call this methodology \emph{neural active manifolds} (NeurAM). 
Our approach is based on autoencoders~\cite{Bal12}, a well established approach for unsupervised dimensionality reduction with neural networks. 
The architecture consists of an encoder, which maps the input into a lower-dimensional latent space, and a decoder, which maps points from the latent space back to the original input space.
Classic autoencoders for usupervised tasks have been used in~\cite{BoF21} to determined active manifolds that are then employed to accelerate multidisciplinary analysis and optimization. This approach relies on a snapshot matrix built using an auxiliary cheaper model.
However, unlike unsupervised learning, where autoencoders are used to minimize the reconstruction error, we are interested in determining a one-dimensional manifold in the space of parameters capturing the entire variability of the model output.
To do so, we design a computational pipeline where a surrogate model with inputs on a one-dimensional latent space is leveraged by the autoencoder to identify the manifold on which the representation of the function can be obtained with minimal error. This surrogate, which is parameterized as a dense neural network, is trained at the same time as the autoencoder's parameters, defining the encoder and decoder, and in principle does not need to be strongly accurate because it is only useful to help in the construction of the manifold.

The resulting one-dimensional NeurAM can then be employed to support for many-query tasks, like sensitivity analysis and uncertainty propagation. 
In particular, we first show how to perform sensitivity analysis along the manifold and derive both local and global indices for the input parameters with respect to a scalar output. 
Our indices, in contrast to more classical indices, such as Sobol'~\cite{Sob90} and Borgonovo~\cite{PBS13}, are defined along a one-dimensional manifold.

Then, we focus on the problem of uncertainty propagation, which can easily become unfeasible for computationally expensive models. 
In order to mitigate this issue, multifidelity estimators~\cite{NgW14,PWG16,GGE20,BLW22} rely on cheaper-to-evaluate, but correlated, low-fidelity models to achieve variance reduction. The statistical correlation among models is directly related to the attainable variance reduction, which corresponds to a greater estimator precision when compared to a single fidelity Monte Carlo estimator with the same computational cost.
Therefore, for poorly correlated model pairs such as those characterized by a dissimilar number of input parameters, methodologies to increase the correlation between the fidelities have been recently proposed, and are based on approaches for linear dimensionality reduction such as AS~\cite{GEG18} and adaptive basis~\cite{ZGE23}.
A first step towards the extension to nonlinear dimensionality reduction using autoencoders is discussed in~\cite{ZGS24,ZGS24b}. Following this line of work, we leverage NeurAM to enhance the performance of sampling multifidelity estimators, complementing numerical test cases showing variance reduction of the resulting estimators with a theoretical analysis under simplified assumptions. 
In particular, we demonstrate that, under such assumptions, NeurAM provides new sampling locations, in the original models' inputs, that correspond to shared samples in the low-dimensional manifold leading to a bi-fidelity correlation that is never smaller than the one produced by the original samples. Importantly, we note that the approach proposed here could be directly leveraged in the context of the multifidelity construction of surrogates, e.g., following recent work discussed in~\cite{ZGG23,ZGG25}. For simplicity of exposure, in this manuscript, we limit the use of NeurAM to the case of sampling estimators and we leave its integration with surrogate techniques to a future study. 

The main contributions of this work include:
\begin{itemize}[leftmargin=*,itemsep=0pt,topsep=0pt]
\item introducing an algorithm to learn a one-dimensional active manifold aided by a simultaneously trained surrogate model;
\item proposing a way to leverage NeurAM for sensitivity analysis; 
\item showing how to use nonlinear dimensionality reduction to improve the correlation between low- and high-fidelity models, and consequently to obtain multifidelity sampling estimators of reduced variance;
\item providing a theoretical result under specific assumptions, showing that the correlation obtained from NeurAM is never smaller than the initial correlation.
\end{itemize}

The rest of the paper is organized as follows. In \cref{sec:method} we introduce NeurAM, and present possible applications of our method in the field of uncertainty quantification in~\cref{sec:applications}. 
We then discuss sensitivity analysis in \cref{sec:sensitivity_analysis} and multifidelity uncertainty propagation in \cref{sec:uncertainty_propagation}. 
Then, in \cref{sec:examples} we demonstrate the proposed approach on  several numerical examples of increasing complexity.
Finally, conclusions and possible future research directions are presented in \cref{sec:conclusion}.

\section{Methodology}\label{sec:method}

In this section we present our methodology to reduce the dimensionality of a computationally expensive model $\Q \colon \R^d \to \R$ with $d>1$, which is the primary objective of this work. Let us assume that we are given a set of $N$ realizations $\{(x_n, \Q(x_n))\}_{n=1}^N$, where the samples $\{ x_n \}_{n=1}^N$ are drawn from an input distribution $\mu$. 
Using the notation of the autoencoders, consider an encoder $\E \colon \R^d \to \R$ to reduce the dimensionality, and a decoder $\D \colon \R \to \R^d$ to re-map a given sample to the original input space. 
A fundamental property of NeurAM is that it aims to represent the entire variability of the model $\Q$, therefore the encoder and the decoder need to satisfy $\Q(x) \simeq \Q(\D(\E(x)))$; see also~\cite{ZGS24}. 
This implies that $\E$ and $\D$ can be obtained by minimizing the loss function
\begin{equation} \label{eq:loss_AE_infeasible}
\ell (\E,\D) = \Ex^\mu \left[ (\Q(X) - \Q(\D(\E(X))))^2 \right],
\end{equation}
where the superscript denotes the fact that the expectation is computed with respect to the measure $\mu$, i.e., $X \sim \mu$. 
\begin{remark}
Even if we employ the same structure of an autoencoder, our method does not behave like an autoencoder, in the sense that we do not aim to reconstruct the original input. In particular, notice that, in the loss function \eqref{eq:loss_AE_infeasible}, we enforce the fact that the evaluation of the model in the reconstructed input is close to the evaluation of the model in the original input, i.e., $\Q(x) \simeq \Q(\D(\E(x)))$, and not necessarily that the reconstructed point is close to the original point, i.e., $x \not\simeq \D(\E(x))$ in general.
\end{remark}
In order to solve the minimization problem, the encoder and the decoder need to be parameterized as $\E(\cdot; \alpha)$ and $\D(\cdot; \beta)$, and the expectation is approximated by its Monte Carlo approximation, yielding the discrete optimization problem
\begin{equation}\label{eq:loss_AE_optim}
\widetilde \ell(\alpha, \beta) = \frac1N \sum_{n=1}^N  (\Q(x_n) - \Q(\D(\E(x_n; \alpha); \beta)))^2.
\end{equation}
In practice, solving \eqref{eq:loss_AE_optim} without incurring additional model evaluations is not feasible. 
In fact, any optimization algorithm, e.g., (stochastic) gradient descent, would require a \emph{new} evaluation of the model $\Q$ at the points $\{ \D(\E(x_n; \alpha); \beta) \}_{n=1}^N$, which are not in the original dataset $\{ x_n \}_{n=1}^N$, at each iteration.

To overcome this limitation, we propose to train a surrogate model $\S \colon \R \to \R$ on the latent space, and learn it simultaneously with the nonlinear transformation from the original points in $\mathbb{R}^d$ to the corresponding points in $\mathbb{R}$ on the neural active manifold. 
We generalize the loss function in equation \eqref{eq:loss_AE_infeasible} to account for the surrogate contribution 
\begin{equation} \label{eq:loss}
\begin{aligned}
\mathcal L(\E,\D,\S) &= \Ex^\mu \left[ (\Q(X) - \S(\E(\D(\E(X)))))^2 \right] \\
&\quad+ \Ex^\mu \left[ (\Q(X) - \S(\E(X)))^2 \right] \\
&\quad+ \Ex^\mu \left[ (\D(\E(X)) - \D(\E(\D(\E(X)))))^2 \right], \\
\end{aligned}
\end{equation}
where the first term corresponds to the one in the original loss $\ell$ with the surrogate $\S$ defined on the latent space replacing the full model, the second term refers to the error of the surrogate model, and the third term enforces the fact that a point on the NeurAM ($\widetilde x$) must be mapped in the point itself ($\widetilde{\widetilde x}$), i.e., the second part of the network in \cref{fig:network}, where a schematic representation of the approach is provided, acts as a classical autoencoder. 
In particular, the three terms in the loss \eqref{eq:loss} are all fundamental since they enforce different properties to the network, and we emphasize again that we do not aim to reconstruct all the original inputs, i.e., $\D(\E(x)) \not\simeq x$ in general. However, if we consider only the points in the reduced manifold, then the couple encoder-decoder must behave as a standard autoencoder.

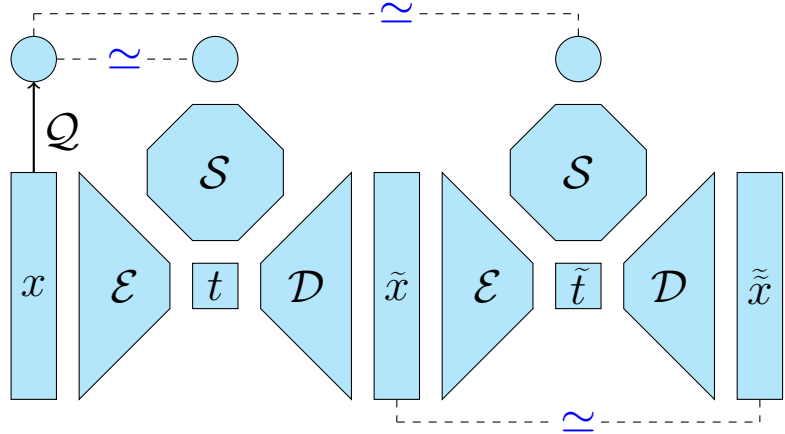
\begin{figure}[ht!]
\centering
\begin{tikzpicture}[scale=.6]
\draw[fill=cyan, fill opacity=.3] (0,0) rectangle (1,5);
\draw[fill=cyan, fill opacity=.3] (1.5,0) -- (3.5,2) -- (3.5,3) -- (1.5,5) -- cycle;
\draw[fill=cyan, fill opacity=.3] (4,2) rectangle (5,3);
\draw[fill=cyan, fill opacity=.3] (7.5,0) -- (5.5,2) -- (5.5,3) -- (7.5,5) -- cycle;
\draw[fill=cyan, fill opacity=.3] (8,0) rectangle (9,5);
\draw[fill=cyan, fill opacity=.3] (0.5,7.5) circle (0.5);
\draw[fill=cyan, fill opacity=.3] (4.5,7.5) circle (0.5);
\draw [->, thick] (0.5,5) -- (0.5,7);
\draw[fill=cyan, fill opacity=.3] (4,3.5) -- (5,3.5) -- (6,4.5) -- (6,5.5) -- (5,6.5) -- (4,6.5) -- (3,5.5) -- (3,4.5) -- cycle;
\node at (0.5,2.5) {\LARGE $x$};
\node at (4.5,2.5) {\LARGE $t$};
\node at (8.5,2.5) {\LARGE $\widetilde x$};
\node at (1.1,5.9) {\LARGE $\mathcal Q$};
\node at (2.5,2.5) {\LARGE $\mathcal E$};
\node at (6.5,2.5) {\LARGE $\mathcal D$};
\node at (4.5,5) {\LARGE $\mathcal S$};
\draw[fill=cyan, fill opacity=.3] (9.5,0) -- (11.5,2) -- (11.5,3) -- (9.5,5) -- cycle;
\draw[fill=cyan, fill opacity=.3] (12,2) rectangle (13,3);
\draw[fill=cyan, fill opacity=.3] (15.5,0) -- (13.5,2) -- (13.5,3) -- (15.5,5) -- cycle;
\draw[fill=cyan, fill opacity=.3] (16,0) rectangle (17,5);
\draw[fill=cyan, fill opacity=.3] (12.5,7.5) circle (0.5);
\draw[fill=cyan, fill opacity=.3] (12,3.5) -- (13,3.5) -- (14,4.5) -- (14,5.5) -- (13,6.5) -- (12,6.5) -- (11,5.5) -- (11,4.5) -- cycle;
\node at (10.5,2.5) {\LARGE $\mathcal E$};
\node at (14.5,2.5) {\LARGE $\mathcal D$};
\node at (12.5,5) {\LARGE $\mathcal S$};
\node at (12.5,2.5) {\LARGE $\widetilde t$};
\node at (16.5,2.5) {\LARGE $\widetilde{\widetilde x}$};
\draw [dashed] (1,7.5) -- (2,7.5);
\draw [dashed] (3,7.5) -- (4,7.5);
\node at (2.5,7.5) {\LARGE \textcolor{blue}{$\simeq$}};
\draw [dashed] (0.5,8) -- (0.5,8.5);
\draw [dashed] (12.5,8) -- (12.5,8.5);
\draw [dashed] (0.5,8.5) -- (8,8.5);
\draw [dashed] (9,8.5) -- (12.5,8.5);
\node at (8.5,8.5) {\LARGE \textcolor{blue}{$\simeq$}};
\draw [dashed] (8.5,0) -- (8.5,-0.5);
\draw [dashed] (16.5,0) -- (16.5,-0.5);
\draw [dashed] (8.5,-0.5) -- (12,-0.5);
\draw [dashed] (13,-0.5) -- (16.5,-0.5);
\node at (12.5,-0.5) {\LARGE \textcolor{blue}{$\simeq$}};
\end{tikzpicture}
\caption{Schematic representation of the NeurAM architecture, where the symbols $\simeq$ are used to indicate the three terms in the loss function in equation \eqref{eq:loss}.}
\label{fig:network}
\end{figure}

We notice that the NeurAM is parameterized by the decoder $\D(t)$ for all $t$ in the interval
\begin{equation} \label{eq:domain_manifold}
\mathcal T = \left[ \min_{x\in\R^d} \E(x), \max_{x\in\R^d} \E(x) \right],
\end{equation} 
and the one-dimensional analogue of the model $\Q$ on the manifold is given by $\S$. Moreover, this methodology provides a surrogate model $\Q_{\mathrm S}$ for $\Q$ defined as
\begin{equation} \label{eq:surrogateQ}
\Q_{\mathrm S}(x) = \S(\E(x)).
\end{equation}
Note that $\Q_{\mathrm S}$ is introduced primarily to assist in training the autoencoder by effectively providing a supervised loss term. Thus, it serves as a means to identify an active manifold that is capable to capture the function's input-output map, rather than being the objective of the analysis itself. In fact, the main focus of this work is not the construction of a surrogate model for $\Q$, but rather the identification of a one-dimensional nonlinear manifold that captures the dominant variability of the model response.
In other words, the primary goal of the proposed approach is dimensionality reduction, not surrogate modeling. From this perspective, our approach is aligned with prior works such as~\cite{CDW14,BGF19}, which introduce active subspaces and active manifolds, respectively, as tools for identifying reduced representations of high-dimensional models. 
Nevertheless, once the manifold is identified, $\Q_\mathrm{S}$ can of course be used as a standalone model for, e.g., fast emulation and sensitivity analysis, assuming enough data is available to generate an accurate approximation. 
This can be checked by monitoring the second term in the loss function~\eqref{eq:loss}.

The following result shows that the minimization problem with loss function \eqref{eq:loss} admits a solution.

\begin{proposition} \label{pro:existence_minimizer}
The loss function $\mathcal L$ defined in \eqref{eq:loss} has at least one global minimizer.
\end{proposition}
\begin{proof}
First, notice that $\mathcal L \ge 0$. Then, define
\begin{equation}
\S^* = \mathcal I \qquad \text{and} \qquad \E^* = \Q,
\end{equation}
and set $\D^*$ to be a right inverse of $\Q$ such that
\begin{equation}
\Q \circ \D^* = \mathcal I.
\end{equation}
We remark that the right inverse exists if and only if the model $\Q$ is surjective, but it is always possible to restrict the codomain of $\Q$ in order to make it surjective, since it maps a $d$-dimensional domain into a one-dimensional one. We now notice that $\mathcal L(\E^*,\D^*,\S^*) = 0$ independently of the measure $\mu$, and therefore $(\E^*,\D^*,\S^*)$ is a global minimizer of $\mathcal L$.
\end{proof}

\begin{remark}
\cref{pro:existence_minimizer} gives the existence of a global minimizer of $\mathcal L$, but does not guarantee its uniqueness. In particular, it is possible to show that $\mathcal L$ has multiple global minimizers. First, from the proof of \cref{pro:existence_minimizer} notice that the choice of the right inverse $\D^*$ is not unique, since $\mathcal Q$ is not injective. Moreover, a family of minimizers could be obtained by setting
\begin{equation}
\mathcal S^* = \frac1{a} \mathcal I, \qquad \E^* = a\Q, \qquad (a\Q) \circ \mathcal D^* = \mathcal I,
\end{equation}
for a parameter $a\in\R$, $a \neq 0$.
\end{remark}

The solution provided by the proof of \cref{pro:existence_minimizer} is not useful from a practical point of view, as it relies on the model $\mathcal Q$, which is computationally expensive, and requires the computation of its right inverse. Therefore, we consider the parameterized encoder $\E(\cdot; \alpha)$, decoder $\D(\cdot; \beta)$, and surrogate model $\S(\cdot; \gamma)$, which are represented by neural networks, and compute their optimal parameters by minimizing the approximated loss function
\begin{equation} \label{eq:loss_discrete}
\begin{aligned}
\widetilde{\mathcal L}(\alpha,\beta,\gamma) &= \frac1N \sum_{n=1}^N (\Q(x_n) - \S(\E(\D(\E(x_n; \alpha); \beta); \alpha); \gamma))^2 \\
&\quad+ \frac1N \sum_{n=1}^N (\Q(x_n) - \S(\E(x_n; \alpha); \gamma))^2 \\
&\quad+ \frac1N \sum_{n=1}^N (\D(\E(x_n; \alpha); \beta) - \D(\E(\D(\E(x_n; \alpha); \beta); \alpha); \beta))^2.
\end{aligned}
\end{equation}
\begin{remark}
Since the surrogate model $\S$ defined on the latent space is a one-dimensional function, several methods can be used to approximate it, such as polynomial interpolation. In this work, for consistency, we decided to focus on neural networks because $\S$ is part of a larger structure which is already based on neural networks.
\end{remark}
In \cref{alg:important_manifolds} we summarize the main steps to construct a neural active manifold. The closeness of the solution of the optimization problem to the global minimizer given in \cref{pro:existence_minimizer} is dependent on the complexity of the model $\Q$, the expressive power of encoder and decoder, and the number $N$ of samples. Nevertheless, even if we cannot achieve a loss function equal to zero in practice, in \cref{sec:examples} we show that the NeurAM determined by this procedure provides a valid representation of the model $\Q$ in one dimension.

\begin{algorithm}
\caption{NeurAM} \label{alg:important_manifolds}
\begin{tabbing}
\hspace{-0.25cm}\textbf{Input:} \= Model $\Q$  \\
\> Input distribution $\mu$ \\
\> Number of samples $N$
\end{tabbing}
\begin{tabbing}
\hspace{-0.25cm}\textbf{Output:} \= Neural active manifold $\mathcal D(t; \beta^*)$ for $t$ in $\mathcal T$ \\
\> Encoder $\E(\cdot;\alpha^*)$ to obtain the latent space \\
\> Surrogate model $\S(\cdot;\gamma^*)$ on the latent space \\
\> Surrogate model $\Q_{\mathrm S}$ of $\Q$
\end{tabbing}
\begin{enumerate}[label=\arabic*:,itemindent=-0.75cm]
\item Parameterize the encoder through a neural network $\E(\cdot;\alpha)$.
\item Parameterize the decoder through a neural network $\D(\cdot;\beta)$.
\item Parameterize the surrogate model through a neural network $\S(\cdot;\gamma)$.
\item Draw the samples and evaluate the model to get $\{ (x_n, \Q(x_n) \}_{n=1}^N$ with $x_n \sim \mu$.
\item Define the loss function $\widetilde{\mathcal L}(\alpha,\beta,\gamma)$ given in equation \eqref{eq:loss_discrete}.
\item Compute the optimal parameters $(\alpha^*,\beta^*,\gamma^*) = \arg\min \widetilde{\mathcal L}(\alpha,\beta,\gamma)$.
\item Compute the domain of the neural active manifold $\mathcal T$ from equation \eqref{eq:domain_manifold}.
\item Define the surrogate model $\Q_{\mathrm S}$ using equation \eqref{eq:surrogateQ}.
\end{enumerate}
\end{algorithm}

\begin{remark}
Even if learning the NeurAM can be computationally expensive depending on the complexity of the neural networks representing encoder $\E(\cdot; \alpha)$, decoder $\D(\cdot; \beta)$, and surrogate model $\S(\cdot; \gamma)$, the  main advantages with respect to other techniques in the literature, such as active subspaces \cite{CDW14} or active manifolds \cite{BGF19}, are that the evaluation of the gradient of the model $\nabla \Q$ is never required and that we can identify a nonlinear manifold described by an arbitrary transformation from the original space, regardless of the function level sets. Moreover, we also train a surrogate model $\Q_{\mathrm S}$ without any additional cost.
\end{remark}

\begin{remark}
If the model $\Q$ is not defined on the whole space $\R^d$, but only on a domain $\Omega \subset \R^d$, then we need to enforce the NeurAM to lie on $\Omega$. Hence, we need to impose the codomain of the decoder $\D$ to be included in $\Omega$. This constraint can be easily enforced from the practical viewpoint if the domain is a box of the form $\Omega = \prod_{i=1}^d [a_i, b_i]$ with $a_i < b_i$ for all $i = 1, \dots, d$, as it is usually the case in applications. Otherwise, if the input parameters are correlated, one would need an invertible map $h$ from the domain $\Omega$ to a simpler box or to the whole space to decorrelate the variables. We leave the problem of determining such a map, which could be done using, e.g., normalizing flows \cite{KPB21}, for future work.
\end{remark}

\section{Application to uncertainty quantification} \label{sec:applications}

In this section we discuss possible applications of NeurAM in the field of uncertainty quantification. In addition to providing a surrogate model, our methodology can be used for sensitivity analysis. Moreover, the one-dimensional latent space can be employed to enhance the performance of multifidelity estimators in the context of uncertainty propagation.

\subsection{Sensitivity analysis} \label{sec:sensitivity_analysis}

Similar to AM~\cite{BGF19}, NeurAM can be used to perform a more informative sensitivity analysis than AS~\cite{CDW14}. Assume to have a sufficiently large number of data such that the NeurAM has been correctly identified and the surrogate model $\Q_{\mathrm S}$ is accurate. Then, by computing the derivatives of the surrogate model $\Q_{\mathrm S}$ with respect to the input parameters along the NeurAM, we can perform local sensitivity analysis, and quantify how the relevance, and consequently the identifiability, of each parameter varies along the manifold, hence providing a dynamic ranking. In particular, consider the quantities $\lambda_i$ for all $i = 1, \dots, d$ defined as
\begin{equation}
\lambda_i(x) = \frac1{\norm{\nabla \Q_{\mathrm S}(x)}^2} \abs{\frac{\partial \Q_{\mathrm S}}{\partial x_i}(x)}^2,
\end{equation}
and notice that they represent how much a single input parameter contributes to the variation of the model. Therefore, by evaluating the quantities $\lambda_i$ along the NeurAM, i.e., replacing $x = \D(t)$, we obtain a measure $\theta_i$ of importance of each parameter that varies along the manifold, and which is given by
\begin{equation} \label{eq:local_SA}
\theta_i(t) = \frac1{\norm{\nabla \Q_{\mathrm S}(\D(t))}^2} \abs{\frac{\partial \Q_{\mathrm S}}{\partial x_i}(\D(t))}^2, \qquad t \in \mathcal T,
\end{equation}
where, by definition \eqref{eq:surrogateQ} and using the chain rule, we have
\begin{equation}
\frac{\partial \Q_{\mathrm S}}{\partial x_i}(\D(t)) = \S'(\E(\D(t))) \frac{\partial \E}{\partial x_i}(\D(t)),
\end{equation}
with $\S'$ being the derivative of the surrogate $\S$ with respect to the latent variable $t$. Note that, from the practical point of view, the derivatives can be computed using automatic differentiation. We also remark that the local indices sum up to one at all points in the manifold, i.e., for all $t \in \mathcal T$
\begin{equation} \label{eq:local_sum1}
\sum_{i=1}^d \theta_i(t) = 1.
\end{equation}
Moreover, letting $\Gamma$ be the NeurAM, we can define global indices by integrating the dynamic indices along the manifold as
\begin{equation} 
\Theta_i = \frac1{\abs{\Gamma}} \int_\Gamma \lambda_i \dd \Gamma,
\end{equation}
for all $i = 1, \dots, d$, which give an overall ranking of the importance of the input parameters. The line integrals can then be computed from the parameterization $\D$ of the manifold $\Gamma$ yielding the indices
\begin{equation} \label{eq:global_SA}
\Theta_i = \frac{\int_{\mathcal T} \theta_i(t) \abs{\D'(t)} \dd t}{\int_{\mathcal T} \abs{\D'(t)} \dd t},
\end{equation}
where $\D'$ stands for the derivative of the decoder $\D$ with respect to the latent variable $t$, and which, due to the equality \eqref{eq:local_sum1}, satisfy
\begin{equation} 
\sum_{i=1}^d \Theta_i = 1.
\end{equation}
These indices, depending on whether the input parameters are normalized or not, might be sensitive to units. In fact, even if the variable $t \in \mathcal T$ describing the manifold is common for all the parameters, the components of the output of the decoder $\D$ have the same units as the inputs parameters of the model $\Q$. It is therefore always important to first normalize the dataset. We note that in this approach for sensitivity analysis based on a one-dimensional reduction, both the local and global indices give a ranking for the relevance of the input parameters without taking into account their interactions. Nevertheless, if one is interested in extracting more information regarding the interaction between variables, classic sensitivity indices, e.g., Sobol', can also be computed at a minimal computational cost by leveraging the surrogate model $\Q_{\mathrm S}$ provided by the NeurAM algorithm.

We finally remark that, even if we suggest to use the surrogate model $\Q_{\mathrm S}$ for sensitivity analysis, we do not always expect $\Q_{\mathrm S}$ to be a perfect approximation of the original model $\Q$. This surrogate is indeed built with the main purpose of helping to determine the NeurAM. Nevertheless, even if the error between $\Q$ and $\Q_{\mathrm S}$ is not negligible, the surrogate model can still be useful for discovering an active manifold, and identifying the most relevant input parameters through sensitivity analysis.

\subsection{Multifidelity uncertainty propagation} \label{sec:uncertainty_propagation}

Inspired by the works \cite{ZGS24,ZGS24b}, neural active manifolds can be used to increase the correlation between high-fidelity and low-fidelity models in the framework of multifidelity uncertainty propagation. For sake of generality, we limit our analysis to multifidelity Monte Carlo estimators~\cite{NgW14} with a single low-fidelity model. Nevertheless, we remark that a straightforward extension to other estimators, e.g., Approximate Control Variate~\cite{GGE20} and multiple low-fidelity sources, is also possible. We assume a computationally expensive high-fidelity model $\Q_\HF \colon \R^d \to \R$ and a cheaper low-fidelity model $\Q_\LF \colon \R^d \to \R$, and consider the problem of estimating the quantity
\begin{equation}
q = \Ex^\mu[\Q_\HF(X)],
\end{equation}
for some probability distribution $\mu$ on $\R^d$. We want to employ the multifidelity Monte Carlo estimator introduced in \cite{NgW14}. Let $w = \mathcal C_\LF/\mathcal C_\HF$ be the cost ratio between the two fidelities, and let $\mathcal B$ be the available computational budget given in terms of evaluations of the high-fidelity model, i.e.,
\begin{equation}
\mathcal B = N_\HF + wN_\LF,
\end{equation}
where $N_\HF$ and $N_\LF$ are the number of high-fidelity and low-fidelity evaluations, respectively. By solving an optimal allocation problem, it is possible to split the total computational budget between the high-fidelity and the low-fidelity models in such a way that the variance of the resulting multifidelity Monte Carlo estimator is minimized \cite{PWG16}. In particular, the solution of the allocation problem is given by
\begin{equation} \label{eq:optimal_allocation}
N_\HF = \frac{\mathcal B}{1 + w \gamma} \qquad \text{and} \qquad N_\LF = \gamma N_\HF = \frac{\gamma \mathcal B}{1 + w \gamma}, \qquad \text{with} \qquad \gamma = \sqrt{\frac{\rho^2}{w(1 - \rho^2)}},
\end{equation}
where $\rho$ is the Pearson correlation coefficient between the high-fidelity and low-fidelity models 
\begin{equation} \label{eq:correlation}
\rho = \frac{\Cov^\mu \left( \Q_\HF(X), \Q_\LF(X) \right)}{\sqrt{\Var^\mu[\Q_\HF(X)] \Var^\mu \left [\Q_\LF(X) \right]}}.
\end{equation}
After fixing the number of evaluations $N_\HF$ and $N_\LF$, we are now ready to define the multifidelity Monte Carlo estimator
\begin{equation} \label{eq:MFMC}
\widehat q = \frac1{N_\HF} \sum_{n=1}^{N_\HF} \Q_\HF(x_n) - \beta \left( \frac1{N_\HF} \sum_{n=1}^{N_\HF} \Q_\LF(x_n) - \frac1{N_\LF} \sum_{n=1}^{N_\LF} \Q_\LF(x_n) \right),
\end{equation}
where the coefficient $\beta$ is given by
\begin{equation} \label{eq:optimal_coefficient}
\beta = \frac{\Cov^\mu \left( \Q_\HF(X), \Q_\LF(X) \right)}{\Var^\mu \left[ \Q_\LF(X) \right]},
\end{equation}
and the samples $\{ x_n \}_{n=1}^{N_\LF}$ are drawn from the distribution $\mu$. We remark that the estimator $\widehat q$ is unbiased, i.e., $\Ex^\mu[\widehat q] = q$, and that the correlation $\rho$ in \eqref{eq:correlation} and the coefficient $\beta$ in \eqref{eq:optimal_coefficient} can be estimated from a pilot sample $\{ x_n \}_{n=1}^N$ with $N < N_\LF$. The performance of the multifidelity Monte Carlo estimator, i.e., its variance, is strongly dependent on the correlation $\rho$. In particular, we have
\begin{equation} \label{eq:variance_MFMC}
\Var^\mu \left[ \widehat q \right] = \frac1{\mathcal B} \Var^\mu \left[ \Q_\HF(X) \right] \left( \sqrt{1 - \rho^2} + \sqrt{w\rho^2} \right)^2,
\end{equation}
which is a decreasing function in $\rho^2$ as long as $\rho^2 > w/(1+w)$. Therefore, for a fixed computational budget, the larger the correlation is in modulus, the smaller the variance is, and consequently the estimation is more precise. 

We now adopt the methodology introduced in \cite{ZGS24} based on nonlinear dimensionality reduction to increase the correlation between the fidelities, and therefore decrease the variance of the estimator. The main limitation in \cite{ZGS24} is the need for a surrogate model to discover the lower-dimensional manifold. We overcome this restriction by using neural active manifolds, which do not rely on any surrogate model, but build a surrogate model on the latent space simultaneously with the learning of the lower-dimensional manifold. Moreover, in \cite{ZGS24} a shared space between the high-fidelity and low-fidelity models is built employing normalizing flows, which, however, may suffer from poor performance in presence of limited available data. Since our NeurAM is one-dimensional, we replace normalizing flows by the inverse transform method, and consider the uniform distribution $\mathcal U([0,1])$ as the reference distribution in the shared space. 

We proceed as follows. We first compute the NeurAM for both the high-fidelity and low-fidelity models by minimizing the loss function \eqref{eq:loss_discrete}, and thus obtaining 
\begin{equation}
(\E_\HF(\cdot; \alpha_\HF), \D_\HF(\cdot; \beta_\HF), \S_\HF(\cdot; \gamma_\HF)) \qquad \text{and} \qquad (\E_\LF(\cdot; \alpha_\LF), \D_\LF(\cdot; \beta_\LF), \S_\LF(\cdot; \gamma_\LF)).
\end{equation}
We note that, since the NeurAMs are trained independently, additional low-fidelity evaluations can be fully leveraged to construct a more accurate low-fidelity manifold. Then, let $\F_\HF$ and $\F_\LF$ be the cumulative distribution functions in the latent spaces of the two fidelities defined as
\begin{equation}
\F_\HF(t) = \Pr^\mu(\E_\HF(X; \alpha_\HF) \le t) \qquad \text{and} \qquad \F_\LF(t) = \Pr^\mu(\E_\LF(X; \alpha_\LF) \le t),
\end{equation}
which can be estimated using the empirical distributions given by the available pilot sample $\{ x_n \}_{n=1}^N$. Notice that the cumulative distribution functions satisfy
\begin{equation} \label{eq:inverse_transform_sampling}
(\F_\HF \circ \E_\HF(\cdot;\alpha_\HF))_\# \mu = \mathcal U([0,1]) \qquad \text{and} \qquad (\F_\LF \circ \E_\LF(\cdot;\alpha_\LF))_\# \mu = \mathcal U([0,1]),
\end{equation}
where $\#$ denotes the push-forward measure, and which means that
\begin{equation} 
\F_\HF(\E_\HF(X; \alpha_\HF)) \sim \mathcal U([0,1]) \qquad \text{and} \qquad \F_\LF(\E_\LF(X; \alpha_\LF)) \sim \mathcal U([0,1]).
\end{equation}
Therefore, we can define the modified low-fidelity model using the uniformly distributed shared space as a bridge between the two fidelities
\begin{equation} \label{eq:modified_QLF}
\widetilde\Q_\LF(x) = \Q_\LF(\D_\LF(\F_\LF^{-1}(\F_\HF(\E_\HF(x; \alpha_\HF))); \beta_\LF)),
\end{equation}
which can then replace the original $\Q_\LF$ in the definition of the multifidelity Monte Carlo estimator in equation \eqref{eq:MFMC}. A schematic visualization of the presented pipeline is given in \cref{fig:UQ_diagram}. The inverse of the cumulative distribution function $\F_\LF^{-1}$ stands for the generalized inverse defined as
\begin{equation}
\F_\LF^{-1}(u) = \inf \{ t \in \R \colon \F_\LF(t) \ge u \}, \qquad \text{for all } u \in [0,1].
\end{equation}

\begin{figure}[ht!]
\centering
\begin{tikzpicture}[scale=.6]
\draw[fill=green, fill opacity=.3] (-2,0) rectangle (-1,5);
\draw[fill=green, fill opacity=.3] (-0.5,0) -- (1.5,2) -- (1.5,3) -- (-0.5,5) -- cycle;
\draw[fill=green, fill opacity=.3] (2,2) rectangle (3,3);
\draw[fill=red, fill opacity=.3] (5.5,0) -- (3.5,2) -- (3.5,3) -- (5.5,5) -- cycle;
\draw[fill=red, fill opacity=.3] (6,0) rectangle (7,5);
\draw[fill=green, fill opacity=.3] (-1.5,7.5) circle (0.5);
\draw[fill=red, fill opacity=.3] (2.5,7.5) circle (0.5);
\draw [->, thick] (-1.5,5) -- (-1.5,7);
\draw[fill=red, fill opacity=.3] (2,3.5) -- (3,3.5) -- (4,4.5) -- (4,5.5) -- (3,6.5) -- (2,6.5) -- (1,5.5) -- (1,4.5) -- cycle;
\node at (-1.5,2.5) {$x$};
\node at (2.5,2.5) {$t_\HF$};
\node at (6.5,2.5) {$\widetilde x_\HF$};
\node at (-0.5,5.9) {\LARGE $\mathcal Q_\HF$};
\node at (0.5,2.5) {\LARGE $\mathcal E_\HF$};
\node at (4.5,2.5) {\LARGE $\mathcal D_\HF$};
\node at (2.5,5) {\LARGE $\mathcal S_\HF$};
\draw[fill=red, fill opacity=.3] (10,0) rectangle (11,5);
\draw[fill=red, fill opacity=.3] (11.5,0) -- (13.5,2) -- (13.5,3) -- (11.5,5) -- cycle;
\draw[fill=green, fill opacity=.3] (14,2) rectangle (15,3);
\draw[fill=green, fill opacity=.3] (17.5,0) -- (15.5,2) -- (15.5,3) -- (17.5,5) -- cycle;
\draw[fill=green, fill opacity=.3] (18,0) rectangle (19,5);
\draw[fill=red, fill opacity=.3] (14.5,7.5) circle (0.5);
\draw[fill=red, fill opacity=.3] (14,3.5) -- (15,3.5) -- (16,4.5) -- (16,5.5) -- (15,6.5) -- (14,6.5) -- (13,5.5) -- (13,4.5) -- cycle;
\node at (10.5,2.5) {$x$};
\node at (12.5,2.5) {\LARGE $\mathcal E_\LF$};
\node at (16.5,2.5) {\LARGE $\mathcal D_\LF$};
\node at (14.5,5) {\LARGE $\mathcal S_\LF$};
\node at (14.5,2.5) {$t_\LF$};
\node at (18.5,2.5) {$\widetilde x_\LF$};
\draw [->, thick] (18.5,5) -- (18.5,7);
\node at (19.5,5.9) {\LARGE $\mathcal Q_\LF$};
\draw[fill=green, fill opacity=.3] (18.5,7.5) circle (0.5);
\draw[fill=green, fill opacity=.3] (8,-1.5) rectangle (9,-0.5);
\draw [->, thick] (2.5,2) -- (2.5,-1) -- (8,-1);
\draw [->, thick] (9,-1) -- (14.5,-1) -- (14.5,2);
\node at (8.5,-1) {$u$};
\node at (3.5,-0.5) {\LARGE $\mathcal F_\HF$};
\node at (13.6,-0.4) {\LARGE $\mathcal F_\LF^{-1}$};
\end{tikzpicture}
\caption{Schematic representation of the multifidelity uncertainty propagation pipeline based on NeurAM.}
\label{fig:UQ_diagram}
\end{figure}
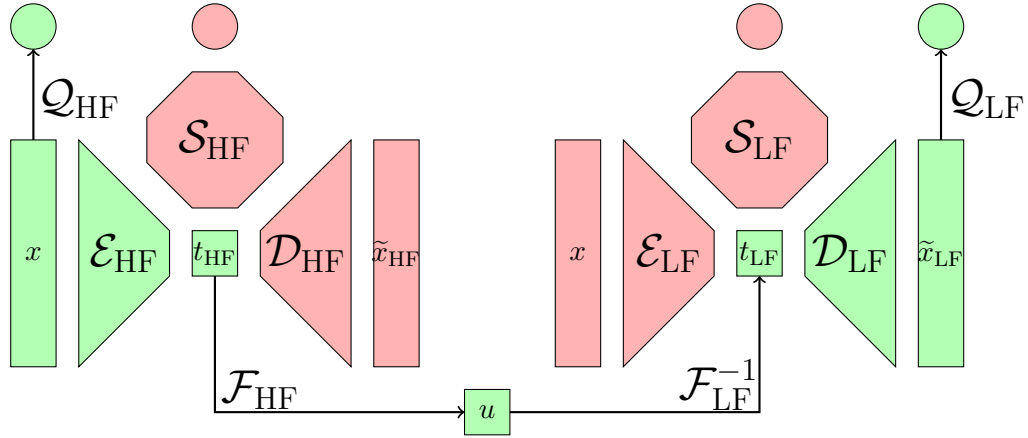

\begin{remark}
We emphasize that the model $\widetilde\Q_\LF$ is purely a reparameterization of the original low-fidelity model $\Q_\LF$, which does not affect the computational cost, except for the evaluations of the networks that can be assumed negligible. In particular, the NeurAM of both the high-fidelity and low-fidelity models have only been used to determine such parameterization. Therefore, the multifidelity estimator remains unbiased since we did not modify the high-fidelity model $\Q_\HF$, and the surrogate models $\S_\HF \circ \E_\HF$ and $\S_\LF \circ \E_\LF$ have not been used to compute the evaluations of the models. 
\end{remark}
\begin{remark}
In this section, for simplicity, we outlined the methodology assuming the high-fidelity and low-fidelity models to have the same parameterization. However, we note that the presented approach is also able to handle models with dissimilar parameterization due to the reduction to a shared one-dimensional manifold; see also~\cite{GEG18,ZGE23}.
\end{remark}
Based on the ideas in \cite{ZGS24}, we expect the modified low-fidelity model to be better correlated with the high-fidelity model, and therefore we aim to achieve a smaller variance for the multifidelity Monte Carlo estimator. Even if the NeurAM of the high-fidelity model $\Q_\HF$ might be less accurate due to the scarcity of available data, the reparameterization $\widetilde\Q_\LF$ of the low-fidelity model, despite being nonoptimal, could still provide a better correlation than the original model $\Q_\LF$. In the next section we consider the idealized setting given by the proof of \cref{pro:existence_minimizer}, and show that, at least in this case, the new correlation 
\begin{equation} \label{eq:modified_correlation}
\widetilde\rho = \frac{\Cov^\mu \left( \Q_\HF(X), \widetilde\Q_\LF(X) \right)}{\sqrt{\Var^\mu[\Q_\HF(X)] \Var^\mu \left [\widetilde\Q_\LF(X) \right]}}
\end{equation}
is indeed better than the original correlation $\rho$ given in \eqref{eq:correlation}. This gives a theoretical guarantee under specific conditions, and therefore justifies the approach described in this section.

\subsubsection{Analysis under idealized setting} \label{sec:formal_analysis}

In this section we assume to reach a loss function equal to zero, and we consider the global minimizer given in the proof of \cref{pro:existence_minimizer}. In particular, let
\begin{equation}
\S_\HF = \mathcal I,  \; \E_\HF = \Q_\HF, \; \Q_\HF \circ \D_\HF = \mathcal I, \quad \text{and} \quad \S_\LF = \mathcal I,  \; \E_\LF = \Q_\LF, \; \Q_\LF \circ \D_\LF = \mathcal I,
\end{equation}
which, due to equation \eqref{eq:modified_QLF}, gives
\begin{equation} \label{eq:modified_QLF_ideal}
\widetilde\Q_\LF(x) = \F_\LF^{-1}(\F_\HF(\Q_\HF(x))).
\end{equation}
Moreover, let us assume that both $\F_\HF$ and $\F_\LF$ are invertible, and, without loss of generality, that $\rho \ge 0$. In fact, if $\rho < 0$, then we can always define $\widetilde\Q_\LF$ starting from $-\Q_\LF$ instead of $\Q_\LF$. 
We can now state the main result of this section: the modified low-fidelity model $\widetilde\Q_\LF$ is better correlated with the high-fidelity model $\Q_\HF$ than the original low-fidelity $\Q_\LF$.
\begin{theorem}
Let $\widetilde\Q_\LF$ be defined as in equation \eqref{eq:modified_QLF_ideal}, and let the correlations $\rho$ and $\widetilde\rho$ be given in \eqref{eq:correlation} and \eqref{eq:modified_correlation}, respectively. Then, it holds
\begin{equation}
\widetilde\rho \ge \rho.
\end{equation}
\end{theorem}
\begin{proof}
First, notice that by the inverse transform sampling $\Q_\LF(X)$ and $\widetilde\Q_\LF(X)$ have the same distribution, and therefore
\begin{equation} \label{eq:equal_variance}
\Ex^\mu \left [\widetilde\Q_\LF(X) \right] = \Ex^\mu \left [\Q_\LF(X) \right] \qquad \text{and} \qquad \Var^\mu \left [\widetilde\Q_\LF(X) \right] = \Var^\mu \left [\Q_\LF(X) \right],
\end{equation}
which implies that $\widetilde \rho \ge \rho$ if and only if 
\begin{equation} \label{eq:goal0}
\Cov^\mu \left( \Q_\HF(X), \widetilde\Q_\LF(X) \right) \ge \Cov^\mu \left( \Q_\HF(X), \Q_\LF(X) \right).
\end{equation}
Then, rewriting the covariances as
\begin{equation}
\begin{aligned}
\Cov^\mu \left( \Q_\HF(X), \widetilde\Q_\LF(X) \right) &= \Ex^\mu[\Q_\HF(X) \widetilde\Q_\LF(X)] - \Ex^\mu[\Q_\HF(X)]  \Ex^\mu[\widetilde\Q_\LF(X)], \\
\Cov^\mu \left( \Q_\HF(X), \Q_\LF(X) \right) &= \Ex^\mu[\Q_\HF(X) \Q_\LF(X)] - \Ex^\mu[\Q_\HF(X)]  \Ex^\mu[\Q_\LF(X)],
\end{aligned}
\end{equation}
and due to equation \eqref{eq:equal_variance}, we deduce that \eqref{eq:goal0} is satisfied if and only if
\begin{equation}
\Ex^\mu[\Q_\HF(X) \widetilde\Q_\LF(X)] \ge \Ex^\mu[\Q_\HF(X) \Q_\LF(X)].
\end{equation}
Let us now rewrite the left-hand side using the fact that the high-fidelity model can be rewritten as
\begin{equation}
\widetilde\Q_\HF(x) = \F_\HF^{-1}(\F_\HF(\Q_\HF(x))).
\end{equation}
In particular, by equation \eqref{eq:inverse_transform_sampling} with $\E_\HF = \Q_\HF$, we have
\begin{equation}
\begin{aligned}
\Ex^\mu[\Q_\HF(X) \widetilde\Q_\LF(X)] &= \int \F_\HF^{-1}(\F_\HF(\Q_\HF(x))) \F_\LF^{-1}(\F_\HF(\Q_\HF(x))) \dd \mu(x), \\
&= \int \F_\HF^{-1}(u) \F_\LF^{-1}(u) \dd (\F_\HF \circ \Q_\HF)_\# \mu(u) \\
&= \int \F_\HF^{-1}(u) \F_\LF^{-1}(u) \dd \mathcal U (u) \\
&= \Ex^{\mathcal U}[\F_\HF^{-1}(U) \F_\LF^{-1}(U)],
\end{aligned}
\end{equation}
where the superscript $\mathcal U$ denotes the fact that the expectation is computed with respect to the uniform distribution $\mathcal U([0,1])$. Therefore, in order to conclude the proof, we have to show that
\begin{equation} \label{eq:goal}
\Ex^{\mathcal U}[\F_\HF^{-1}(U) \F_\LF^{-1}(U)] \ge \Ex^\mu[\Q_\HF(X) \Q_\LF(X)].
\end{equation}
We analyze the two sides separately, and we first consider the left-hand side. The expectation can be rewritten as
\begin{equation}
\begin{aligned}
\Ex^{\mathcal U}[\F_\HF^{-1}(U) \F_\LF^{-1}(U)] &= \int_0^\infty \int_0^\infty \Pr^{\mathcal U}( \F_\HF^{-1}(U) \ge t, \F_\LF^{-1}(U) \ge s ) \dd t \dd s \\
&\quad + \int_{-\infty}^0 \int_{-\infty}^0 \Pr^{\mathcal U}( \F_\HF^{-1}(U) \le t, \F_\LF^{-1}(U) \le s ) \dd t \dd s \\
&\quad - \int_{-\infty}^0 \int_0^\infty \Pr^{\mathcal U}( \F_\HF^{-1}(U) \ge t, \F_\LF^{-1}(U) \le s ) \dd t \dd s \\
&\quad - \int_0^\infty \int_{-\infty}^0 \Pr^{\mathcal U}( \F_\HF^{-1}(U) \le t, \F_\LF^{-1}(U) \ge s ) \dd t \dd s,
\end{aligned}
\end{equation}
which, since $\F_\HF$ and $\F_\LF$ are strictly increasing, implies
\begin{equation}
\begin{aligned}
\Ex^{\mathcal U}[\F_\HF^{-1}(U) \F_\LF^{-1}(U)] &= \int_0^\infty \int_0^\infty \Pr^{\mathcal U}( U \ge \F_\HF(t), U \ge \F_\LF(s) ) \dd t \dd s \\
&\quad + \int_{-\infty}^0 \int_{-\infty}^0 \Pr^{\mathcal U}( U \le \F_\HF(t), U \le \F_\LF(s) ) \dd t \dd s \\
&\quad - \int_{-\infty}^0 \int_0^\infty \Pr^{\mathcal U}( U \ge \F_\HF(t), U \le \F_\LF(s) ) \dd t \dd s \\
&\quad - \int_0^\infty \int_{-\infty}^0 \Pr^{\mathcal U}( U \le \F_\HF(t), U \ge \F_\LF(s) ) \dd t \dd s.
\end{aligned}
\end{equation}
Then, using the fact that $U \sim \mathcal U([0,1])$, we obtain
\begin{equation} \label{eq:decompostion_LHS}
\begin{aligned}
\Ex^{\mathcal U}[\F_\HF^{-1}(U) \F_\LF^{-1}(U)] &= \int_0^\infty \int_0^\infty (1 - \max\{ \F_\HF(t), \F_\LF(s) \}) \dd t \dd s \\
&\quad + \int_{-\infty}^0 \int_{-\infty}^0 \min\{ \F_\HF(t), \F_\LF(s) \} \dd t \dd s \\
&\quad - \int_{-\infty}^0 \int_0^\infty \max\{ 0, \F_\LF(s) - \F_\HF(t) \} \dd t \dd s \\
&\quad - \int_0^\infty \int_{-\infty}^0 \max\{ 0, \F_\HF(t) - \F_\LF(s) \} \dd t \dd s.
\end{aligned}
\end{equation}
On the other hand, the right-hand side of \eqref{eq:goal} reads
\begin{equation} \label{eq:decomposition_RHS}
\begin{aligned}
\Ex^\mu[\Q_\HF(X) \Q_\LF(X)] &= \int_0^\infty \int_0^\infty \Pr^\mu( \Q_\HF(X) \ge t, \Q_\LF(X) \ge s ) \dd t \dd s \\
&\quad + \int_{-\infty}^0 \int_{-\infty}^0 \Pr^\mu( \Q_\HF(X) \le t, \Q_\LF(X) \le s ) \dd t \dd s \\
&\quad - \int_{-\infty}^0 \int_0^\infty \Pr^\mu( \Q_\HF(X) \ge t, \Q_\LF(X) \le s ) \dd t \dd s \\
&\quad - \int_0^\infty \int_{-\infty}^0 \Pr^\mu( \Q_\HF(X) \le t, \Q_\LF(X) \ge s ) \dd t \dd s.
\end{aligned}
\end{equation}
We now compare each term in the right-hand side of equations \eqref{eq:decompostion_LHS} and \eqref{eq:decomposition_RHS} separately. First, we have
\begin{equation}
\begin{aligned}
\max\{ \F_\HF(t), \F_\LF(s) \} &= \max\{ \Pr^\mu( \Q_\HF(X) \le t ), \Pr^\mu( \Q_\LF(X) \le s ) \} \\
&\le \Pr^\mu( \Q_\HF(X) \le t ) + \Pr^\mu( \Q_\LF(X) \le s ) \\
&\quad - \Pr^\mu( \Q_\HF(X) \le t, \Q_\LF(X) \le s ),
\end{aligned}
\end{equation}
which implies
\begin{equation} \label{eq:bound1}
\max\{ \F_\HF(t), \F_\LF(s) \} \ge \Pr^\mu( \Q_\HF(X) \ge t, \Q_\LF(X) \ge s ).
\end{equation}
Then, for the second term we find
\begin{equation} \label{eq:bound2}
\begin{aligned}
\min\{ \F_\HF(t), \F_\LF(s) \} &= \min\{ \Pr^\mu( \Q_\HF(X) \le t ), \Pr^\mu( \Q_\LF(X) \le s ) \} \\
&\ge \Pr^\mu( \Q_\HF(X) \le t, \Q_\LF(X) \le s ).
\end{aligned}
\end{equation}
For the third term we have
\begin{equation} \label{eq:bound3}
\begin{aligned}
\Pr^\mu( \Q_\HF(X) \ge t, \Q_\LF(X) \le s ) &= \Pr^\mu( \Q_\LF(X) \le s ) - \Pr^\mu( \Q_\HF(X) \le t, \Q_\LF(X) \le s ) \\
&\ge  \max\{ 0, \Pr^\mu( \Q_\LF(X) \le s ) - \Pr^\mu( \Q_\HF(X) \le t) \} \\
&= \max\{ 0, \F_\LF(s) - \F_\HF(t) \},
\end{aligned}
\end{equation}
and similarly for the fourth term
\begin{equation} \label{eq:bound4}
\begin{aligned}
\Pr^\mu( \Q_\HF(X) \le t, \Q_\LF(X) \ge s ) &= \Pr^\mu( \Q_\HF(X) \le t ) - \Pr^\mu( \Q_\HF(X) \le t, \Q_\LF(X) \le s ) \\
&\ge  \max\{ 0, \Pr^\mu( \Q_\HF(X) \le t ) - \Pr^\mu( \Q_\LF(X) \le s) \} \\
&= \max\{ 0, \F_\HF(t) - \F_\LF(s) \}.
\end{aligned}
\end{equation}
Finally, collecting the bounds \eqref{eq:bound1}, \eqref{eq:bound2}, \eqref{eq:bound3}, \eqref{eq:bound4} together with the decompositions \eqref{eq:decompostion_LHS}, \eqref{eq:decomposition_RHS}, we deduce the inequality \eqref{eq:goal}, which implies the desired result.
\end{proof}

\begin{remark} \label{rem:ideal}
The analysis in this section is performed under the assumption that the neural active manifold is constructed using a global minimizer of the loss function, and hence we put ourselves in the idealized framework in which the models can be reconstructed exactly from the one-dimensional NeurAM. Even if this cannot be achieved in practice, this theoretical analysis gives a formal justification for why it is worth modifying the low-fidelity model, i.e., draw new samples from the shared space obtained from the NeurAM of the two fidelities. Therefore, we remark that, even if we cannot achieve the best correlation $\widetilde\rho$ in practice, our methodology, in most of the cases, gives a new correlation that is better than the original correlation $\rho$, as we observe in the following numerical experiments. Notice that the ideal correlation obtained from equation \eqref{eq:goal}, i.e.,
\begin{equation} \label{eq:ideal_corr}
\widetilde \rho = \frac{\Cov^{\mathcal U} \left( \F_\HF^{-1}(U), \F_\LF^{-1}(U) \right)}{\sqrt{\Var^{\mathcal U}[\F_\HF^{-1}(U)] \Var^{\mathcal U} [\F_\LF^{-1}(U)]}},
\end{equation}
can be used as a reference value for the new correlation that we can get in practice, since it is computed in the idealized framework where the reduction to one dimension still allows us to reconstruct the original model exactly. Finally, we emphasize that the effectiveness of our strategy relies on accurately capturing the models variability through the low-dimensional manifolds. Therefore, the resulting correlation directly depends on the amount of available data and the complexity of the models. Nevertheless, even with a limited computational budget, we should be able to achieve larger correlation, and consequently variance reduction, although it may not be optimal. Otherwise, if no improvement in terms of correlation is observed, the standard multifidelity approach based on the original model inputs remains a valid alternative.
\end{remark}

\section{Numerical examples} \label{sec:examples}

In this section we apply our algorithm to construct a neural active manifold for several models.
We first consider simple two-dimensional analytical functions to demonstrate the performance of NeurAM while providing intuition on the process of reducing the problem dimensionality by directly plotting the active manifold. Two computationally challenging test cases are then analyzed to demonstrate that the approach is viable for realistic problems. 
The first problem we consider is the so-called Hartmann problem \cite{SPC16}, which has been used previously as a test case for other dimensionality reduction strategies, such as AS~\cite{GCS17} and AM~\cite{BGF19}. 
We also study a more complex cardiac electrophysiology model with biphasic response, requiring the solution of a preliminary classification task. 
Finally, we compare the multifidelity UQ pipeline introduced in the previous section with standard multifidelity Monte Carlo estimators for model with multiple inputs, and perform sensitivity analysis. 
We remark that, for each numerical example, we use a different set of hyperparameters (layers and neurons) for the neural networks corresponding to the encoder $\E$, decoder $\D$, and surrogate model $\S$. 
All the hyperparameters are optimized using Optuna~\cite{ASY19} monitoring the validation loss (20\% of the dataset) for $100$ iterations.
In particular, we constrain the number of layers to be in $\{ 1, \dots, 4 \}$ and the number of neurons per layer to be in $\{ 1, \dots, 16 \}$, and we train all the networks for 10000 epochs. Hence, we emphasize that the resulting neural networks have a limited size, implying a minimal computational cost for their training. 

\subsection{Two-dimensional models} \label{sec:2Dmodels}

We analyze several two-dimensional synthetic test cases to assess the performance of the proposed approach and visualize the neural active manifold.

\subsubsection{Parabolic neural active manifold} \label{sec:parabola}

We first consider the model $Q \colon \R^2 \to \R$ defined as
\begin{equation}
Q(x) = x_1^2 + x_2,
\end{equation}
and let the distribution of the input parameters be $\mu = \mathcal U([0,1]^2)$. Notice that this is an example where the global minimizer given by the proof of \cref{pro:existence_minimizer} can be computed analytically. In particular, we have
\begin{equation} \label{eq:parabola_NeurAM}
\S^*(t) = t, \qquad \E^*(x) = x_1^2 + x_2, \qquad \D^*(t) = \begin{bmatrix} \sqrt{\frac{t}2} & \frac{t}2 \end{bmatrix}^\top,
\end{equation}
which implies that $\mathcal L(\E^*, \D^*, \S^*) = 0$. The NeurAM is then parameterized by the decoder $\D^*(t)$ for all values of $t$ in the interval
\begin{equation}
\mathcal T = \left[ \min_{x_1,x_2\in[0,1]} (x_1^2 + x_2), \max_{x_1,x_2\in[0,1]} (x_1^2 + x_2) \right] = [0, 2],
\end{equation}
and is represented by a parabola of the form $x_2 = x_1^2$ for all $x_1 \in [0,1]$. 
In \cref{fig:parabola_manifold} we illustrate the NeurAM that we computed analytically on top of the contour plot of the model, and we show through red arrows how the points are projected on the manifold by the autoencoder, i.e., computing $\D(\E(x))$. 
Moreover, from the parity plot on the right, we observe that the exact and surrogate models coincide, i.e., $\Q(x)=\Q_{\mathrm S}(x) = \S(\E(x))$, resulting in the entire model variance being captured.

Let us now slightly modify the function $Q$, and consider a more complex model $\Q \colon \R^2 \to \R$ given by
\begin{equation}
\Q(x) = \sin(Q(x)) = \sin(x_1^2 + x_2),
\end{equation}
for which the NeurAM remains the same if we choose as surrogate model $\S(t) = \sin(t)$. 
We then run our algorithm with an increasing number of training samples $N = 10,$ $ 30,$ $50, 100, 500, 1000$, and compute the mean absolute error (MAE) and the mean squared error (MSE) for 
\begin{equation} \label{eq:errors_def}
e_1(x) = \Q(x) - \Q(\D(\E(x))), \qquad e_2(x) = \Q(x) - \Q_{\mathrm S}(x),
\end{equation}
over $100$ repetitions. 
Note that $e_1$ only measures the error due to the dimensionality reduction, while $e_2$ quantifies the error in the surrogate model. 
The MAE and MSE for $e_1$ and $e_2$ are computed using a testing set with $1000$ randomly selected samples. 
In \cref{fig:parabola_error} we present the resulting error trends, using error bars to represent standard deviations produced by repeated runs. 
As expected, both errors decrease with the increase of the training dataset, until they reach a plateau. We remark that the standard deviation of the model $\Q$ over its domain is
\begin{equation}
\sqrt{\Ex^\mu \left[ \left( \Q(X) - \Ex^\mu[\Q(X)] \right)^2 \right]} \simeq 0.258,
\end{equation}
implying that the relative errors are approximately of the same order of the absolute errors reported in \cref{fig:parabola_error}.

\begin{figure}[ht!]
\centering
\includegraphics{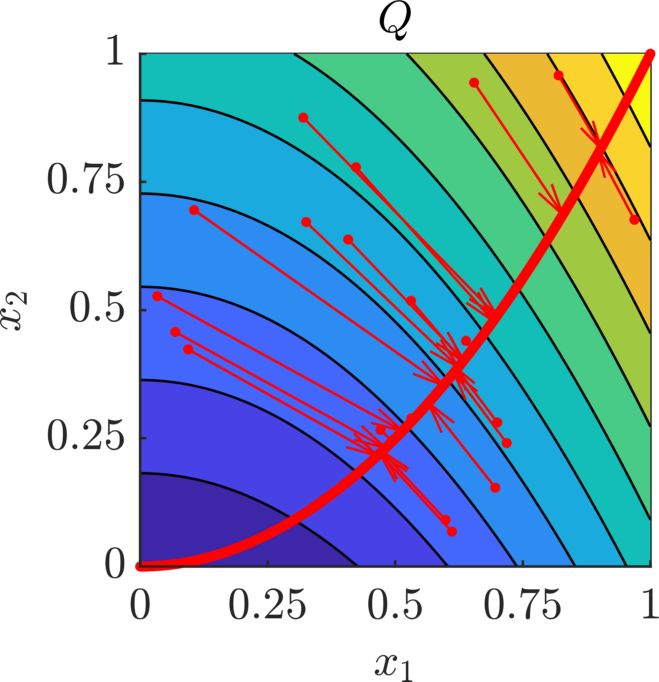} \hspace{2cm}
\includegraphics{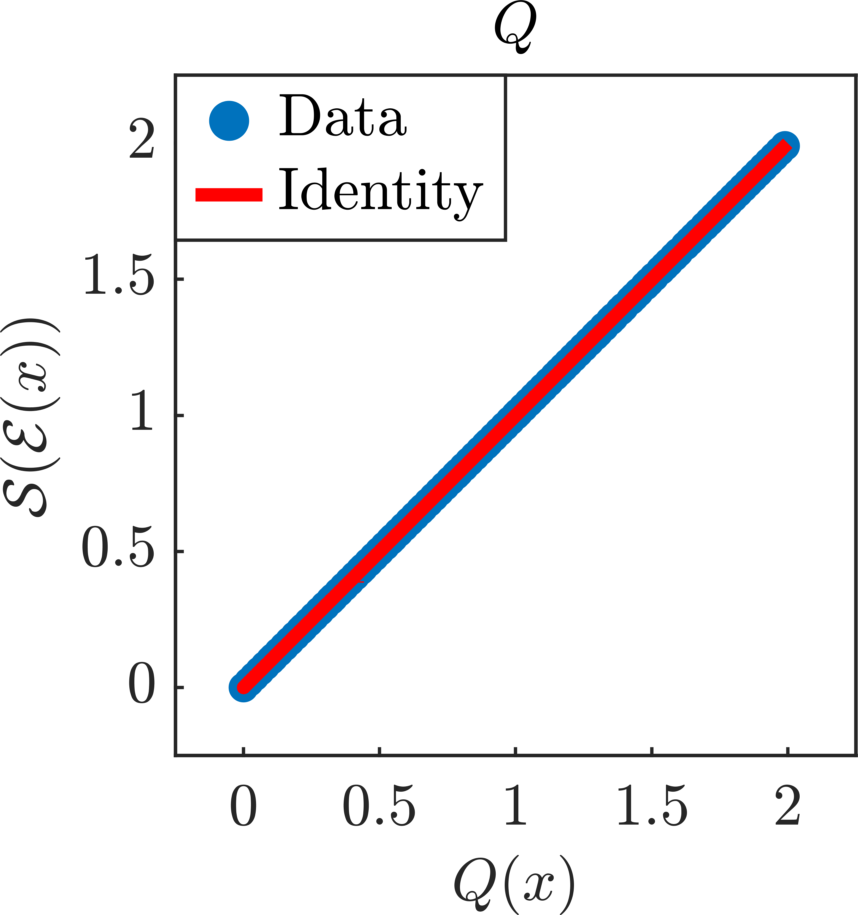}
\caption{Analytical results for the two-dimensional model $Q$. Left: projection of input samples on the one-dimensional NeurAM. Right: parity plot between the exact and surrogate models.}
\label{fig:parabola_manifold}
\end{figure}

\begin{figure}[ht!]
\centering
\includegraphics{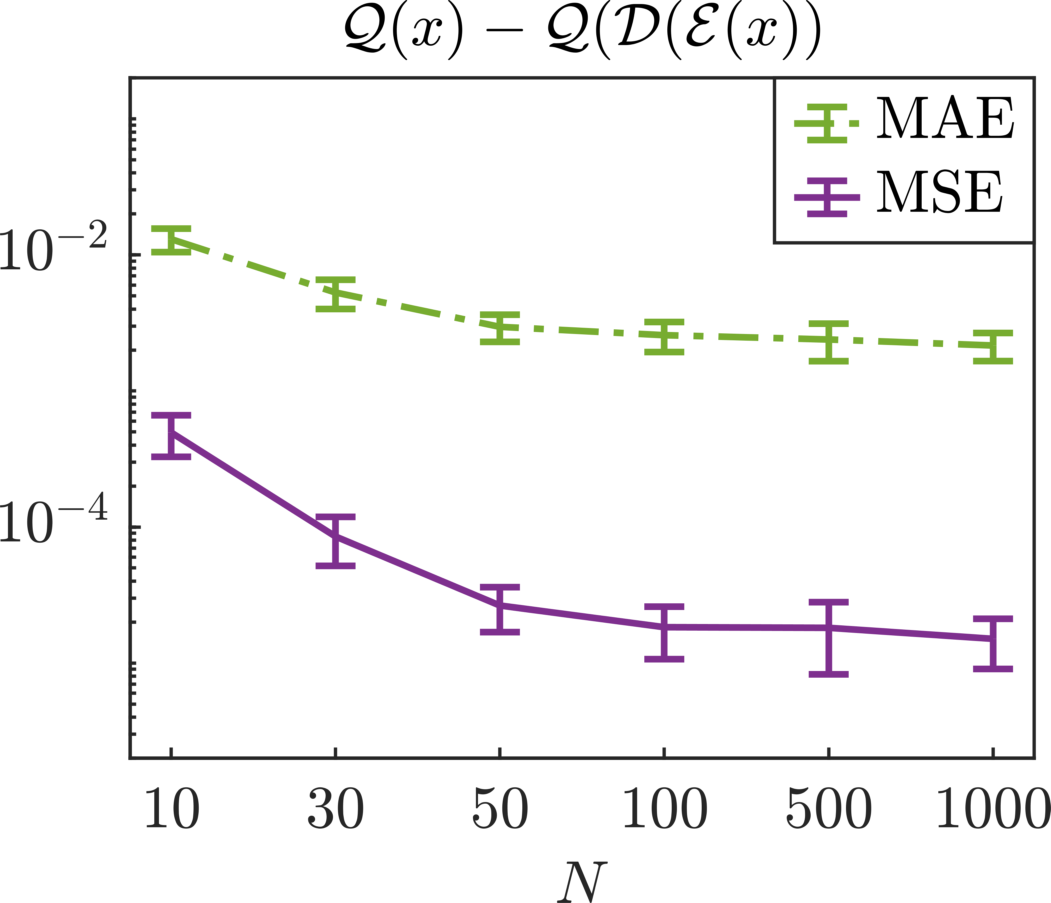} \hspace{1cm}
\includegraphics{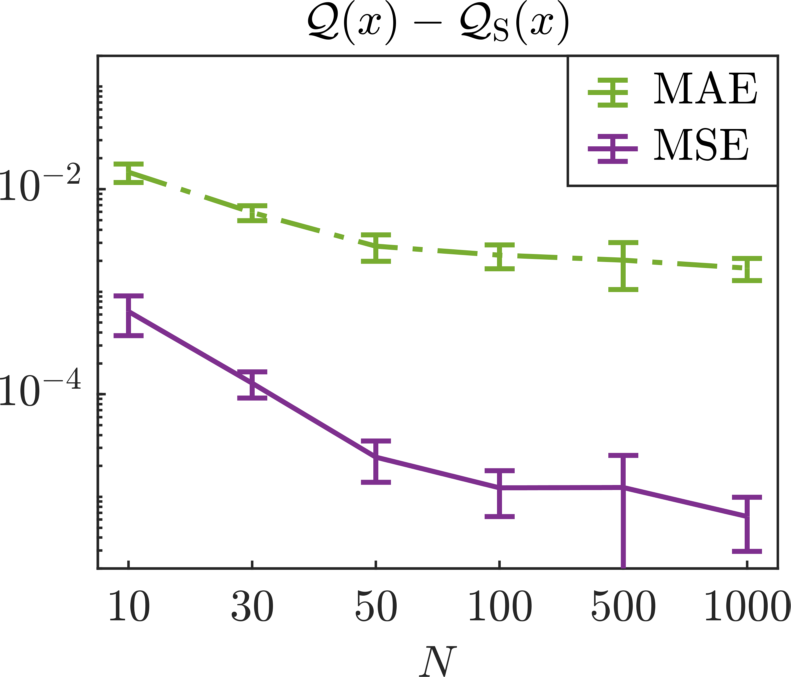}
\caption{Approximation errors (MAE and MSE) in equation \eqref{eq:errors_def} as functions of the size of the training set $N$ for the two-dimensional model $\Q$. Left: error $e_1$ due to the dimensionality reduction. Right: error $e_2$ of the surrogate model.}
\label{fig:parabola_error}
\end{figure}

\subsubsection{Comparison with active manifolds and active subspaces} \label{sec:numerical_comparison}

We consider the following two-dimensional models that have been used to test AM in~\cite{BGF19}
\begin{equation}
\begin{aligned}
\Q_1(x) &= e^{x_2 - x_1^2}, \\
\Q_2(x) &= x_1^2 + x_2^2, \\
\Q_3(x) &= x_1^3 + x_2^3 + 0.2x_1 + 0.6x_2,
\end{aligned}
\end{equation}
and we let the distribution of the input parameters be $\mu = \mathcal U([-1,1]^2)$.
We first compute NeurAM using $N = 1000$ samples. 
The results in \cref{fig:paperAM_manifold} show the projections determined through dimensionality reduction (red arrows), while the parity plots show how the surrogate model on the latent space, i.e., as a function of the reduced variable, is able to capture practically the entire variability for all the models. We remark that these plots are obtained for one possible realization of the NeurAM, meaning that the method can generate alternative NeurAM with similar overall results.

We then compare the performance of NeurAM with results for AS and AM, as reported in \cite[Table 1]{BGF19}, where the gradients of the models are computed analytically.
We run our algorithm $100$ times using $N=1000$ training samples, and we compute the errors using equation \eqref{eq:errors_def} on a testing set of size $1000$. 
Even though our approach presents a slightly larger variability, which is mainly due to the fact that the solution of the optimization algorithm is not unique, \cref{fig:paperAM_comparison} shows that NeurAM is consistently able to achieve smaller errors than competing algorithms. 
Moreover, our method is always able to find a lower-dimensional representation, unlike AM for which a percentage of points in the domain is left out as shown in the table in \cref{fig:paperAM_comparison}, meaning that we can always compute the projection of a point on the lower-dimensional manifold, similarly to AS. Finally, NeurAM requires fewer training points (1000 for NeurAM and 8000 for AM and AS as stated in~\cite{BGF19}), and, more crucially, does not need the computation of the  gradient, which is instead necessary for both AS and AM.

\begin{figure}[ht!]
\centering
\begin{tabular}{ccc}
\includegraphics{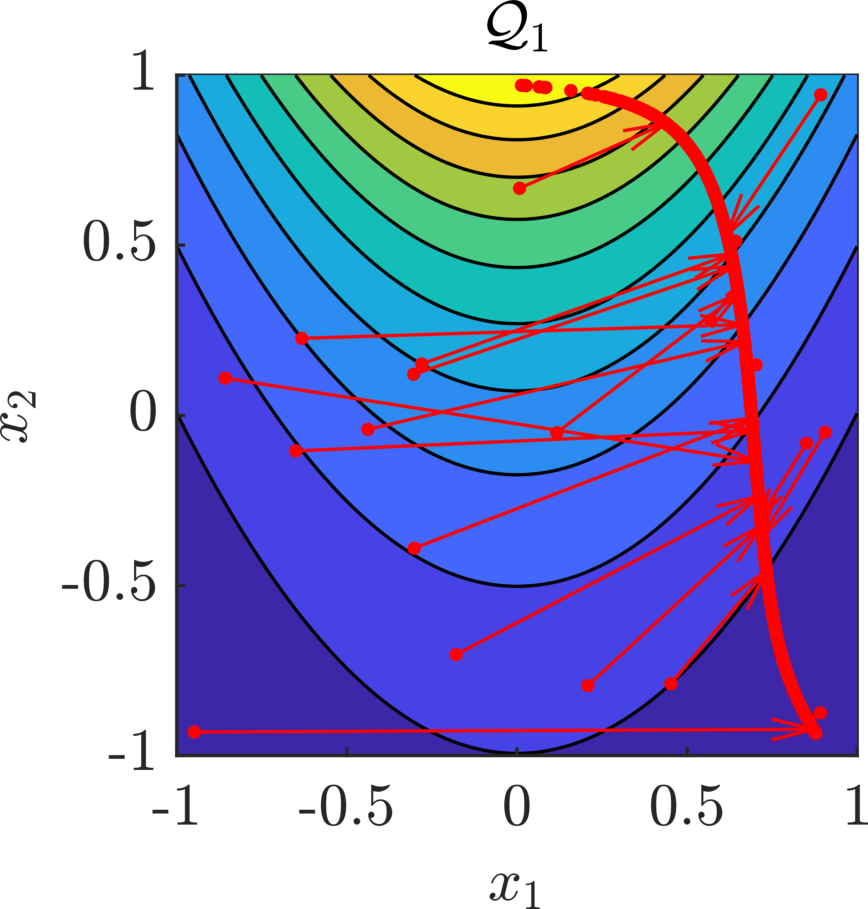} &
\includegraphics{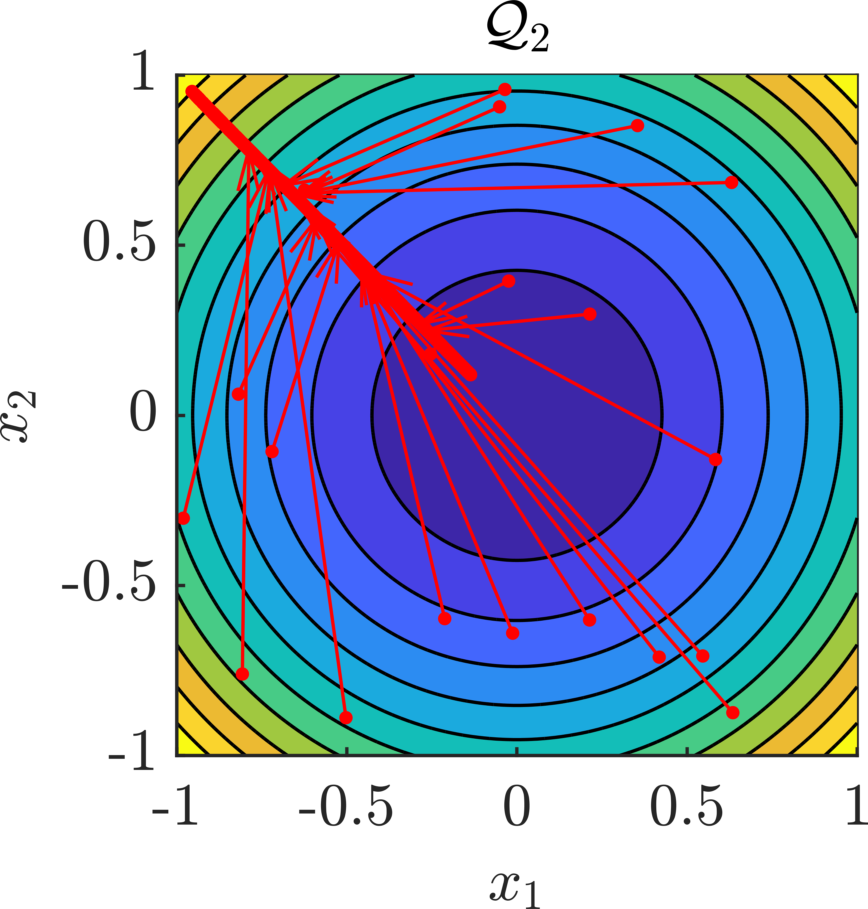} &
\includegraphics{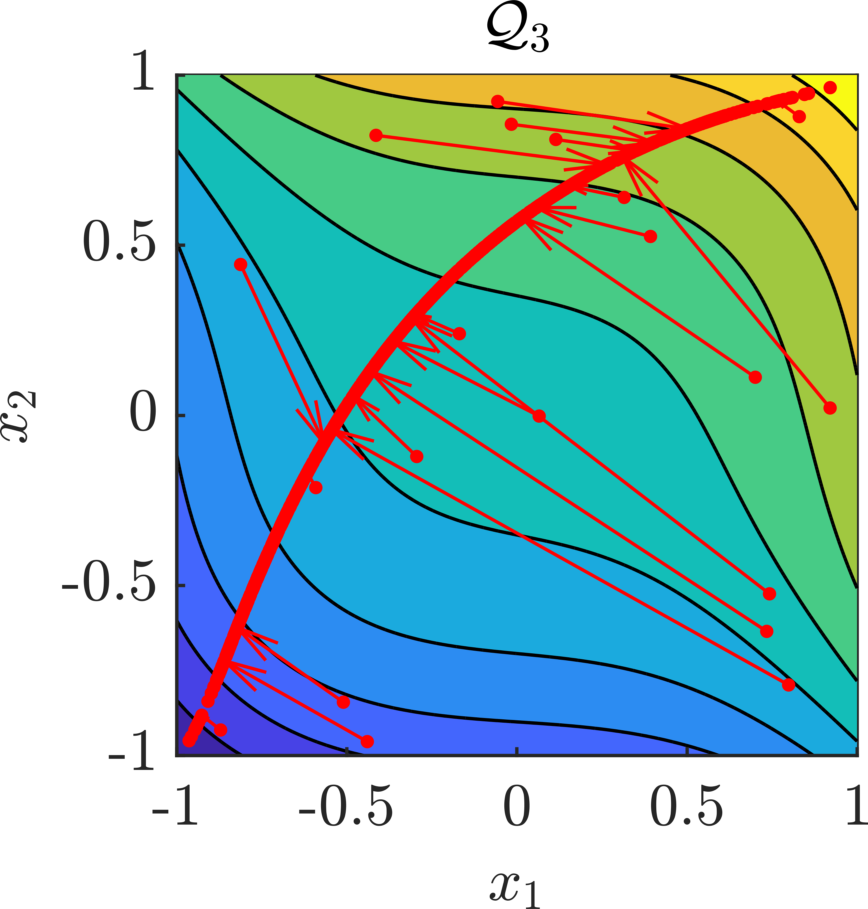} \\
&&\\
\includegraphics{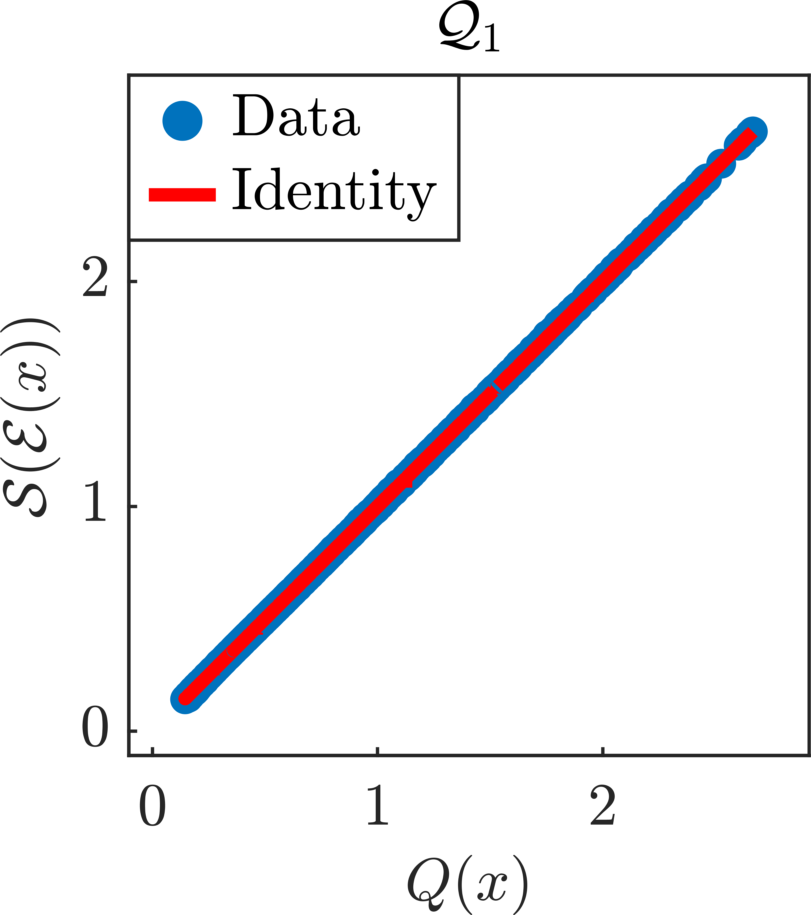} &
\includegraphics{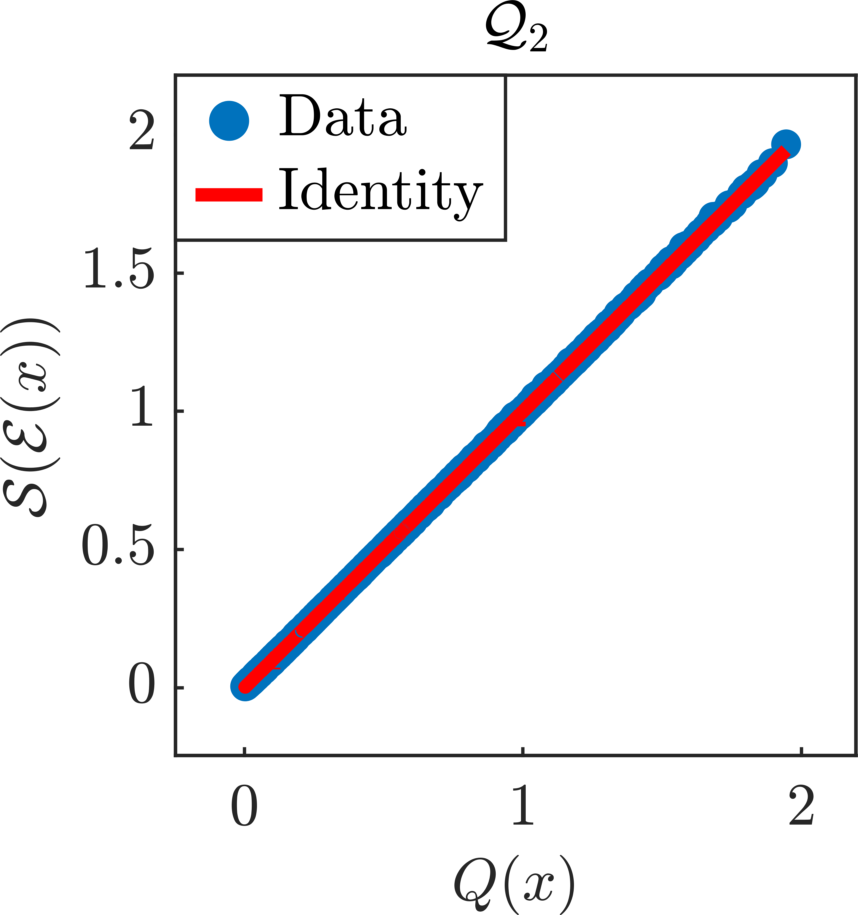} &
\includegraphics{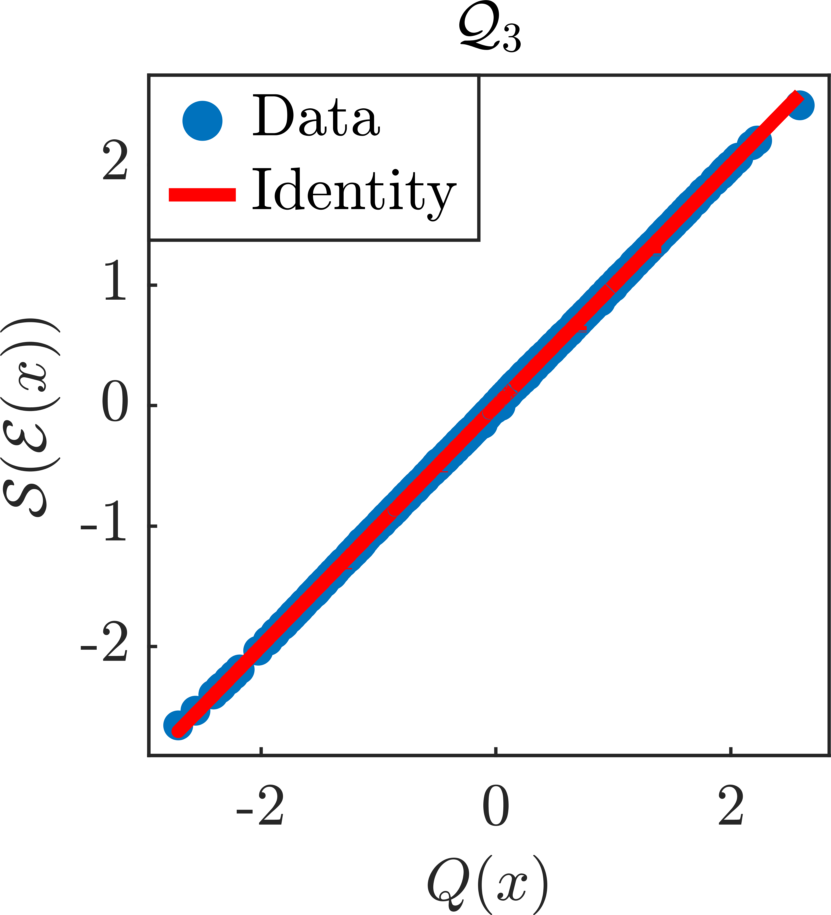}
\end{tabular}
\caption{Results for the two-dimensional models $\Q_1, \Q_2, \Q_3$, for one realization of the NeurAM. Top: projection of input samples on the one-dimensional NeurAM. Bottom: parity plot between the exact and surrogate models.}
\label{fig:paperAM_manifold}
\end{figure}

\begin{figure}[ht!]
\small
\centering
\begin{tabular}{ccc}
\includegraphics{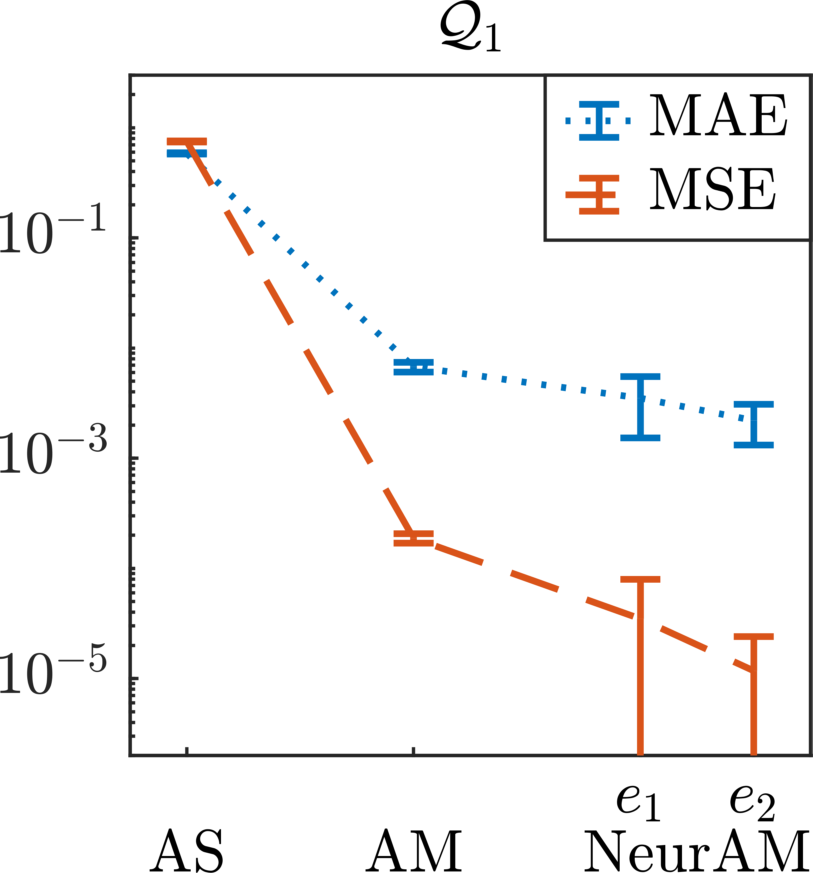} &
\includegraphics{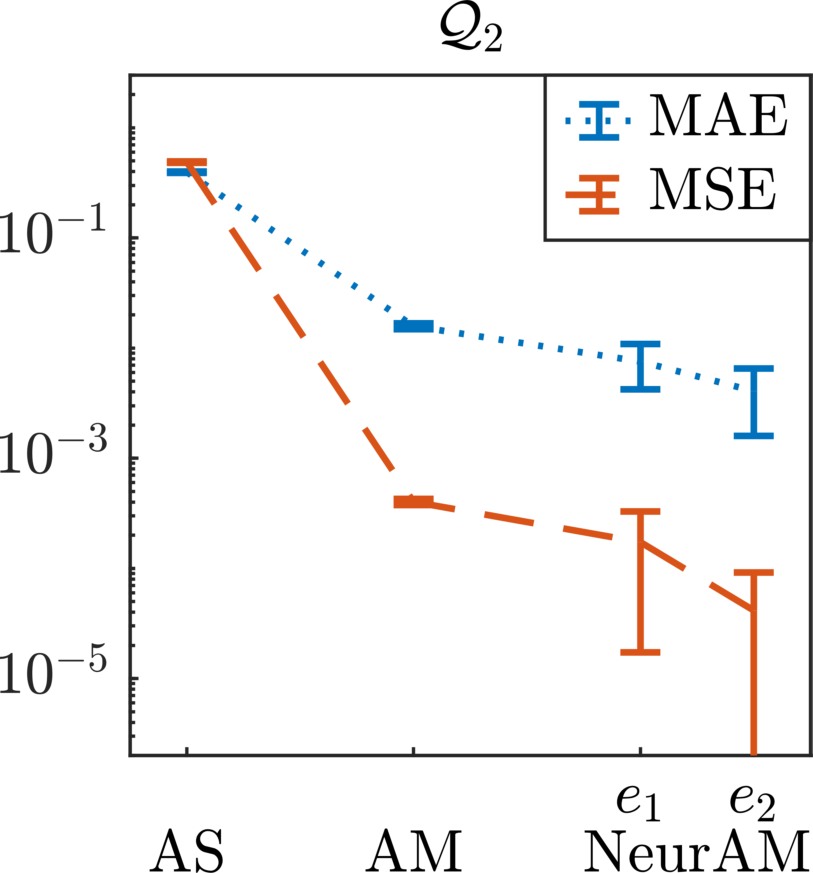} &
\includegraphics{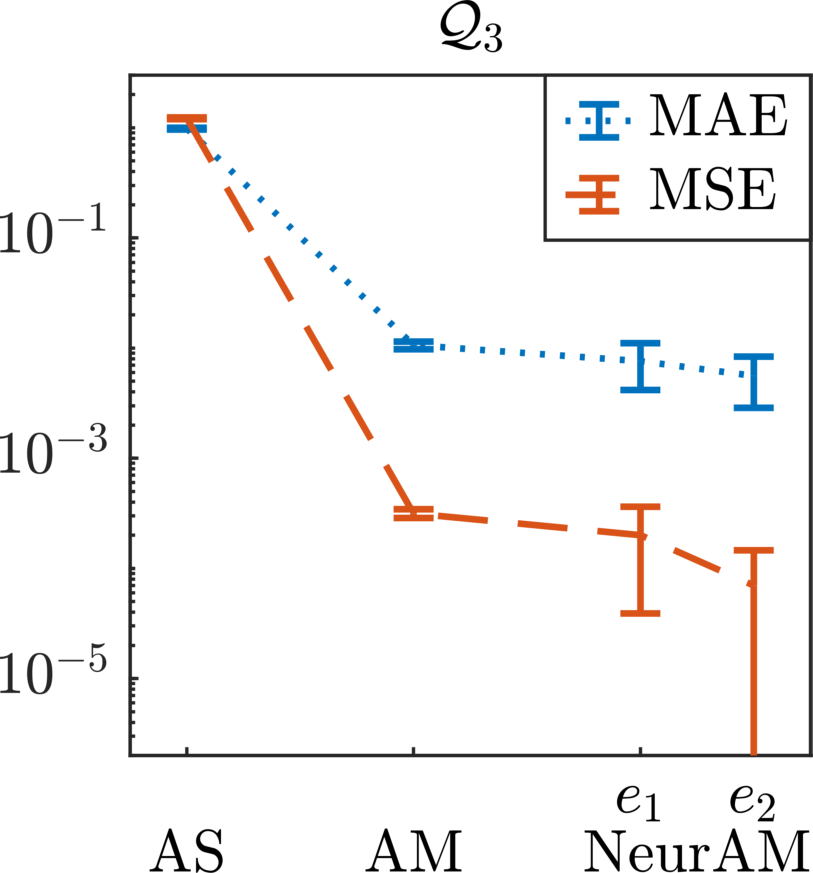} \\[0.5cm]
\midrule
$\Q_1$ & $\Q_2$ & $\Q_3$ \\[0.125cm]
\midrule
\begin{tabular}{ccc}
AS & AM & NeurAM \\
\midrule
100 \% & 86.7 \% & 100 \%
\end{tabular} &
\begin{tabular}{ccc}
AS & AM & NeurAM \\
\midrule
100 \% & 77 \% & 100 \%
\end{tabular} &
\begin{tabular}{ccc}
AS & AM & NeurAM \\
\midrule
100 \% & 92.9 \% & 100 \%
\end{tabular} \\
\bottomrule
\end{tabular}
\caption{Comparison of approximation errors (MAE and MSE) for AS, AM, and NeurAM, applied to the two-dimensional models $\Q_1, \Q_2, \Q_3$. Top: plot of the approximation errors, where the values for AS and AM are reported from \cite[Table 1]{BGF19}, and the errors $e_1$ and $e_2$ for NeurAM are defined in equation \eqref{eq:errors_def}. These results consider 8000 training samples for AS and AM, and 1000 samples for NeurAM.
Bottom: percentage of test points for which the algorithm successfully found an approximation.}
\label{fig:paperAM_comparison}
\end{figure}

\subsubsection{NeurAM for multifidelity uncertainty propagation} \label{sec:numerical_MF}

In this section we show another advantage of NeurAM, the ability to enhance the performance of multifidelity Monte Carlo estimators, without any additional modification or computational cost.
Specifically, we consider the following high-fidelity and low-fidelity models
\begin{equation}
\begin{aligned}
\Q_\HF(x) &= e^{0.7x_1 + 0.3x_2} + 0.15\sin(2\pi x_1), \\
\Q_\LF(x) &= e^{0.01x_1 + 0.99x_2} + 0.15\sin(3\pi x_2),
\end{aligned}
\end{equation}
proposed in~\cite{GEG18} to test the application of AS to multifidelity UQ, and then further employed in~\cite{ZGS24} in the context of nonlinear dimensionality reduction.
We assume $x\sim\mu = \mathcal U([-1,1]^2)$, and focus on the estimation of the quantity of interest $q$ expressed as
\begin{equation}
q = \Ex^\mu[\Q_\HF(X)] = \frac{25}{21} \left( e^{-1} - e^{-\frac25} - e^{\frac25} + e\right).
\end{equation}
We further assume a computational budget $\mathcal B = 1000$ and a cost of our low-fidelity model equal to $\mathcal C_\LF = 0.01 \mathcal C_\HF$ ($w = 0.01$) to mirror cost differences in realistic applications.

We compute NeurAM for both the high-fidelity and low-fidelity models separately using $N = 1000$ training samples, and present the results in \cref{fig:paperAS_manifold}, which shows that the variability for both models is entirely captured.
We remark that, even if we might expect a vertical NeurAM, as $\Q_\LF(x)$ only varies along the vertical direction, any curve going from the bottom (minimum value) to the top (maximum value) of the domain is a valid NeurAM.
What is crucial is how points outside the manifold are projected on it, which occurs horizontally, i.e., along the direction where the model response is constant. 
This indeed suggests the correctness of the results for this case.

We then move to the problem of estimating the quantity of interest $q$ using a multifidelity Monte Carlo estimator. 
We compare multifidelity Monte Carlo from~\eqref{eq:MFMC} with our modified estimator, where $\widetilde\Q_\LF$ is used instead of $\Q_\LF$, and computed applying NeurAM to map high-fidelity to low-fidelity model inputs.
We also include a single fidelity Monte Carlo estimator in our comparison as a reference.
\cref{fig:paperAS_MFMC} shows the results obtained from $100$ repeated evaluations of the three estimators.
We observe that the original models $\Q_\HF$ and $\Q_\LF$ are badly correlated, leading to a poor performance of the multifidelity Monte Carlo estimator, which indeed does not improve over single-fidelity Monte Carlo. 
Conversely, NeurAM significantly increases the correlation between the models, leading to a multifidelity estimator with reduced variance. 
Finally, the box plot in \cref{fig:paperAS_MFMC} also includes the ideal correlation from equation~\eqref{eq:ideal_corr}, which seems to act as an upper bound for the correlation achievable through NeurAM.

\begin{figure}[ht!]
\centering
\begin{tabular}{ccc}
\includegraphics{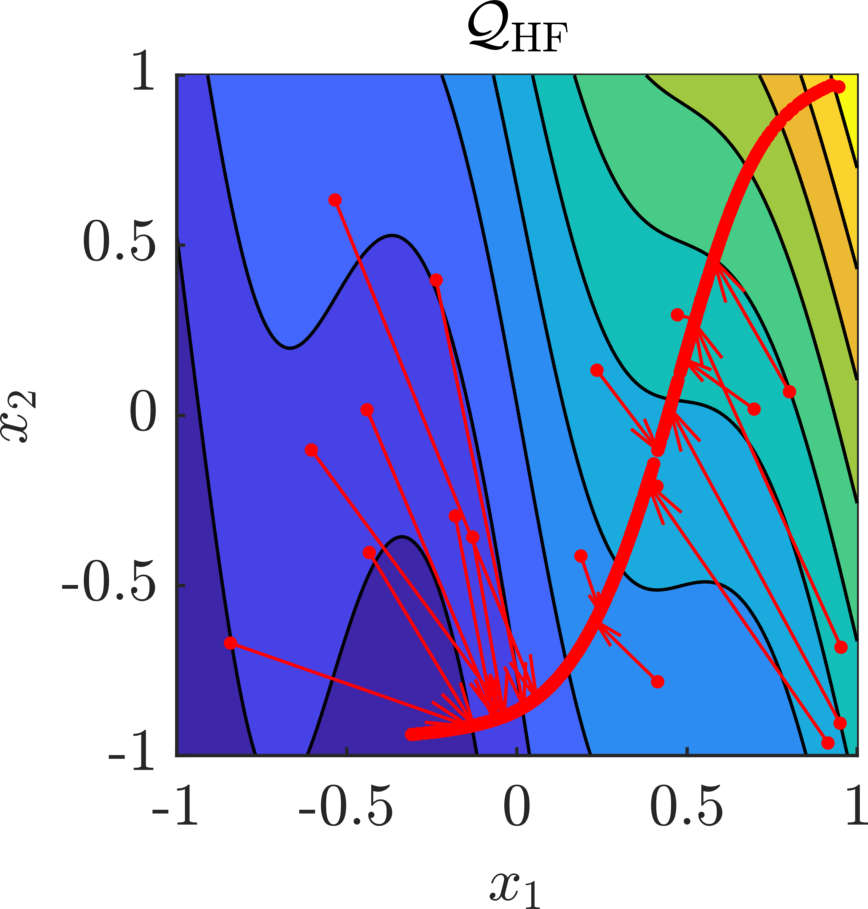} &$\qquad$&
\includegraphics{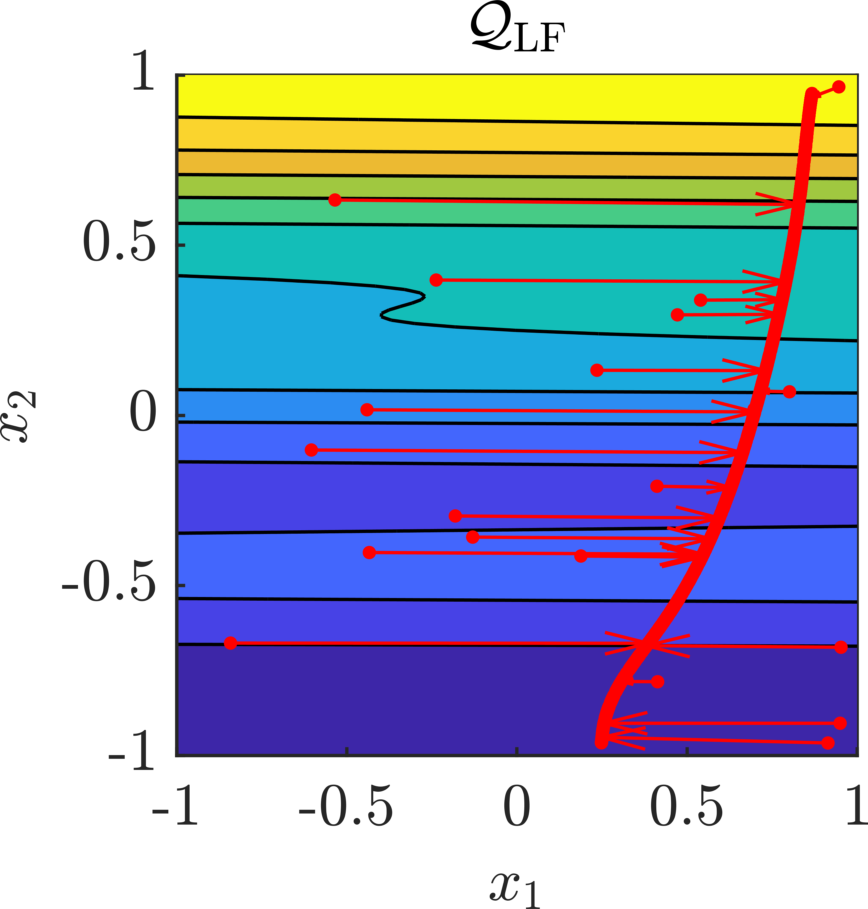} \\
\includegraphics{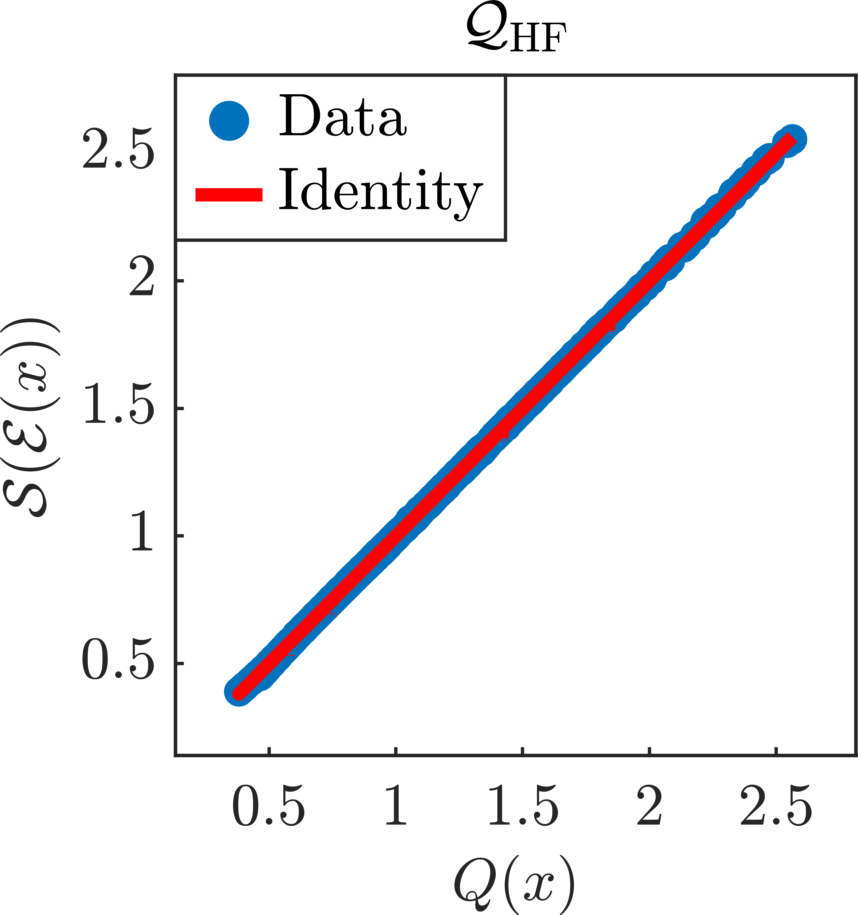} &$\qquad$& 
\includegraphics{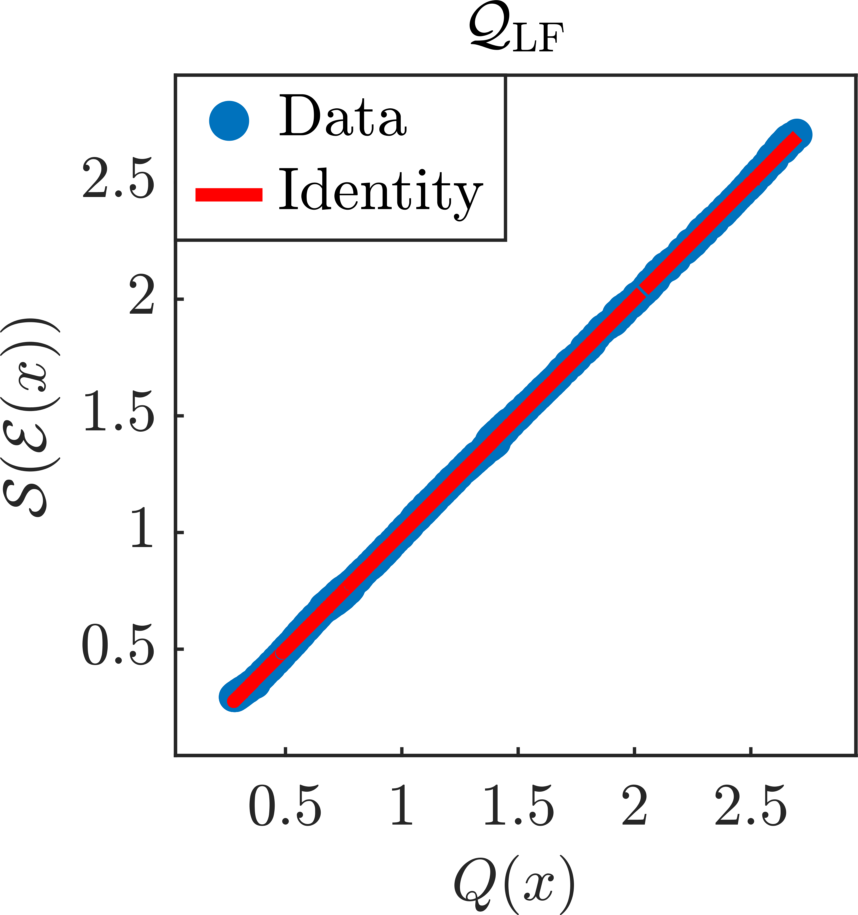}
\end{tabular}
\caption{Results for the two-dimensional models $\Q_\HF, \Q_\LF$, for one realization of the NeurAM. Top: projection of input samples on the one-dimensional NeurAM. Bottom: parity plot between the exact and surrogate models.}
\label{fig:paperAS_manifold}
\end{figure}

\begin{figure}[ht!]
\centering
\includegraphics{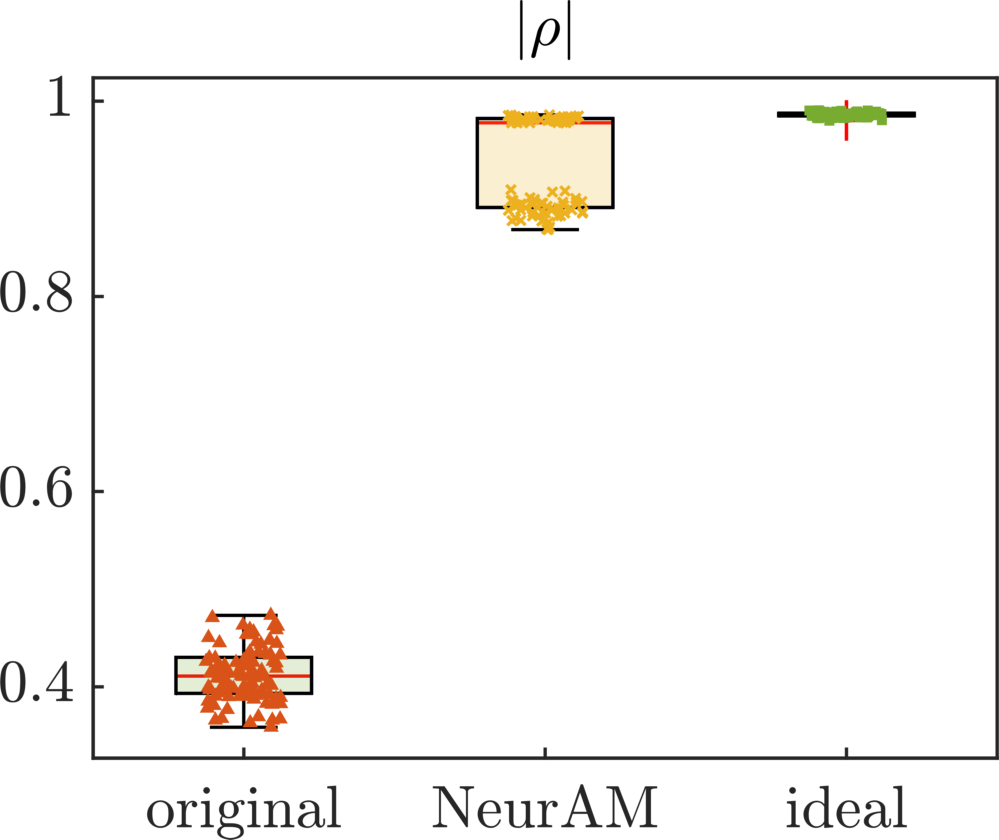} \hspace{1.5cm}
\includegraphics{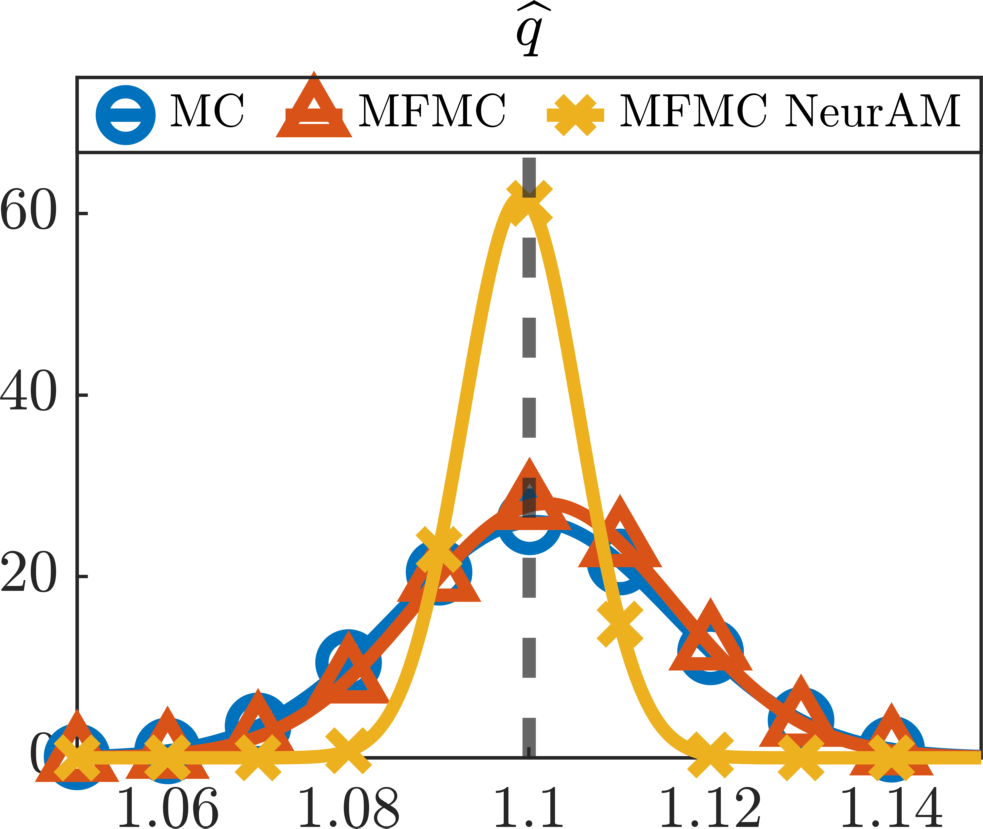}
\caption{Numerical results for the multifidelity uncertainty propagation. Left: Pearson correlation coefficients. Right: approximated distributions from single-fidelity and multifidelity Monte Carlo estimators. The dashed line indicates the exact value of the quantity of interest $q$.}
\label{fig:paperAS_MFMC}
\end{figure}

\subsubsection{Sensitivity analysis for the two-dimensional models} 

Consider the model $Q$ with the parabolic NeurAM in \cref{sec:parabola}, for which we can compute the sensitivity indices analytically. 
Substituting the expressions in equation~\eqref{eq:parabola_NeurAM} into the local indices~\eqref{eq:local_SA} we obtain 
\begin{equation}
\theta_1(t) = \frac{2t}{2t+1} \qquad \text{and} \qquad \theta_2(t) = \frac{1}{2t+1},
\end{equation}
for $t \in [0,2]$ which implies that the first input is more important than the second if $t > 1/2$, while the opposite is true for $t \le 1/2$. 
Moreover, computing the integrals in equation~\eqref{eq:global_SA} leads to the global indices
\begin{equation}
\begin{aligned}
\Theta_1 &= \frac{\int_0^2 \frac{2t}{2t+1} \sqrt{\frac{2t+1}{8t}} \dd t}{\int_0^2 \sqrt{\frac{2t+1}{8t}} \dd t} = \frac{\int_0^2 \sqrt{\frac{4t}{2t+1}} \dd t}{\int_0^2 \sqrt{\frac{2t+1}{t}} \dd t} = \frac{4\sqrt5 - 2\log(2 + \sqrt5)}{4\sqrt5 + 2\log(2 + \sqrt5)} \simeq 0.512, \\
\Theta_2 &= \frac{\int_0^2 \frac1{2t+1} \sqrt{\frac{2t+1}{8t}} \dd t}{\int_0^2 \sqrt{\frac{2t+1}{8t}} \dd t} = \frac{\int_0^2 \sqrt{\frac1{t(2t+1)}} \dd t}{\int_0^2 \sqrt{\frac{2t+1}{t}} \dd t} = \frac{4\log(2 + \sqrt5)}{4\sqrt5 + 2\log(2 + \sqrt5)} \simeq 0.488.
\end{aligned}
\end{equation}

We then move to the models $\Q_1,\Q_2,\Q_3$, and $\Q_\HF,\Q_\LF$ in \cref{sec:numerical_MF,sec:numerical_comparison}, respectively. 
Both local and global indices are reported in~\cref{fig:SA2D}, where we also compare the results with the first order Sobol' sensitivity indices, even if we do not expect perfect agreement since our sensitivity indices are based on derivatives, while Sobol' indices are based on an analysis of variance.
For the last $\Q_\LF$ model, where the only significant variable is $x_2$ as shown in~\cref{fig:paperAS_manifold}, the indices are constantly equal to $0$ and $1$ for $x_1$ and $x_2$, respectively, as expected.
For the second model $\Q_2$, where one would expect equal importance for both the variables, we notice that our local and global indices slightly diverge from the value $0.5$. This is due to the fact that the indices are dependent on the single realization of the NeurAM. Nevertheless, for all the realizations, the values oscillate around $0.5$, implying that both the variables are important. Moreover, computing the average global indices over $100$ runs, we obtain $\bar\Theta_1 = 0.499$ and $\bar\Theta_2 = 0.501$, as expected.
For the other models, the importance of the input variables as computed from NeurAM and Sobol' indices is consistent, with the only exception of $\Q_1$ for which the interaction between the two components is more relevant. In particular, we observe a better agreement between the two types of indices for $\Q_2,\Q_3,\Q_\LF$, where the interaction between the variables is limited.
Our indices are therefore more suitable to correctly rank the variables in the presence of a negligible interaction.
Extensions of NeurAM to quantify sensitivity accounting for interaction among inputs are left to future work.

\begin{figure}[ht!]
\centering
\includegraphics{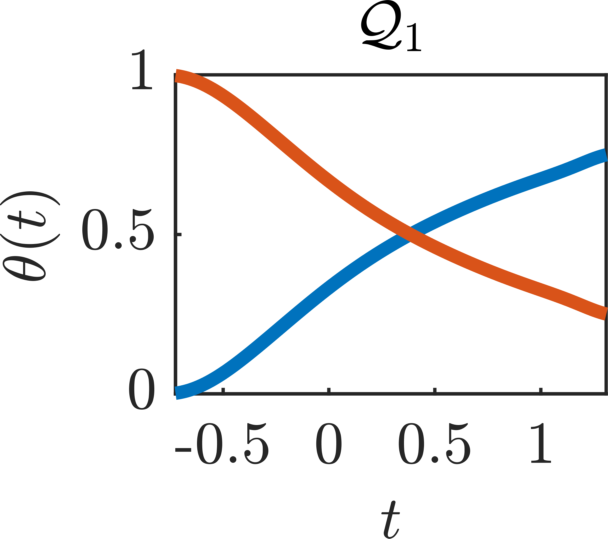}
\includegraphics{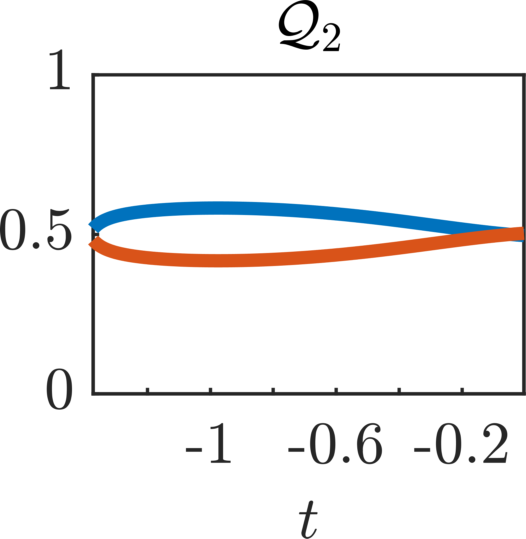}
\includegraphics{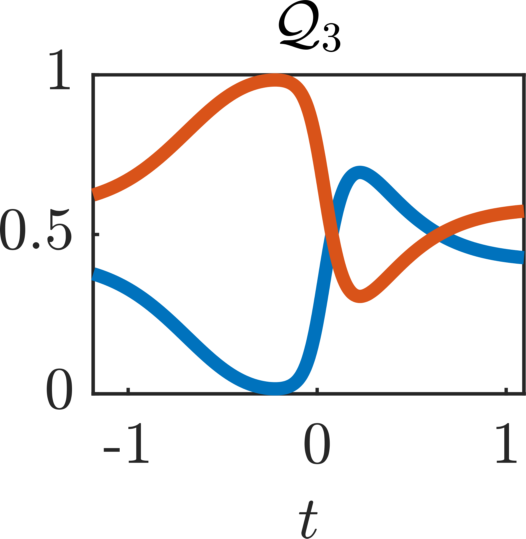}
\includegraphics{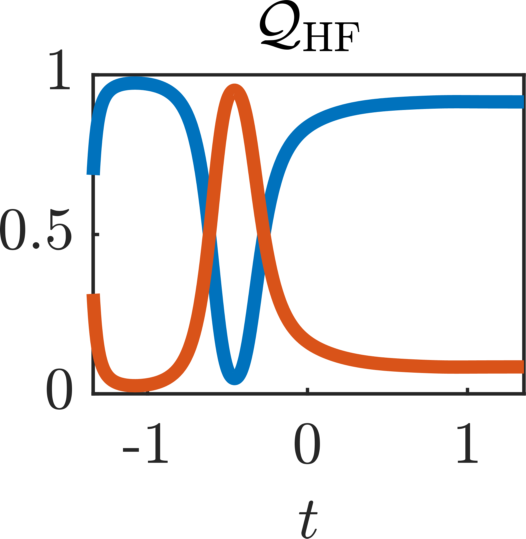}
\includegraphics{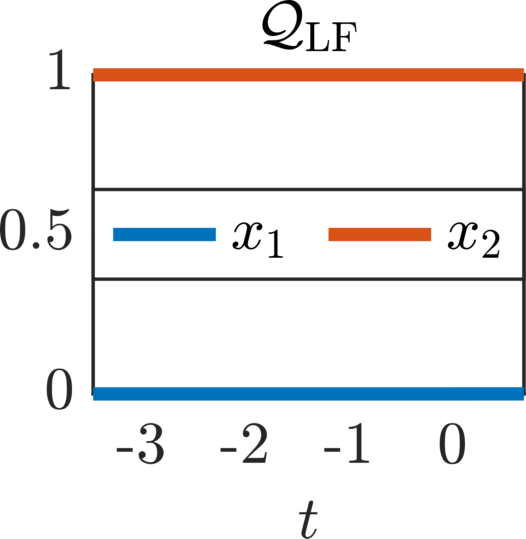}
\begin{center}
\small
\begin{tabular}{cccccc}
\toprule
& $\Q_1$ & $\Q_2$ & $\Q_3$ & $\Q_\HF$ & $\Q_\LF$ \\
\midrule
\begin{tabular}{c}
\\
\midrule
$\Theta$ \\
$S$
\end{tabular} 
& 
\begin{tabular}{cc}
$x_1$ & $x_2$ \\
\midrule
0.555 & 0.445 \\
0.176 & 0.762
\end{tabular} 
& 
\begin{tabular}{cc}
$x_1$ & $x_2$ \\
\midrule
0.558 & 0.442 \\
0.500 & 0.500
\end{tabular} 
& 
\begin{tabular}{cc}
$x_1$ & $x_2$ \\
\midrule
0.329 & 0.671 \\
0.320 & 0.680
\end{tabular} 
& 
\begin{tabular}{cc}
$x_1$ & $x_2$ \\
\midrule
0.625 & 0.375 \\
0.801 & 0.173
\end{tabular} 
& 
\begin{tabular}{cc}
$x_1$ & $x_2$ \\
\midrule
0 & 1 \\
0 & 1
\end{tabular} 
\\
\bottomrule
\end{tabular}
\end{center}
\caption{Sensitivity analysis for one realization of the NeurAM for the two-dimensional models $\Q_1, \Q_2, \Q_3, \Q_\HF, \Q_\LF$. Top: local indices $\theta(t)$ for both $x_1$ (blue) and $x_2$ (red). Bottom: global indices $\Theta$ compared with Sobol' indices $S$.}
\label{fig:SA2D}
\end{figure}

\subsection{Hartmann problem}

In this section we consider a more complex example, the Hartmann problem, which is a verification test case for magnetohydrodynamics applications. The Hartmann problem models a laminar flow of a conducting fluid between two parallel plates, while a magnetic field is applied in the transverse direction~\cite{SPC16}.  
This problems admits an analytical solution and has been used previously to test AS and AM in~\cite{GCS17,BGF19}. 
In particular, the two quantities of interest, i.e., the average flow velocity $u$ and the induced magnetic field $B$, are given by the following expression~\cite{GCS17,BGF19}
\begin{equation}
\begin{dcases}
u(\mu,\rho,\partial p_0 / \partial x,\eta, B_0) = - \frac{\partial p_0}{\partial x} \frac{\eta}{B_0^2} \left[ 1 - \frac{B_0 l}{\sqrt{\eta\mu}} \coth \left( \frac{B_0 l}{\sqrt{\eta \mu}} \right) \right], \\
B(\mu,\rho,\partial p_0 / \partial x,\eta,B_0) = \frac{\partial p_0}{\partial x}  \frac{l \mu_0}{2 B_0} \left[ 1 - 2 \frac{\sqrt{\eta\mu}}{B_0 l} \tanh \left( \frac{B_0 l}{2 \sqrt{\eta\mu}} \right) \right],
\end{dcases}
\end{equation}
where the magnetic constant $\mu_0 = 1$ and the length $l = 1$ are fixed, while the remaining parameters are log-uniformly distributed in the ranges given in~\cref{tab:parameters_Hartmann}.
Note that the fluid density $\rho$ does not actually appear in the equations for $u$ and $B$. 
We first normalize the input parameters in $[-1,1]^5$, and run our algorithm using $N = 1000$ training samples. 
Since we cannot visualize the NeurAM for a five-dimensional problem, \cref{fig:Hartmann_manifold} only shows a parity plot for both the average flow velocity $u$ and the induced magnetic field $B$.
We observe that the model variability is almost completely captured by NeurAM, but, unlike the two-dimensional models in the previous sections, a few points are observed to deviate from the identity.

We then compare the performance of our methodology with the numerical results for AS and AM reported in~\cite[Table 3]{BGF19}.
We run NeurAM $100$ times, and then compute the MAE and MSE of the errors $e_1$ and $e_2$ in~\eqref{eq:errors_def} at each iteration for both $u$ and $B$.
The comparative plot in \cref{fig:Hartmann_comparison} shows an improved performance of NeurAM for both QoIs, particularly the induced magnetic field $B$.
This can be justified in light of the more complex dependence between $B$ and the five inputs, whereas linear dimensionality reduction seems sufficient for $u$ (see, e.g., the sensitivity indices in \cref{fig:Hartmann_SA}).

We also use the optimally trained NeurAM for multifidelity propagation and sensitivity analysis. In this context, we assume the induced magnetic field $B$ and the average flow velocity $u$ to represent the low-fidelity and high-fidelity model response, respectively, and focus on computing the expectation $\bar u = \Ex^\nu[u]$, where $\nu$ denotes the joint distribution of the input parameters given in \cref{tab:parameters_Hartmann}. 
In order to mimic cost differences in realistic applications, we let $w = 0.01$ be the cost ratio $\mathcal C_\LF/\mathcal C_\HF$ between the two fidelities. 
We then compute $100$ estimates of $\bar u$ with a computational budget $\mathcal B = 1000$, using single-fidelity Monte Carlo and multifidelity Monte Carlo with the original low-fidelity model or using NeurAM to determine new sample locations according to~\eqref{eq:modified_QLF}.
In \cref{fig:Hartmann_MFMC} we show that the new low-fidelity sampling locations computed by NeurAM improve the correlation between $B$ and $u$, and in several cases reach the ideal correlation given in \cref{rem:ideal}. As a consequence, the resulting approximated distribution of the multifidelity Monte Carlo estimator with NeurAM has a significantly smaller variance than standard multifidelity Monte Carlo.

Finally, in \cref{fig:Hartmann_SA} we show that NeurAM can be used for sensitivity analysis. In particular, we plot the quantities $\theta_i$, $i=1,\dots,5$, given in equation \eqref{eq:local_SA} for each input. 
The input ranking varies along the manifold for $B$, while it remains approximately constant for $u$. For example, the dependence of $B$ on $B_0$ increases along the manifold, implying a smaller contribution of $\partial p_0 / \partial x$, which is the most important input parameter. 
On the other hand, we notice that the second most relevant parameter for $u$ is $\mu$, while the remaining inputs give a negligible contribution. 
Moreover, in both cases we observe that the index $\theta(t)$ for $\rho$ is constantly zero as expected, since it does not appear in the equations for $u$ and $B$. In \cref{fig:Hartmann_SA} we finally compare the global indices $\Theta_i$ with the first order Sobol' sensitivity indices $S_i$ computed evaluating the exact models over $2^{10}$ samples, showing agreement between the two rankings.

\begin{table}[ht!]
\begin{center}
\begin{tabular}{ccccccc}
\toprule
Parameter && Notation && Range && Unit \\
\midrule
Fluid viscosity && $\mu$ && $[0.05, 0.2]$ && $\mathrm{kg}$ $\mathrm{m}^{-1}$ $\mathrm{s}^{-1}$ \\
Fluid density && $\rho$ && $[1, 5]$ && $\mathrm{kg}$ $\mathrm{m}^{-3}$ \\
Applied pressure gradient && $\frac{\partial p_0}{\partial x}$ && $[0.5, 3]$ && $\mathrm{kg}$ $\mathrm{m}^{-2}$ $\mathrm{s}^{-2}$ \\
Resistivity && $\eta$ && $[0.5, 3]$ && $\mathrm{kg}$ $\mathrm{m}^3$ $\mathrm{s}^{-3}$ $\mathrm{A}^{-2}$ \\
Applied magnetic field && $B_0$ && $[0.1, 1]$ && $\mathrm{kg}$ $\mathrm{s}^{-2}$ $\mathrm{A}^{-1}$ \\
\bottomrule
\end{tabular}
\end{center}
\caption{Parameter space for the Hartmann problem.}
\label{tab:parameters_Hartmann}
\end{table}

\begin{figure}[ht!]
\centering
\includegraphics{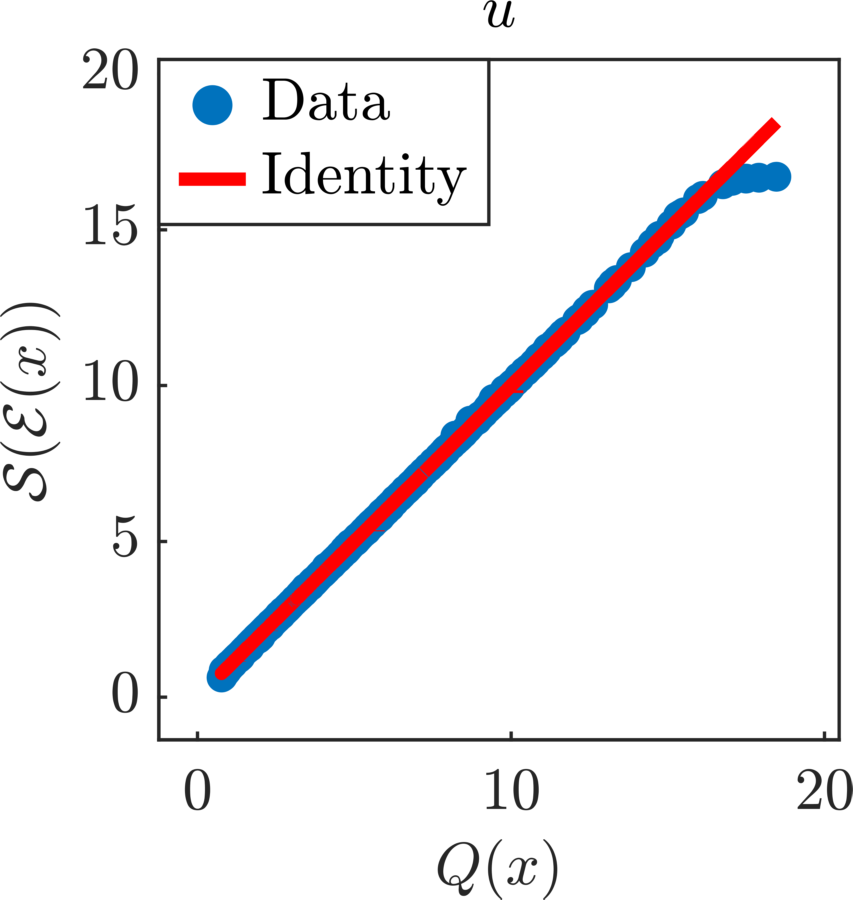} \hspace{2cm}
\includegraphics{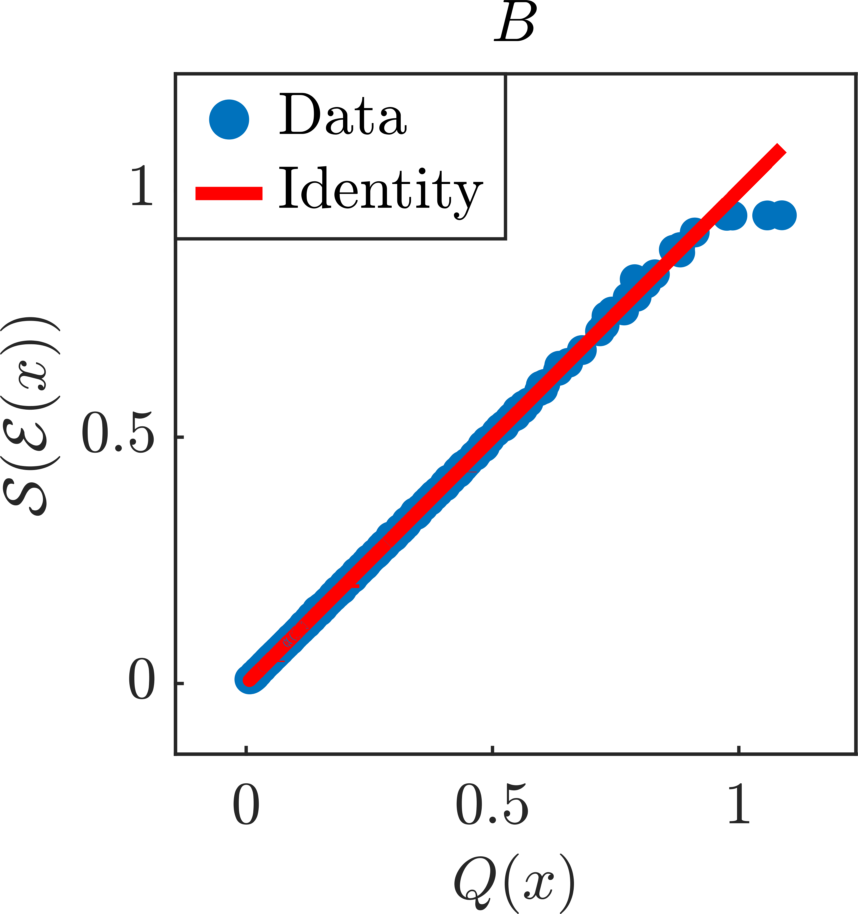}
\caption{Parity plot between the exact and surrogate models, for one realization of the NeurAM for the Hartmann problem. Left: average flow velocity $u$. Right: induced magnetic field $B$.}
\label{fig:Hartmann_manifold}
\end{figure}

\begin{figure}[ht!]
\centering
\includegraphics{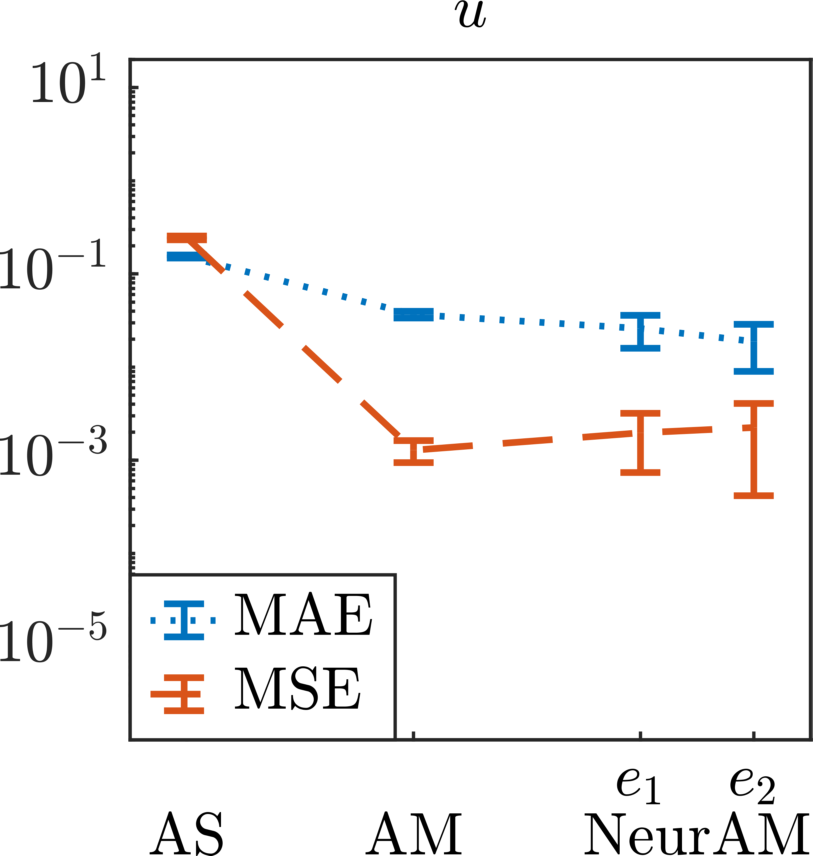} \hspace{2cm}
\includegraphics{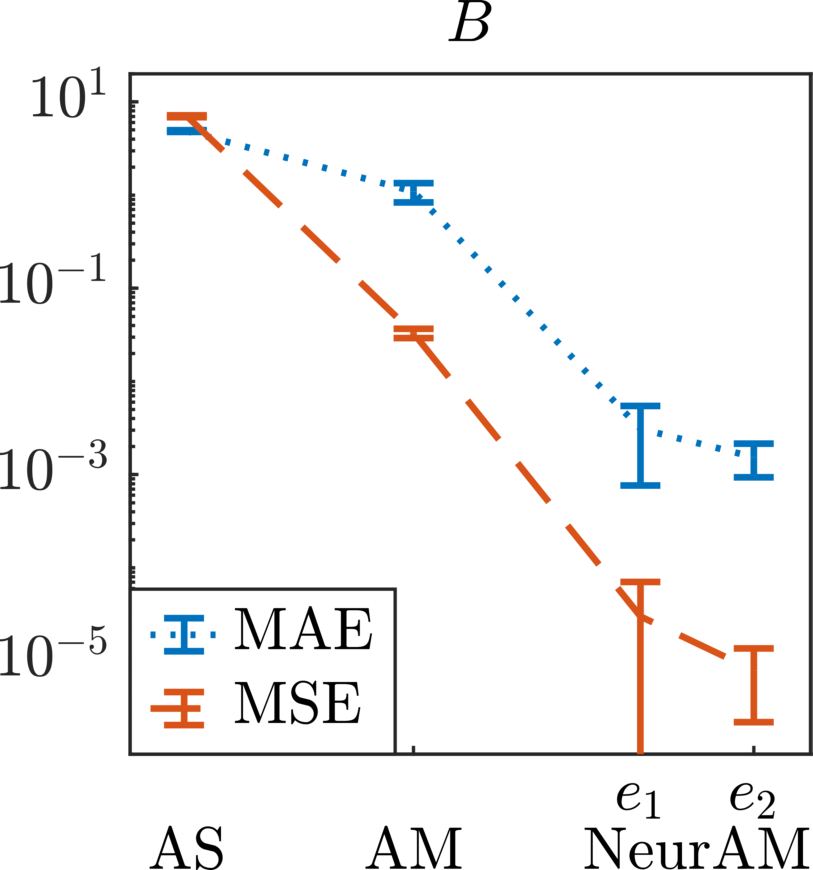} 
\caption{Comparison of approximation errors (MAE and MSE) between AS, AM and NeurAM for the Hartmann problem. The values for AS and AM are reported from \cite[Table 3]{BGF19}, and the errors $e_1$ and $e_2$ for NeurAM are defined in equation~\eqref{eq:errors_def}. Left: average flow velocity $u$. Right: induced magnetic field $B$.
}
\label{fig:Hartmann_comparison}
\end{figure}

\begin{figure}[ht!]
\centering
\includegraphics{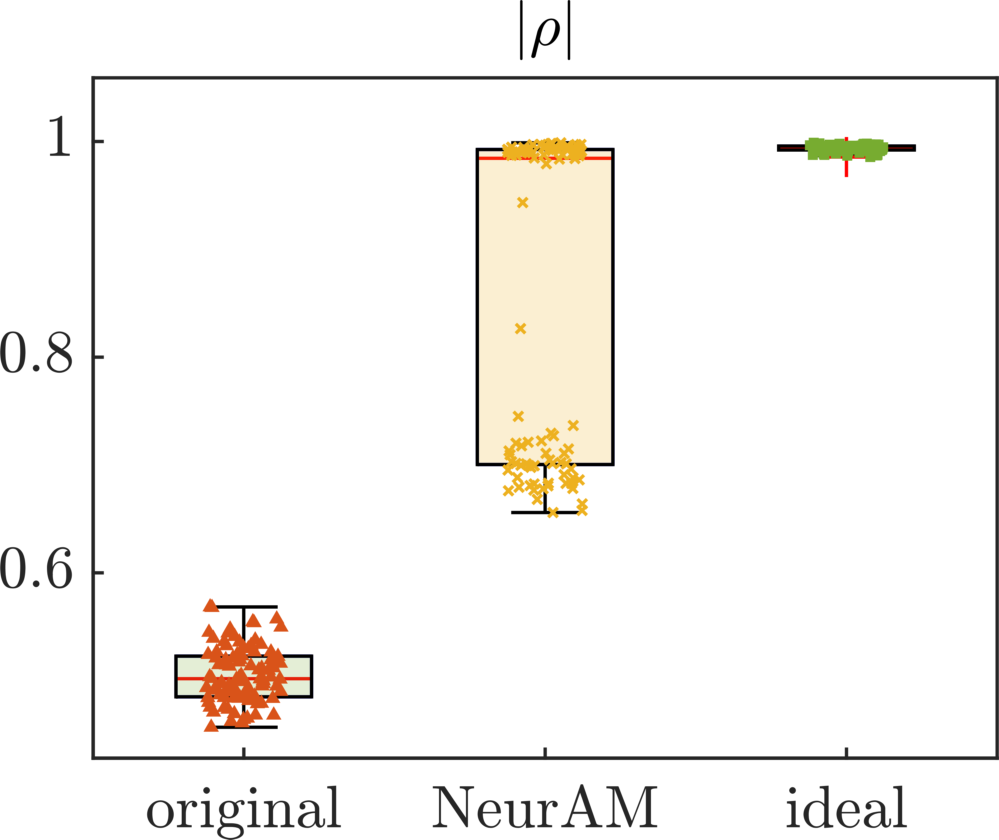} \hspace{1.5cm}
\includegraphics{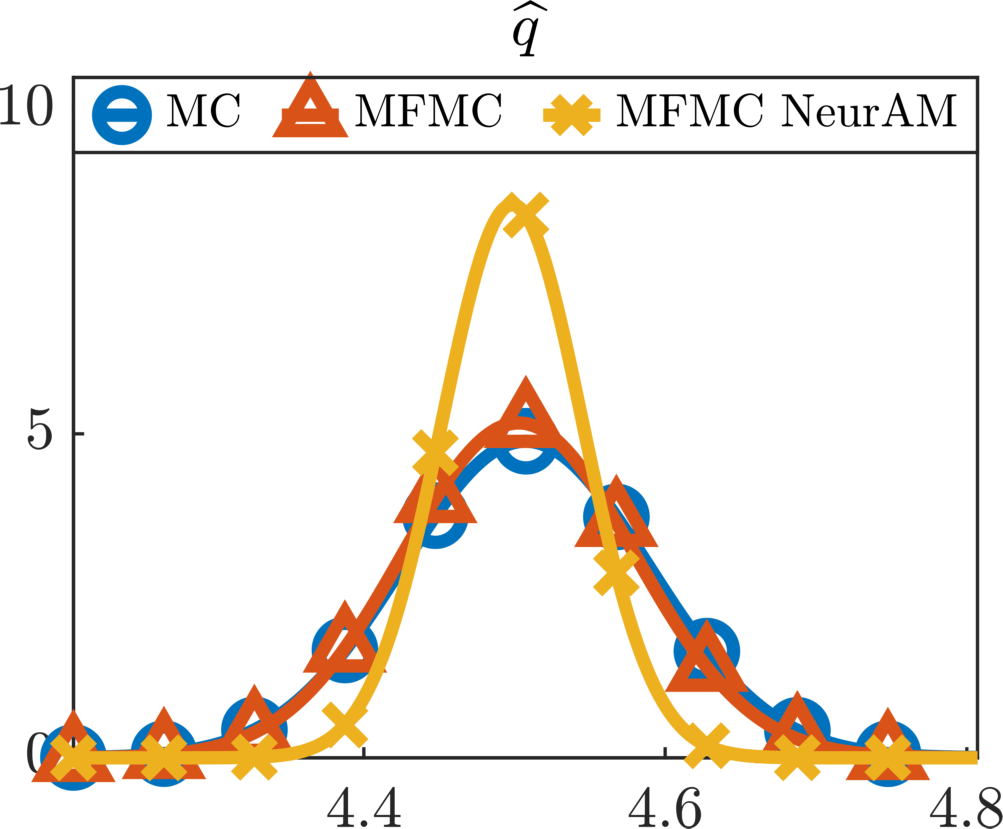}
\caption{Multifidelity uncertainty propagation for the Hartmann problem. Left: Pearson correlation coefficients. Right: approximated distributions for the single fidelity and multifidelity Monte Carlo estimators.}
\label{fig:Hartmann_MFMC}
\end{figure}

\begin{figure}[ht!]
\centering
\includegraphics{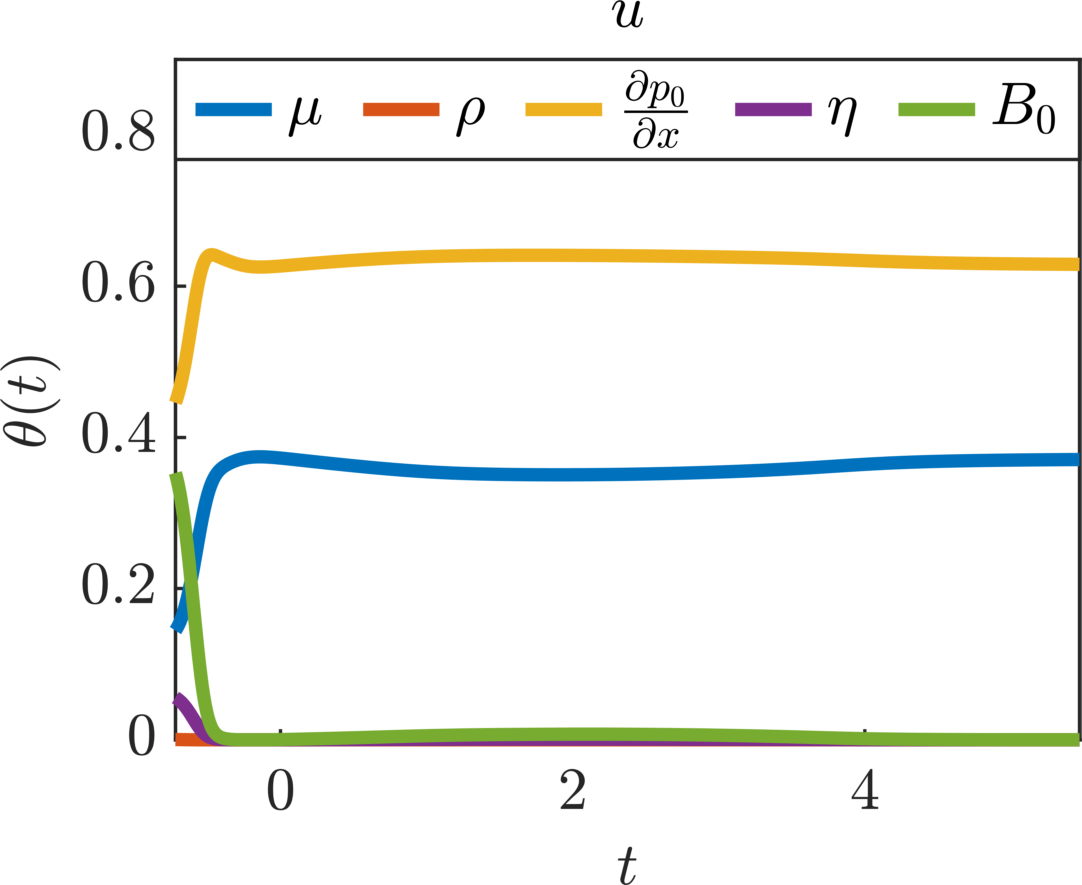} \hspace{1.5cm}
\includegraphics{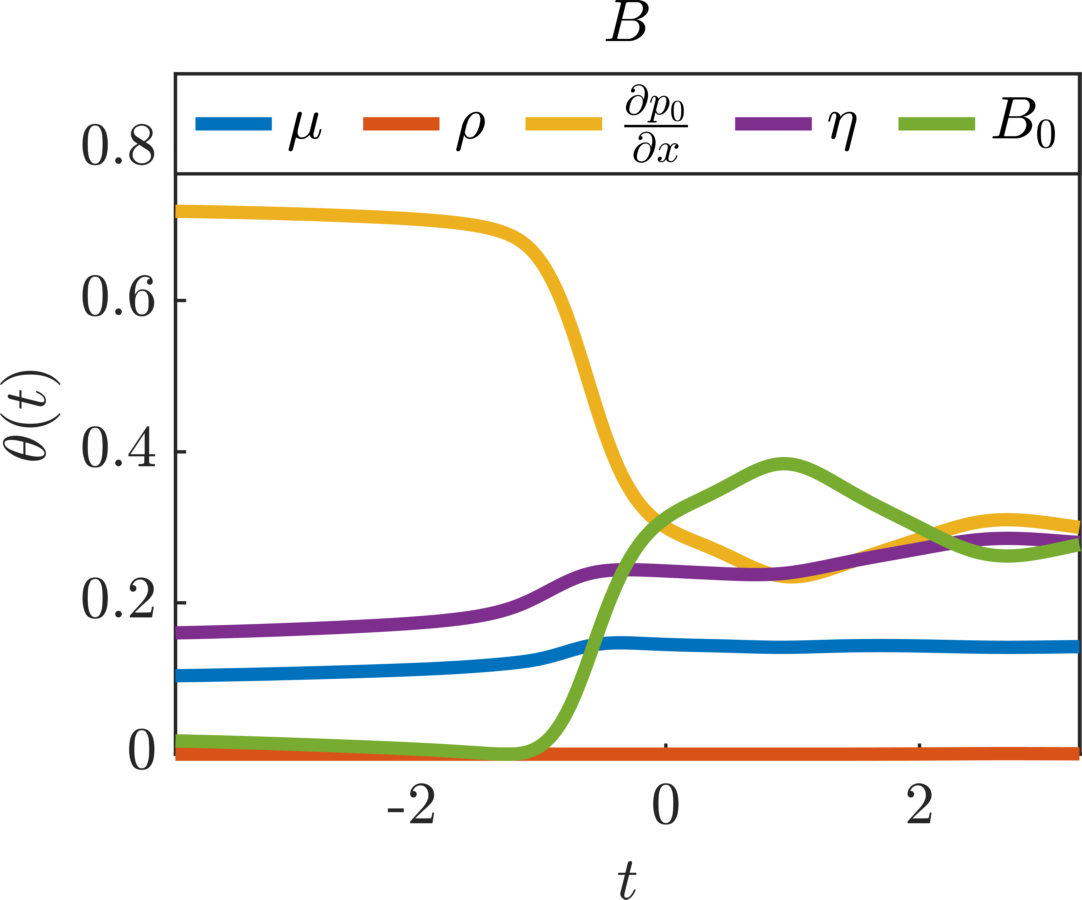}
\begin{center}
\small
\begin{tabular}{cc}
\toprule
$u$ & $B$ \\
\midrule
\begin{tabular}{cccccc}
& $\mu$ & $\rho$ & $\partial p_0 / \partial x$ & $\eta$ & $B_0$ \\
\midrule
$\Theta$ & $0.332$ & $2\mathrm{e-}5$ & $0.619$ & $0.007$ & $0.043$ \\
$S$ & $0.299$ & $0$ & $0.572$ & $0.001$ & $0.025$
\end{tabular} 
&
\begin{tabular}{cccccc}
& $\mu$ & $\rho$ & $\partial p_0 / \partial x$ & $\eta$ & $B_0$ \\
\midrule
$\Theta$ & $0.140$ & $4\mathrm{e-}5$ & $0.372$ & $0.238$ & $0.250$ \\
$S$ & $0.116$ & $0$ & $0.299$ & $0.199$ & $0.212$
\end{tabular} \\
\bottomrule
\end{tabular}
\end{center}
\caption{Sensitivity analysis for one realization of the NeurAM for the Hartmann problem. Left: average flow velocity $u$. Right: induced magnetic field $B$. Top: local indices $\theta(t)$ from NeurAM. Bottom: global NeurAM ($\Theta$) and Sobol' ($S$) sensitivity indices.}
\label{fig:Hartmann_SA}
\end{figure}

\subsection{Cardiac electrophysiology}

In this section we consider a more complex model which cannot be expressed as an analytic function of the input parameters. In particular, we model cardiac electrophysiology in a two-dimensional square $\Omega = [0, L]^2$ of length $L = 250$ mm representing a slab of cardiac tissue using the biophysically detailed monodomain equation \cite{Qua19,CPS14}, coupled with the ten Tusscher-Panfilov ionic model \cite{TEP06}. We formulate the model as follows 
\begin{eqnarray} \label{eq:cardiac_model}
\left\{\begin{array}{ll}
\displaystyle
\frac{\partial \Pot}{\partial t}+\Iion(\Pot,\Ionic)
-\nabla\cdot(\Di I \nabla \Pot)=\Iapp(\omega,\tau; \tboxstim) & \mbox{ in }\Omega\times(0,T],\\[2mm]
(\DiffTens\nabla \Pot)\cdot {n}=0  & \mbox{ on }\partial\Omega\times(0,T],\\[2mm]
\displaystyle
\frac{d\Ionic}{dt}=\RhsIonic(\Pot,\Ionic; \GNa, \GCaL, \GKr, \GKs)
& \mbox{ in }\Omega\times(0,T],\\[2mm]
\displaystyle
\Pot(\omega,0)=\Pot_0(\omega),\
\Ionic(\omega,0)=\Ionic_0(\omega) &  \mbox{ in }\Omega.
\end{array}\right.
\end{eqnarray}
The transmembrane potential $\Pot$ describes the propagation of the electric signal over the two-dimensional slab of cardiac tissue, the vector $\Ionic=(y_1,\ldots,y_{M+P})$ defines the probability density functions of $M=12$ gating variables, which represent the fraction of open channels across the membrane of a single cardiomyocyte, and the concentrations of $P=6$ relevant ionic species. Among them, sodium $Na^{+}$, intracellular calcium $Ca^{2+}$ and potassium $K^{+}$ play an important role in the physiological processes \cite{BGR15} dictating heart rhythmicity and sarcomere contractility \cite{BFR09}. The right hand side $\RhsIonic(\Pot,\Ionic)$ defines the dynamics of the gating and concentration variables. The ionic current $\Iion(\Pot,\Ionic)$ models the impact of the system of ordinary differential equations at the cellular, microscopic scale on the tissue, macroscopic scale coming for the first partial differential equation. The analytical expressions of both $\RhsIonic(\Pot,\Ionic)$ and $\Iion(\Pot,\Ionic)$ derive from the mathematical formulation of the ten Tusscher-Panfilov ionic model \cite{TEP06}. The right-hand side $\RhsIonic$ depends also on the maximal $Ca^{2+}$ and $Na^{+}$ current conductances $\GCaL$ and $\GNa$, and the maximal rapid and slow delayed rectifier current conductances $\GKr$ and $\GKs$. The diffusion tensor $\Di I$ defines an isotropic propagation of the electric signal driven by the homogeneous isotropic conductivity $\Di$. We impose the condition of an electrically isolated domain by prescribing homogeneous Neumann boundary conditions $\partial\Omega$, where $n$ is the outward unit normal vector to the boundary. The action potential is triggered by a uniform current $\Iapp(\omega,\tau; \tboxstim)$ that is applied at time $\tau = 0$ $\mathrm{s}$ on $\{0\} \times [0, 125]$, which is the left side of the square domain, followed by another stimulus at a variable time $\tau = \tboxstim$ on the $[0, 125]^2$. Finally, the maximal rapid and slow delayed rectifier current conductances are fixed at $\GKr = 0.153 \; \mathrm{nS}$ $\mathrm{pF}^{-1}$ and $\GKs = 0.392 \; \mathrm{nS}$ $\mathrm{pF}^{-1}$, respectively, and the final simulation time is $T = 2$ $\mathrm{s}$.

We perform space discretization of the model \eqref{eq:cardiac_model} using $\mathbb{P}_1$ finite elements. The tetrahedral meshes are comprised of 20353 cells and 10400 degrees of freedom, and the average mesh size is $h = 3.5$ mm. Regarding time discretization, we first update the variables of the ionic model and then the transmembrane potential by employing an implicit-explicit numerical scheme \cite{RSA22,PRS22,FPR23}. Specifically, in the monodomain equation, the diffusion term is treated implicitly and the ionic term is treated explicitly. Moreover, the ionic current is discretized by means of the ionic current interpolation scheme \cite{KSK13}.
We employ a time step size $\Delta t = 0.1$ ms with the forward Euler scheme.

In \cref{tab:parameters_cardiac} we report descriptions, ranges and units for the four model parameters.
We initially generate a dataset of 10000 electrophysiology simulations (5000 for training and 5000 for testing) by exploring the space of parameters with stratified sampling (latin hypercube sampling~\cite{MBC00}).
The quantity of interest $\Q(X)$ that is monitored by our method is the space integral of the transmembrane potential $\Pot$ at the final time $T$, i.e.,
\begin{equation}
\Q(X) = \int_{\Omega} u(\omega, T; X) \dd \omega,
\end{equation}
where
\begin{equation}
X = \begin{bmatrix} \GCaL & \GNa & \Di & \tboxstim \end{bmatrix}^\top \in \R^4
\end{equation}
is the vector of inputs. 
We remark that all the input parameters and the quantity of interest have been normalized in the interval $[-1,1]$. 
The parameter space covered by the model \eqref{eq:cardiac_model} enables numerical simulations exhibiting both sinus rhythm and sustained arrhythmia, as shown in \cref{fig:cardiac_evolution}. 
Indeed, this test case presents a \emph{bifurcation}, meaning that the response varies abruptly even for locations that are relatively close in the space of parameters (alternative ways to describe this behavior would be to say that the model response is \emph{biphasic} or characterized by sharp gradients), which leads to a challenging mathematical problem for the presented dimensionality reduction approach. 
The bifurcation is visible in \cref{fig:cardiac_classification}, where we plot the histogram of the quantity of interest, and notice that $\Q$ takes either values close to $-1$, with a very few exceptions in the interval $(-0.9,0)$ that are not visible in the histogram, or values approximately in the interval $(0.4,1)$. We therefore first perform a classification to divide the input space in two classes, based on whether the normalized quantity of interest is positive or negative, and then compute the NeurAM in each class separately. We build a neural network classifier using the training data, and then plot the results on the test set in \cref{fig:cardiac_classification}. We observe that the classifier is able to correctly discriminate between sinus rhythm and sustained arrhythmia in $97.68\%$ of the cases.

We now apply our methodology to learn the NeurAM for the two classes. In \cref{fig:cardiac_manifold}, where we plot the exact and surrogate models as functions of the reduced variable, we notice that the NeurAM is able to capture most of the variability of the model. We remark that, for the case of the sinus rhythm, the model takes only values in the neighborhood of $-1$, and therefore the surrogate model is almost constant. Note that the plot on the left of \cref{fig:cardiac_manifold} might be misleading, as almost all the points are concentrated around $(-1,1)$, and only a few points are left out from the identity line and give a nonnegligible error. Regarding the case of sustained arrhythmia, where there is more variability, the surrogate model is able to provide a reliable approximation, but it is not as accurate as in the test cases in \cref{sec:2Dmodels} due to the increased complexity of the problem and the limited computational budget given by the computational cost of the numerical simulations. In \cref{tab:cardiac_error} we report the MAE and the MSE for the error $e_2$ in equation \eqref{eq:errors_def} given by the surrogate model. We remark that in this example we cannot compute the error $e_1$ given by the dimensionality reduction because we would need to run additional simulations in the projected points, which are not in the original dataset. We notice that the errors in \cref{tab:cardiac_error}, compared with the standard deviation $\sigma$ of the model over the input domain, are particularly small for the single classes, meaning that NeurAM gives accurate surrogate models for each class. We also notice that the overall error, obtained first applying the classifier and then the surrogate model of the predicted class, is slightly larger. This is due to the fact that a wrong classification yields a large error for the particular sample, implying a larger overall error contribution. Nevertheless, the derived surrogate model still provides a reliable approximation of the real model since the majority of the data is correctly classified.

We finally perform a sensitivity analysis for the cardiac electrophysiology model, considering both classes of sinus rhythm and sustained arrhythmia. The local sensitivity indices $\theta_i(t)$ and the global sensitivity indices $\Theta_i$ are illustrated in \cref{fig:cardiac_SA} for all the parameters $i$. We first notice that all the variables appear to be unimportant in the case of sinus rhythm, except for $\tboxstim$, which is the only one that affects the quantity of interest for the parameter space explored in \cref{tab:parameters_cardiac}, since it is crucial to distinguish this physiological state from the onset of possible rhythm disorders \cite{SFA21}. Regarding the case of sustained arrhythmia, we observe how the relevance of each input parameter varies along the NeurAM. In particular, we notice that, as reflected by the global indices in the table of \cref{fig:cardiac_SA}, $\GCaL$ is the most important input parameter, since it plays an important role in arrhythmogenesis \cite{BGR15}, followed by $\tboxstim$, $\Di$ and $\GNa$, respectively. Moreover, from the dynamic indices, we deduce that $\tboxstim$ is more relevant in the first and last part of the NeurAM, while $\Di$ is more important in the middle part.

\begin{table}[ht!]
\begin{center}
\begin{tabular}{cccc}
\toprule
Parameter & Notation & Range & Unit \\
\midrule
Stimulation time & $\tboxstim$ & [300, 400] & $\mathrm{ms}$ \\
Isotropic conductivity & $\Di$ & [0.2114, 0.4757] & $\mathrm{mm}^{2}$ $\mathrm{ms}^{-1}$ \\
Maximal $Ca^{2+}$ conductance & $\GNa$ & [9.892, 22.2570] & $\mathrm{nS}$ $\mathrm{pF}^{-1}$ \\
Maximal $Na^{+}$ conductance & $\GCaL$ & [2.653e-05, 5.970e-05] & $\mathrm{cm}^3$ $\mathrm{ms}^{-1}$ $\mu \mathrm{F}^{-1}$ \\
\bottomrule
\end{tabular}
\caption{Input parameter space for the cardiac electrophysiology model.}
\label{tab:parameters_cardiac}
\end{center}
\end{table}

\begin{figure}[ht!]
\centering
\begin{tabular}{cccccc}
$T = 0.2 \; \mathrm{s}$ & $T = 0.4 \; \mathrm{s}$ & $T = 0.6 \; \mathrm{s}$ & $T = 0.8 \; \mathrm{s}$ & $T = 1 \; \mathrm{s}$ & $T = 2 \; \mathrm{s}$ \\
\includegraphics[scale=0.055]{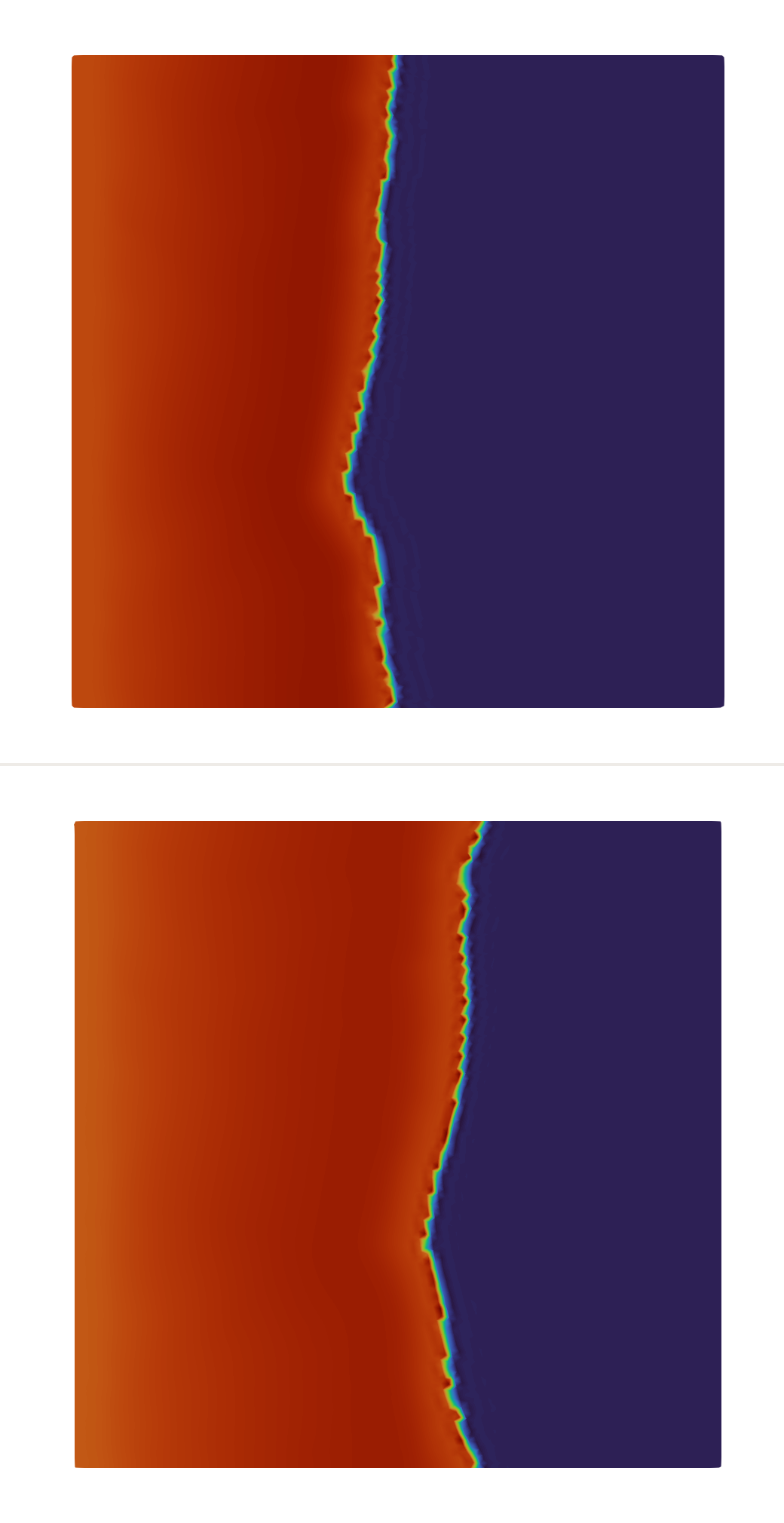} &
\includegraphics[scale=0.055]{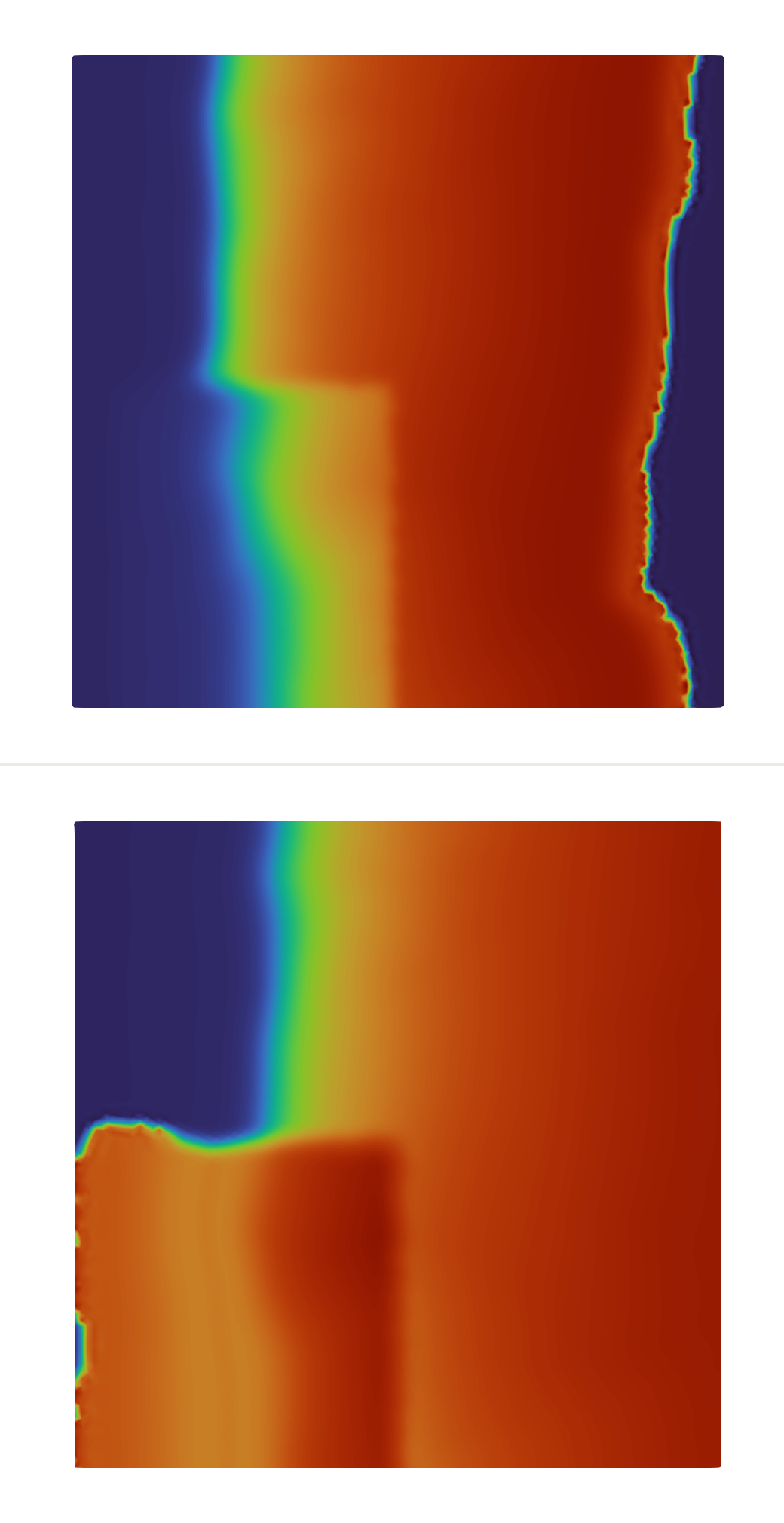} &
\includegraphics[scale=0.055]{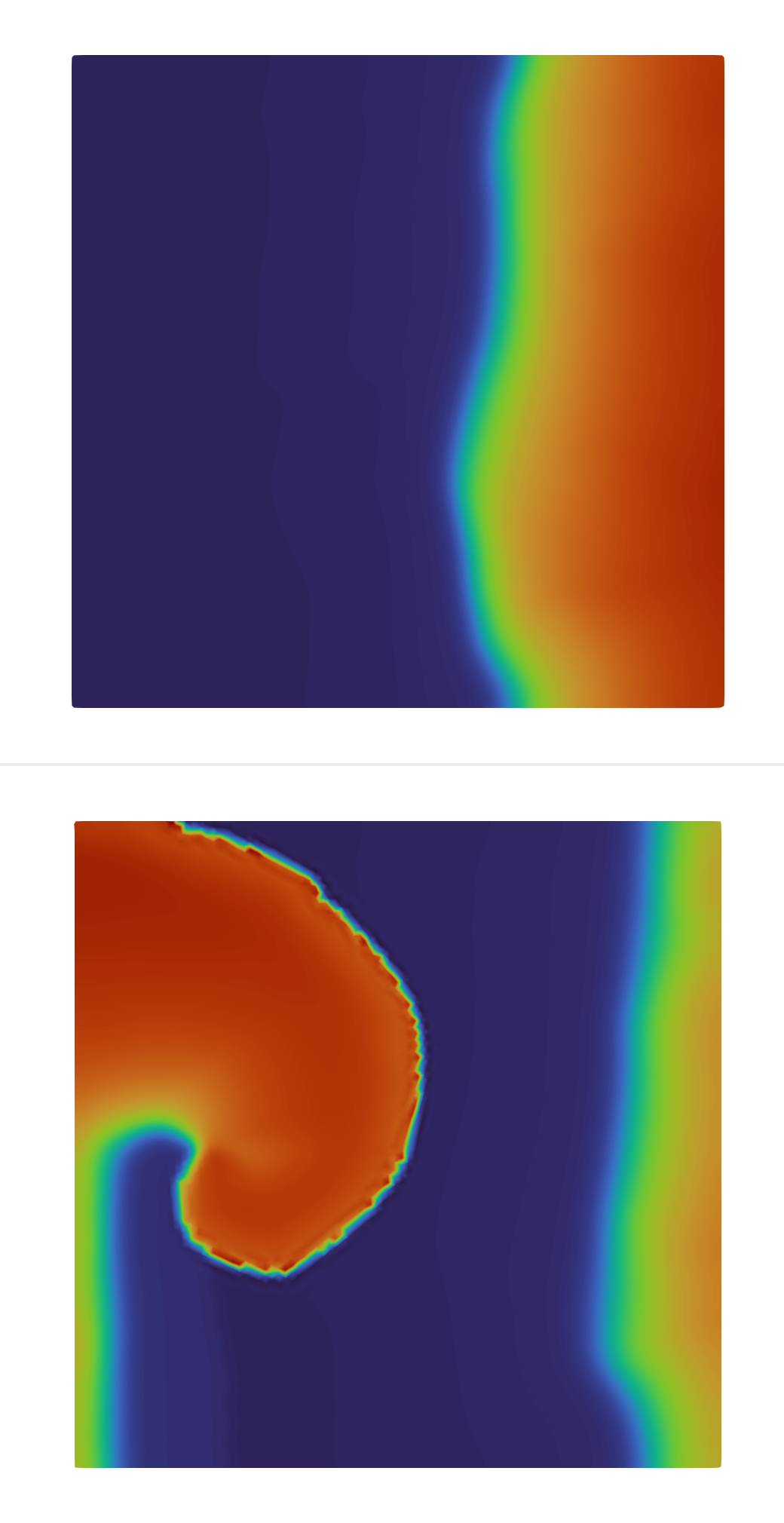} &
\includegraphics[scale=0.055]{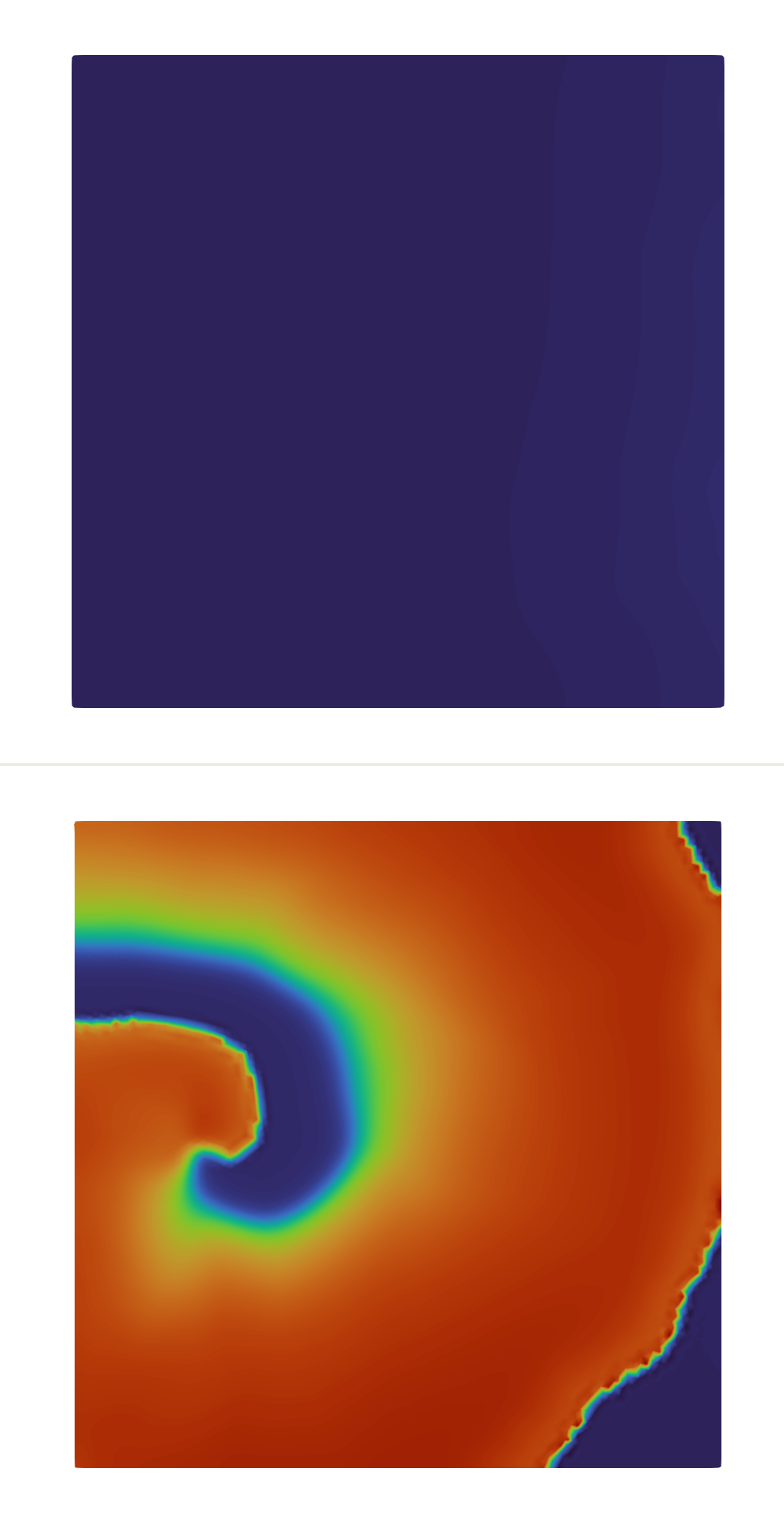} &
\includegraphics[scale=0.055]{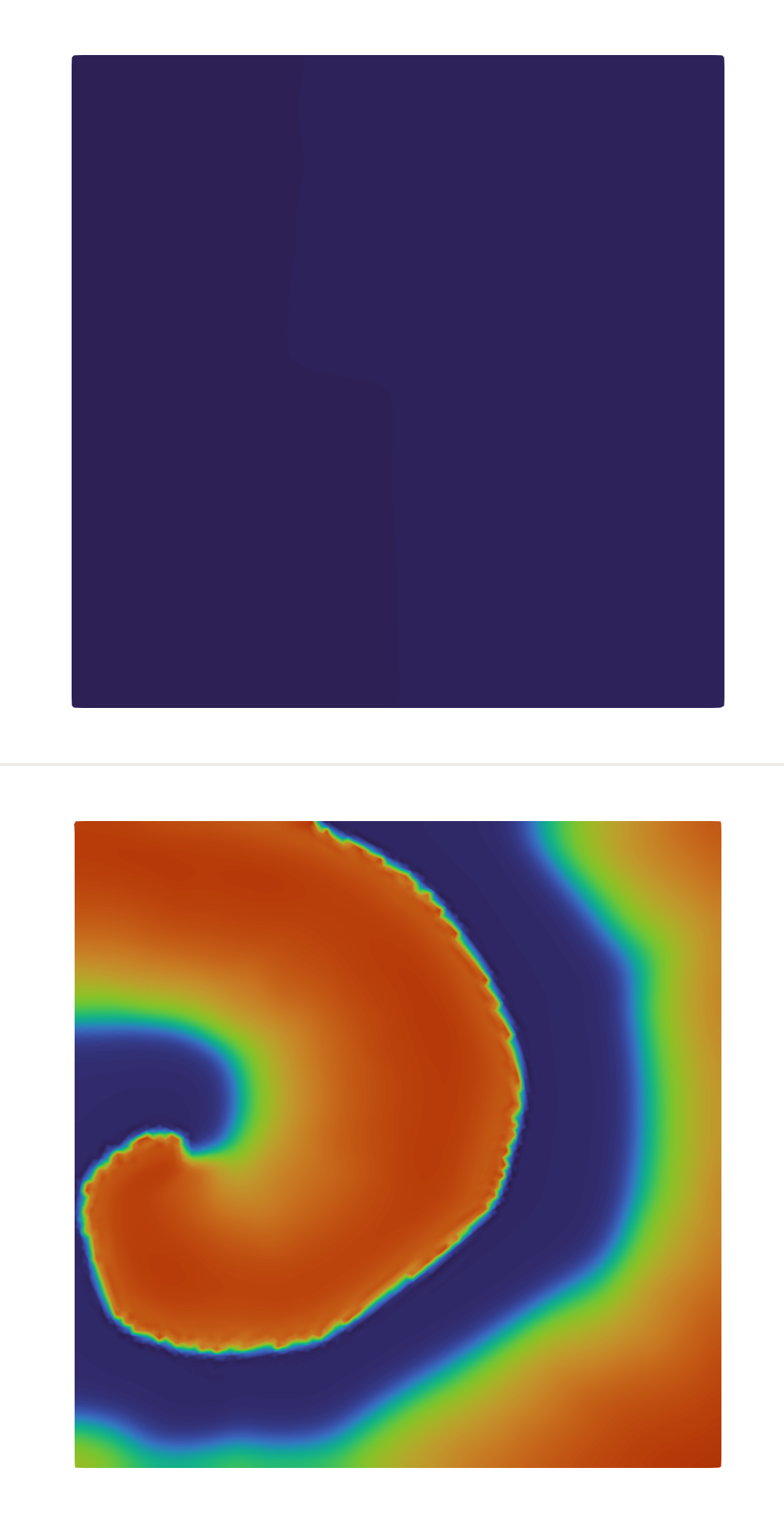} &
\includegraphics[scale=0.055]{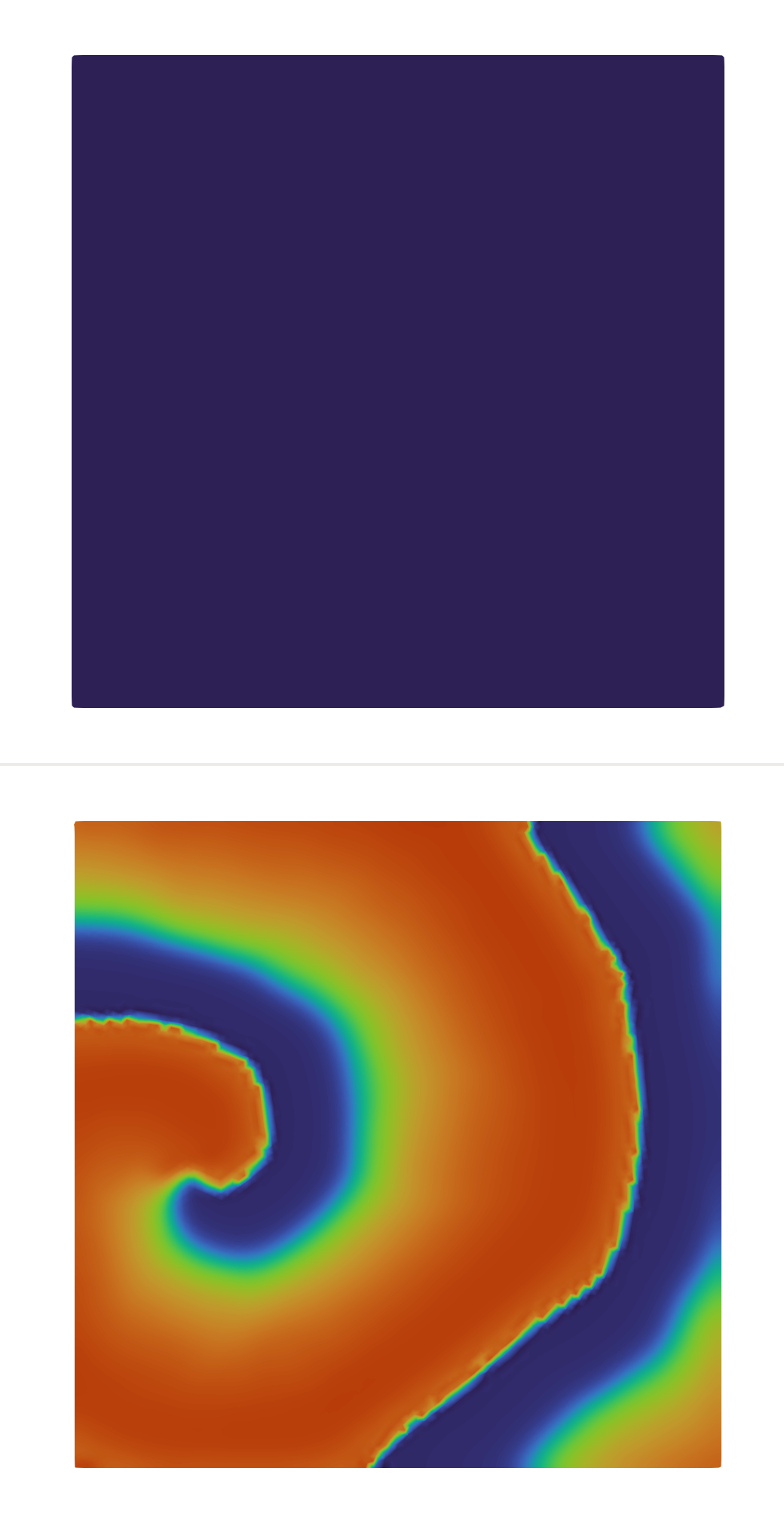}
\end{tabular}
\caption{Examples of evolution in time of the transmembrane potential $u$ in the cardiac electrophysiology model, for different time instants in the numerical simulation. Top: sinus rhythm. Bottom: sustained arrhythmia.}
\label{fig:cardiac_evolution}
\end{figure}

\begin{figure}[ht!]
\centering
\includegraphics{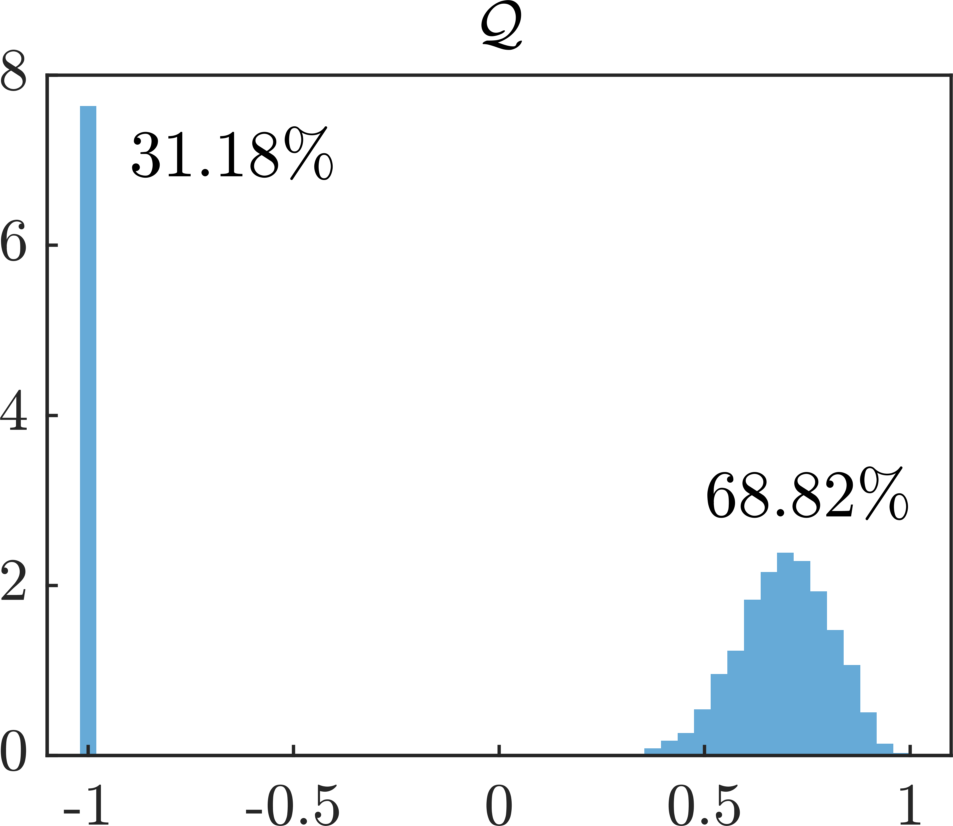} \hspace{1.5cm}
\includegraphics{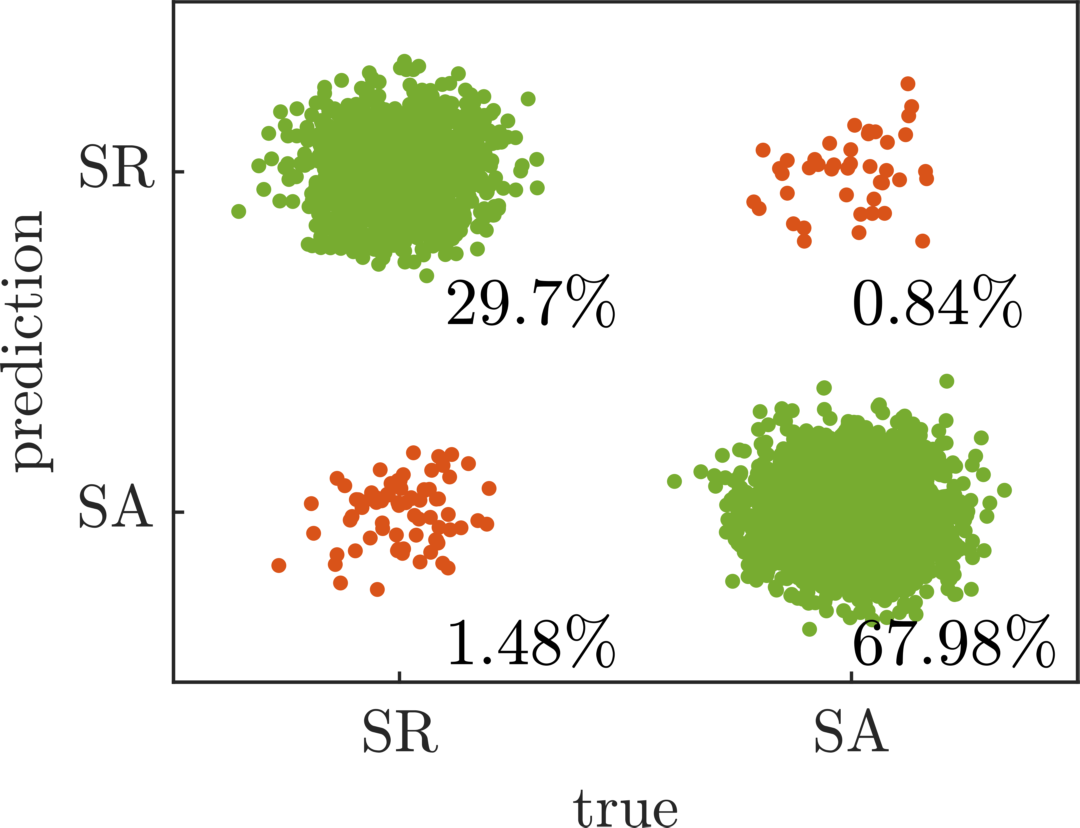}
\caption{Classification of cardiac electrophysiology response. Left: histogram of the quantity of interest $\Q$. Right: results from the classification algorithm.}
\label{fig:cardiac_classification}
\end{figure}

\begin{figure}[ht!]
\centering
\includegraphics{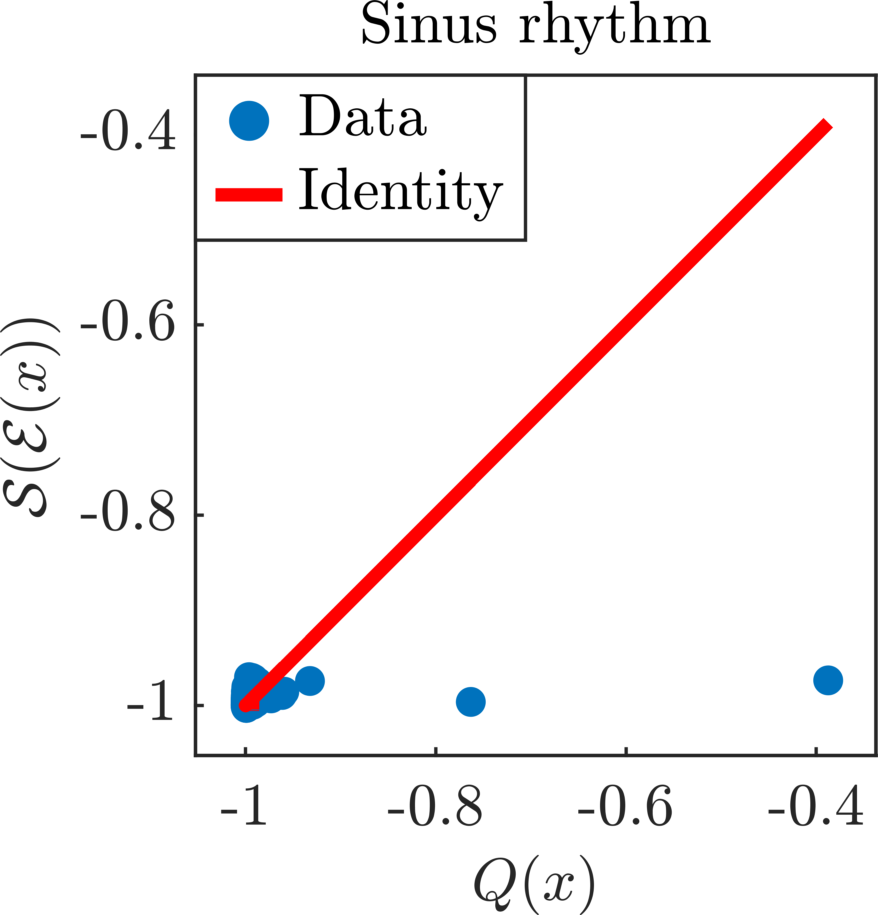} \hspace{2cm}
\includegraphics{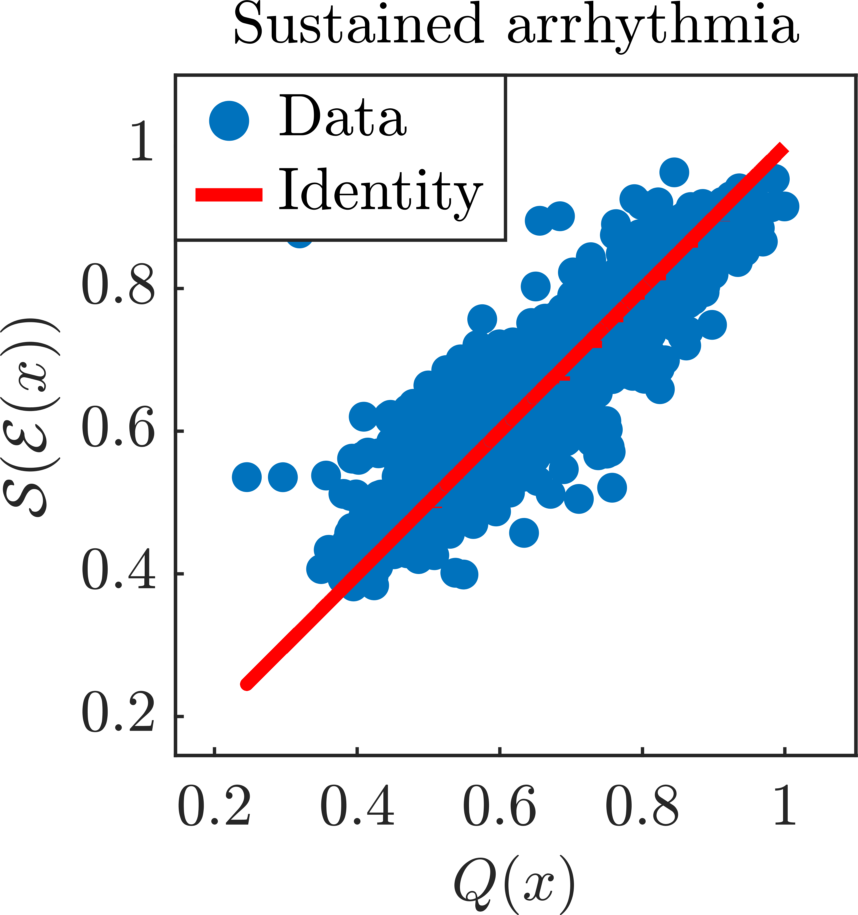}
\caption{Parity plot between the exact and surrogate models, for one realization of the NeurAM for the cardiac electrophysiology model. Left: sinus rhythm. Right: sustained arrhythmia.}
\label{fig:cardiac_manifold}
\end{figure}

\begin{table}[ht!]
\centering
\begin{tabular}{ccc}
\toprule
Sinus rhythm & Sustained arrhythmia & Overall \\
\midrule
\begin{tabular}{ccc}
MAE & MSE & $\sigma$ \\
\midrule
0.0002 & 0.0039 & 0.0240
\end{tabular} 
&
\begin{tabular}{ccc}
MAE & MSE & $\sigma$ \\
\midrule
0.0018 & 0.0293 & 0.1127
\end{tabular} 
&
\begin{tabular}{ccc}
MAE & MSE & $\sigma$ \\
\midrule
0.0689 & 0.0602 & 0.7905
\end{tabular} \\
\bottomrule
\end{tabular}
\caption{Approximation error (MAE and MSE) of the surrogate model, compared with the standard deviation $\sigma$ of the original model, for both sinus rhythm and sustained arrhythmia in the cardiac electrophysiology model. In the third column the overall error is computed employing the classifier as a first step.}
\label{tab:cardiac_error}
\end{table}

\begin{figure}[ht!]
\centering
\includegraphics{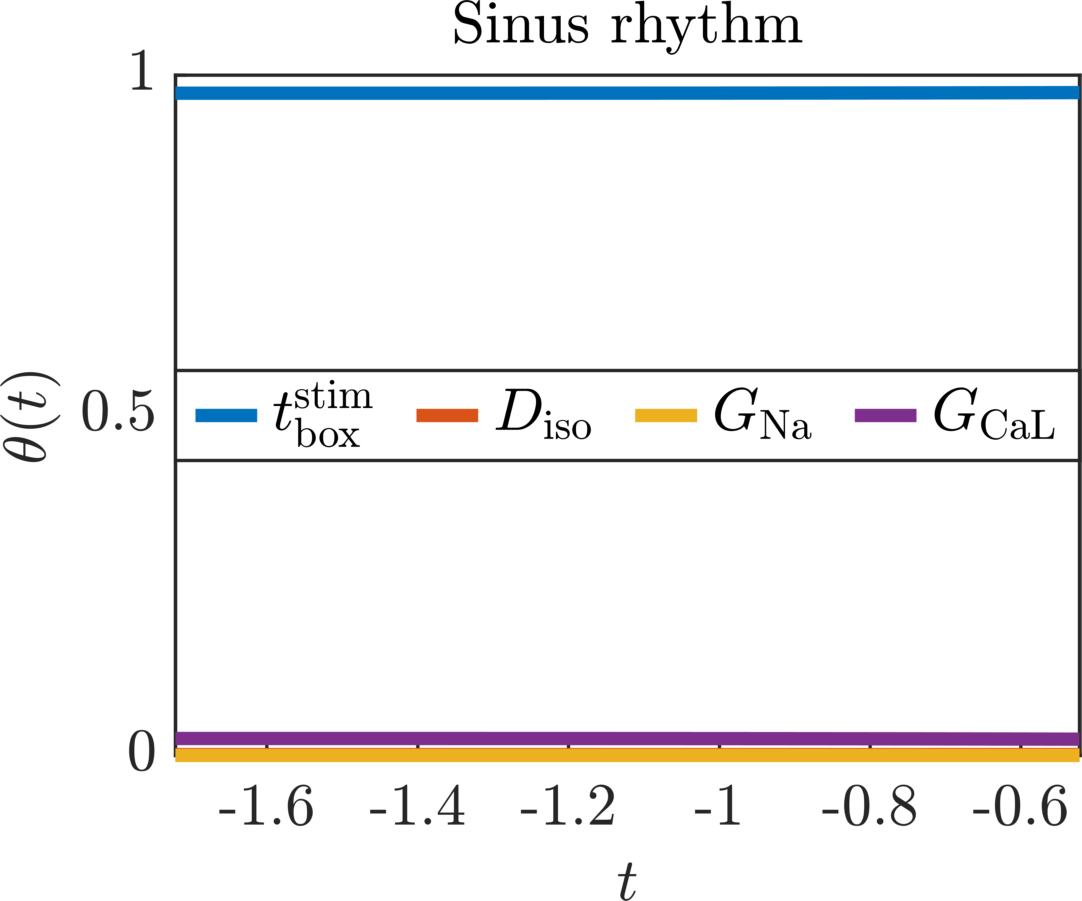} \hspace{1.5cm}
\includegraphics{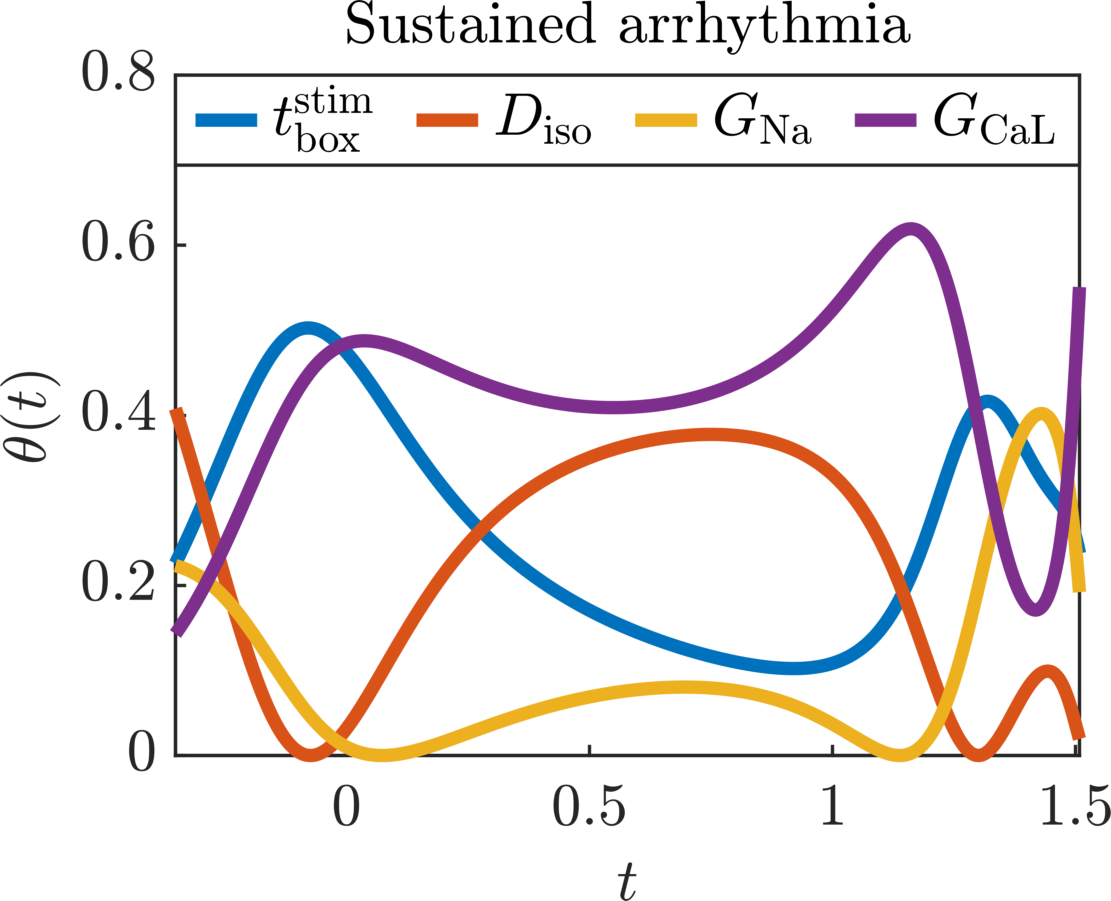}
\begin{center}
\small
\begin{tabular}{cc}
\toprule
Sinus rhythm & Sustained arrhythmia \\
\midrule
\begin{tabular}{ccccc}
& $\tboxstim$ & $\Di$ & $\GNa$ & $\GCaL$ \\
\midrule
$\Theta$ & $0.974$ & $0.001$ & $1\mathrm{e-}4$ & $0.024$
\end{tabular}
&
\begin{tabular}{ccccc}
& $\tboxstim$ & $\Di$ & $\GNa$ & $\GCaL$ \\
\midrule
$\Theta$ & $0.272$ & $0.192$ & $0.118$ & $0.418$
\end{tabular} \\
\bottomrule
\end{tabular}
\end{center}
\caption{NeurAM sensitivity analysis for the cardiac electrophysiology model. Left: sinus rhythm. Right: sustained arrhythmia. Top: local indices $\theta(t)$. Bottom: global indices $\Theta$.}
\label{fig:cardiac_SA}
\end{figure}

\section{Conclusion} \label{sec:conclusion}

In this work we introduced a novel technique for dimensionality reduction of computationally expensive models. 
In particular, we employed autoencoders with one-dimensional latent space to obtain a manifold in the parameter space (neural active manifold or NeurAM) capturing as much as possible of the model output variability. 
The autoencoder is combined with a surrogate model with inputs in the latent space that is trained at the same time as the encoder and decoder. 
NeurAM provides a concise one-dimensional analogue of the original computationally expensive model, and can be used to efficiently perform uncertainty quantification tasks. 
We first showed how NeurAM can be used to perform both local and global sensitivity analysis. In particular, we proposed new local indices providing a dynamic ranking of the input parameters along the manifold. Moreover, by integrating these local indices, we derived global indices that quantify the overall ranking of the input parameters. 
We then focused on the problem of multifidelity uncertainty propagation, selecting a low- and high-fidelity model pair.
Following \cite{ZGS24}, we used NeurAM to determine a shared space bridging the inputs of the two models, and to generate new low-fidelity input samples resulting in higher correlation between the low- and high-fidelity model response.
We remark that the main advantages with respect to \cite{ZGS24} are that we do not require a surrogate model to be computed beforehand, and that we replace normalizing flows with simple projections on the inverse cumulative distribution function, which benefits from increased robustness.
Furthermore, we provided a complete theoretical analysis of this approach under idealized conditions, where we prove that the use of NeurAM results in higher correlation.
In the multifidelity setting, an interesting direction would be to incorporate low-fidelity data into the training of the high-fidelity NeurAM, potentially by leveraging it within both the surrogate modeling component and the probability mapping.
Implementing such a strategy would require a redefinition of the NeurAM loss function to accommodate multifidelity data, as the current formulation is inherently single-fidelity. 
Moreover, developing such a multifidelity formulation would require a careful investigation of whether it is more computationally tractable to first train the low-fidelity component separately, possibly at the cost of reduced acceleration in high-fidelity training, or to integrate both fidelities within a unified architecture. While the latter is more challenging to train, it offers the greatest potential for high-fidelity data reduction.

We presented extensive numerical experimentation that showed the advantages of this nonlinear technique for dimensionality reduction, and corroborated the analysis. 
We first considered simple two-dimensional models, and then moved to more complex test cases, including the Hartmann problem, and a cardiac electrophysiology model that presents a bifurcation in the parameter space. 
NeurAM produces smaller errors compared to linear approaches such as AS~\cite{CDW14}, or more recent nonlinear approaches like AM~\cite{BGF19}.
On the other hand, our method has a slightly larger variability which is mainly caused by the fact that the solution to the minimization problem that determines the NeurAM is not unique. 
Therefore, even if the error is in most of the cases smaller than for traditional approaches in the literature, we are interested in looking for additional conditions to enforce, in order to guarantee the uniqueness of the NeurAM. 
Moreover, the theoretical argument in \cref{sec:formal_analysis} is valid under idealized conditions. An interesting future development consists in lifting this assumption, considering a more concrete scenario.
Finally, we treated the bifurcation in the cardiac electrophysiology model by first performing a classification, and then computing NeurAM separately for the two classes. In future work, we would like the algorithm to be able to autonomously perform domain decomposition, without the need of a pre-trained classifier. 
In this context, we might benefit from a two-dimensional manifold, and therefore we should first study an extension to NeurAM with multidimensional latent space, by drawing inspiration from the active subspaces method, where the eigenvalues of the gradient covariance matrix guide the selection of the effective dimension. 
Even if a one-dimensional latent space is already quite powerful, and it performed well for most of the problems we have tried, we believe that developing a procedure for automatic latent dimension selection represents an interesting direction for future research on the extension to the multidimensional setting. Moreover, regarding the application to multifidelity estimators, moving beyond a one-dimensional latent space introduces significant challenges for transforming probability distributions. In fact, the inverse transform sampling technique employed in our current framework is suitable for one-dimensional distributions. We note that this problem has already been successfully tackled in the context of linear dimension reduction strategies; see, e.g.,~\cite{ZGE23,ZGG23,ZGG25}. Since the limited availability of high-fidelity data is the main bottleneck in realistic applications, we believe that the development of multifidelity probability maps is a promising direction that could ultimately enable efficient transformations between probability distributions when the latent space is multidimensional.

\subsection*{Acknowledgements}

This work is supported by NSF CAREER award \#1942662 (DES), NSF CDS\&E award \#2104831 (DES), NSF award \#2105345 (ALM), and NIH grants \#R01EB029362 and \#R01HL167516 (ALM). 
This work used computational resources from the Stanford Research Computing Center (SRCC). 
Sandia National Laboratories is a multi-mission laboratory managed and operated by National Technology \& Engineering Solutions of Sandia, LLC (NTESS), a wholly owned subsidiary of Honeywell International Inc., for the U.S. Department of Energy’s National Nuclear Security Administration (DOE/NNSA) under contract DE-NA0003525. This written work is authored by an employee of NTESS. The employee, not NTESS, owns the right, title and interest in and to the written work and is responsible for its contents. Any subjective views or opinions that might be expressed in the written work do not necessarily represent the views of the U.S. Government. The publisher acknowledges that the U.S. Government retains a non-exclusive, paid-up, irrevocable, world-wide license to publish or reproduce the published form of this written work or allow others to do so, for U.S. Government purposes. The DOE will provide public access to results of federally sponsored research in accordance with the DOE Public Access Plan. G. Geraci was partially supported by a Laboratory Directed Research \& Development (LDRD) project.

\bibliographystyle{siamnodash}
\bibliography{biblio}

\end{document}